\documentclass[a4paper,10pt]{article}

\usepackage{preamble}

\begin{document}
\maketitle
\vspace{\stretch{1}}

\section*{Abstract}
\paragraph{} The study of Frobenius endomorphism provides numerous information about its corresponding Abelian variety. To understand the action of the Frobenius endomorphism, one may be interested in its eigenvalues. According to Weil's third conjecture ("Riemann hypothesis over finite fields"), they all have absolute value less than or equal to $2g\sqrt{p}$. Thus, the eigenvalues of the Frobenius endomorphism all belong to the same compact subset of the complex plane, and are roots of the same monic polynomial with integer coefficients (the characteristic polynomial of the Frobenius endomorphism). Such complex numbers are called \emph{algebraic integers "totally" in a compact subset}, which means algebraic integers all conjugates of which belong to a same given compact subset of the complex plane.

\medskip

\par The study of such algebraic integers helps to understand the eigenvalues of the Frobenius endomorphism, especially their distribution. In this paper, we will study the following question : under which conditions a compact subset of the complex plane has a finite or infinite number of algebraic integers \emph{totally} in it ? The problem can be studied in light of the notion of capacity of a compact subset, which comes from potential theory. In this paper, we will present the theory of capacity and some theorems (Fekete, Szegö, Robinson) derived from it that partially answer the question: in the case of a union of real segments, when the capacity is smaller (resp. larger) than 1, it contains a finite (resp. infinite) number of algebraic integers totally in it. For instance, for real line segments, the limit length is 4.

\medskip

This paper is written as part of a collective project conducted in École Polytechnique (France). It is aimed towards undergraduate audience in mathematics, with basic knowledge in algebra, topology, analysis, and dwells into a modern topic of research.

\vspace{\stretch{1}}

\section*{Acknowledgements}

\paragraph{} We are very grateful to our supervisor Javier \textsc{Fres\'an}. This project could not exist without Javier, without his involvement, his availability, his good humour and his always relevant advice. Most importantly, thanks to him, we experienced the pleasure of studying mathematics during this project.

\medskip

We would also like to thank our coordinator Stéphane \textsc{Bijakowski} for his kindness and constructive comments on our work.

\vspace{\stretch{1}}

\newpage
\vspace{\stretch{1}}

\tableofcontents
\newpage

\newpage

\section*{Preamble}

\subsection{Motivation}
\paragraph{}
Since the XIX-th century, research in arithmetic uses concepts from other branches of mathematics, leading to great progress with surprising efficiency. One of the branches created by this diversification is arithmetic geometry, which combines algebra and geometry to solve number theory problems. A quick description of the success of this union can be found in \cite{articlegéoarithmétique}.
\newline

One of the main subject of study in this area is the behavior of geometric objects (\textit{e.g.} "curves") defined over finite fields. These "curves" are rigorously defined as the domain on which a certain number of multi-variable polynomials defined over the considered finite field vanish. For instance, the first bisector is the domain in which the polynomial $P(X,Y) = X - Y$ vanishes. As a reminder, a finite field is a finite set with a well defined addition and multiplication. For example, we can take the field  $\mathbb{F}_p=\mathbb{Z}/p\mathbb{Z}$, where $p$ is a prime number, as well as the extensions $\mathbb{F}_{p^n}$ obtained by adding the root of an irreducible polynomial of degree $n$. 
\newline

One of the remarkable properties of these curves is that we can associate them with abelian groups in which it is possible -- in a certain way -- to "sum" two points on the curve, and to find an opposite to each point. These groups, called \emph{abelian varieties} exhibit very interesting properties and are still today an active area of mathematical research, with numerous results 
\cite{articleRechercheRécent}\cite{articleVariétés}. 
\newline

To study the curves over finite fields, a classical method is to study the action of homomorphisms (\emph{i.e.} applications compatible with the algebraic structure of the field) on those curves. One of the most famous and important homomorphisms in such cases is the so-called Frobenius endomorphism $\mathrm{Fr}$ which raises the coordinates of a point of the curve to the $p$\textit{-th} power, where $p$ is the characteristic of the field (the smallest integer such that $p\times1 = 0$ in the field, with $1$ the unit of the field).
In general algebra, the study of Frobenius endomorphism allows to deduce a number of properties on the set on which it acts. As such, it is natural to try  to study it in the case of curves over finite fields. Especially, since $x^p=x$ for all $x \in \mathbb{F}_p$, the number of fixed points of the Frobenius endomorphism is the number of points of the curve with coordinates in $\mathbb{F}_p$. In the same way, the fixed points of $\mathrm{Fr}^n$ are the points of the curve whose coordinates belong to $\mathbb{F}_{p^n}$. 
\newline

It is then possible to construct the generating function $\exp(\sum_{n \geq 1} N_n \frac{T^n}{n})$, where $N_n$ is the number of points of the curve with coordinates in $\mathbb{F}_{p^n}$. According to the Weil conjectures \cite{articleFrob}, this power series is actually a rational function : more precisely, a quotient $P(T)/(1-T)(1-pT)$ with $P \in \mathbb{Z}[X]$ a monic polynomial with integer coefficients. This polynomial is the characteristic polynomial of $\mathrm{Fr}$ acting -- not on the curve as previously seen -- but on the abelian variety defined from the curve. The eigenvalues of this Frobenius endomorphism are the root of the characteristic polynomial, which, according to Weil's third conjecture, ("Riemann hypothesis over finite fields"), are all inferior in absolute value to $2g \sqrt{p}$, in which $g$ is the dimension of the abelian variety. In particular, they are all included in the same compact subset (\emph{i.e.} closed and bounded) $K$ of the complex plane. The roots of such a polynomial are called \emph{algebraic integers totally included in $K$}. Of course, all monic polynomials with such roots are not associated with Frobenius endomorphisms. But understanding the behavior of this kind of polynomials allows a better comprehension of Frobenius endomorphisms.
\newline

With that being said, it is natural to inquire about the distribution of those eigenvalues, which can be deduced from the distribution of algebraic integers totally in a compact subset. This distribution itself can be deduced from the case where $K$ is a real line segment, studied by Robinson who showed in 1962 the following result: if the length of $K$ is strictly more than 4, there is an infinite number of algebraic integers totally in it; if the length is strictly inferior to 4, there is only a finite number of such numbers. The case in which the length is exactly 4 is -- for now -- only solved for a few special cases. The paper we present here gives a proof of this result.
\newpage

\subsection{Outline of the paper}

\paragraph{}
In a formal way, \emph{algebraic integers} are complex numbers which are the roots of a monic polynomial with coefficients in $\Z$. We shall call $\Q$-conjugates of an algebraic integer $\alpha$ the roots of the unique monic polynomial $P$ with rational coefficients of minimal degree for which  $P(\alpha) = 0$ (minimal polynomial). Given a compact set $K$ of $\C$, we say that an algebraic integer is totally in $K$ if all his $\Q$-conjugates are in $K$. We call \emph{degree of an algebraic integer} the degree of his minimal polynomial.

\medskip

The study of the distribution of algebraic integers totally in a real line began in the beginning of the XX-th century, with a first result from Schur in 1918 which solves the case of real lines with length strictly below 4. Many other results follow in the next years. Then, some 40 years later, Robinson eventually demonstrates his theorem, leaving only the case of length 4 unresolved.
\newline

The proof we shall give here for Robinson's theorem follows an article from
Jean-Pierre Serre \cite{serre2018bourbaki} which summarizes a lecture given during the Bourbaki seminar of March 2018 in Paris. The purpose of the present paper is to be understandable by a non-specialist audience, but one that has a solid grasp on the basic elements of algebra (group, polynomials), topology (weak convergence) and analysis (continuity, basic complex analysis). It seems to us that an undergraduate audience can follow our work and the proof we shall give without too much trouble. The more interested readers can find in the appendix additions on measure theory, on more subtle analysis points used in our reasoning, as well as proofs deemed too technical and without much interest as far as understanding the essence of this proof goes.
\newline

In the first part (Elementary remarks), we shall make a range of first observations on this problem, which will lead to ideas of proof, or invalidate certain proof schemes that could be thought of. This part includes an algorithmic approach that allows us to get an intuition on the result (although the programs' complexity does not allow for a very deep dive into polynomials of high degrees). Finally, we will demonstrate a first result first found by Kronecker (1857) which deals with the case where the compact subset of the complex plane is the unit circle. From this case, we will deduce the behavior of the real line $[-2,2]$, and, more widely, of any real line of the form $[n-2,n+2]$, with $n \in \mathbb{Z}$. From these remarks, we will give first bounds for the length of real lines containing an infinite/finite number of algebraic integers totally in it.
\newline

In the second part, we will look into the theory of capacity, which is the natural frame of study for dealing with the " size" of a compact set for this problem. This notion of capacity is deeply linked to the concepts introduced in measure theory, and some of the results from this part -- once combined with results from the last part -- allow us to prove a slightly stronger result than Robinson's theorem. In particular, we prove the convergence (in a certain sense) of a measure associated with a sub-sequence of polynomials towards the measure of equilibrium of the compact set.
\newline

We will then show in the third part Fekete's theorem \ref{thmFekete} which solves the case of a real line of length strictly less than four, then, Fekete-Szegö's theorem \ref{thmFeketeSzego} which is the equivalent of Robinson's theorem for complex compact subset of the complex plane (as opposed to the real lines of Robinson's theorem). We will see how the proof of this theorem is not sufficient to yield the real case, although some ideas can be reused to solve the latter, which is why the study of this theorem is very interesting for this proof.
\newline

In the last part of this paper, we will eventually give the full proof of Robinson's theorem, using the tools developed in the first parts as well as a few results pertaining to the field of algebraic geometry. After a quick introduction to the core concept of algebraic curves, we will use those results to finish the proof of the main result of this paper.
\newpage

\section{Elementary remarks}\label{section1}

\paragraph{}
Let us start by reminding ourselves of the definition of our object of study: algebraic integers totally in a compact set.

\begin{defi}
\begin{enumerate}[label=(\roman*)]
    \item An \emph{algebraic integer} is a complex number that is a root of some monic polynomial with coefficients in $\mathbb{Z}$.
    \item Given an algebraic integer $\zeta\in\C$, we call \emph{minimal polynomial of $\zeta$} the unique monic polynomial with integer coefficients among these with $\zeta$ as a root.
    \item Let us denote by $\gal(\zeta)$ the set of all roots of the minimal polynomial of $\zeta$.
    \item Let $E$ be a subset of $\C$ and $\zeta$ an algebraic integer. We say that $\zeta\in\C$ is \emph{totally in $E$} if $\gal(\zeta)\subset E$.
   
\end{enumerate}
\end{defi}

\medskip

\subsection{First results}

\paragraph{}
In this first section, we will start by stating three "elementary" remarks about the problem: first, at a fixed degree, there is only a finite number of algebraic integers totally in a given compact set; then, the derivatives of a minimal polynomial of algebraic integers totally in a compact set have interesting properties; finally, if $K_1 \subset K_2$ are two compact sets, then the numbers of algebraic integers totally in $K_1$ and $K_2$ respectively, can be compared.

\medskip

\begin{rmq}
Let $K$ be a compact subset of $\mathbb{C}$. Let $P$ be a monic polynomial of degree $n$ the roots of which are all in $K$, we can upper bound the $k$-th coefficient of $P$ by $ \sup(K)^k \binom{n}{k}$. It is enough to assure that, at fixed degree, there exists a finite number of such polynomials with all roots in $K$, and therefore a finite number of algebraic integers of degree $n$ totally in $K$.
\\

This remark leads us to several interesting approaches: 
\begin{itemize}\label{rmqBornes}
    \item We can find (algorithmically) the adequate polynomials of small degrees.
    \item If we want to prove that there exists a finite number of algebraic integers in a compact set, it is enough to prove that the degree of such an integer is bounded.
\end{itemize}
\end{rmq}

\begin{rmq}\label{rmqDerivee}
Let us consider a monic polynomial $P$ with integer coefficients whose roots are algebraic integers totally in a compact set $K$, where the degree of $P$ is greater than or equal to $2$, its derivative is not a monic polynomial. However it is still a polynomial with integer coefficients, and its roots are still in $K$, if this compact set is convex (which is the case for real line segments). Indeed, it follows from Gauss-Lucas theorem that the roots of $P'$ are in the convex hull of the set of roots of $P$.

Therefore we notice that the knowledge of $P'$ gives us information on the roots of $P$. This idea will be useful when we compute algorithmically algebraic integers totally in a compact set.

\end{rmq}

\begin{rmq}\label{rmqInclusion}
Let $K_1$ and $K_2$ be $2$ compact sets such that $K_1 \subset K_2$, then there is at least as many algebraic integers totally in $K_2$ as those totally in $K_1$ (since any algebraic integer totally in $K_1$ is an algebraic integer totally in $K_2$). This somewhat obvious remark will be of great use for proofs where it is easier to work on a smaller compact set.
\end{rmq}

\subsection{Algorithmic approach}

\paragraph{}
In order to get an intuition about the algebraic integers totally in a compact set, we have decided to conceive and implement an algorithm listing algebraic integers totally in a line segment, the minimal polynomial of which, is of fixed degree. These results will later be compared to the demonstrated theorems. 

\medskip

Let $K$ be a compact subset of $\C$. Two questions were kept to be answered by numerical experiments:
\begin{itemize}
    \item What are the algebraic integers totally in $K$? A particular focus is done on the distribution of such numbers. The results can be compared to the cases studied in section \ref{subsection13}. This question seems difficult because it requires a fine knowledge of the algebraic integers totally in $K$. That is why the core question of our paper is actually :
    \item Is there an infinite number of algebraic integers totally in $K$ ? The numerical answer to this question can be compared to the results of sections \ref{sectionthFEk} and \ref{section3}.
\end{itemize}
First, we will state the algorithmic principle that we use, then we will comment on the results we obtained.

\subsubsection{The algorithmic principle}
\paragraph{} In order to get numerical results, only monic polynomials with integer coefficients, the degree of which, is less than a given integer $n$ will be studied. The goal is to find a list as short as possible with all minimal polynomials, with degree less than $n$, of algebraic integers totally in $K$.
\medskip

\paragraph{} The remark \ref{rmqBornes} allows to give bounds to the coefficients of polynomial with a given degree, the roots of which, are all in $K$. Therefore, for a given degree, the polynomials that could be the minimal polynomial of an algebraic integer totally in $K$ can be enumerated:  

\vspace{\stretch{1}}

\begin{algorithm}
\caption{Naive algorithm}
\begin{algorithmic} 
\STATE{$l$ empty list}
\FORALL{$P$ monic polynomial of degree less than n, the coefficients of which, are integers in the bounds of remarks \ref{rmqBornes}}
\IF{the roots of $P$ are totally in $K$}
\STATE{add the roots of $P$ to $l$}
\ENDIF
\ENDFOR
\RETURN {$l$}
\end{algorithmic}
\end{algorithm}
\vspace{\stretch{2}}

\clearpage

However, this method is not very efficient. Indeed, the complexity is greater than the product of the double of the bounds:
$$ \displaystyle\prod_{k=1}^n  (2\sup(K))^k \binom{n}{k} C(n)\geqslant (2\sup(K))^{n^2/2},$$ where $C(n)$ is the complexity of computing the roots of a polynomial of degree $n$ (done by computing the eigenvalues of the companion matrix). For instance, when the compact set $K$ is a line segment subset of [-2.5,2.5], this algorithm cannot list the algebraic integers totally in $K$ with minimal polynomial of degree higher than $n=4$. This data is unfortunately not enough to make interesting conjectures.

\paragraph{} Since we mainly focus on the case of real segments, we implemented another specific algorithm  for segments of $\R$. Remark \ref{rmqDerivee} can therefore apply and gives a great improvement in efficiency. Indeed, when $K$ is a segment of $\R$, for all algebraic integers totally in $K$, all the derivatives of its minimal polynomial have simple roots and all the zeros are in $K$. More generally, let $l<n-1$ and let us consider the polynomials $P=\displaystyle\sum^n_{k=0}a_k X^k$ where the $n-l-1$ greatest coefficients $a_{l+1},\dots,a_{n-1}$ are fixed and $a_n=1$. Then $P^{(l)} = l!a_l + \displaystyle\sum_{k=1}^{n-l}\frac{(l+k)!}{k!}a_{l+k}X^k$ has $n-l$ roots in $K$. More specifically, $-a_l$ is between the maximum of all local minimums and the minimum of all local maximums of $\displaystyle\sum_{k=1}^{n-l}\frac{(l+k)!}{k!}a_{l+k}X^k$. Thus, by repeating for $l$ ranging from $n-2$ to $0$, we obtain stricter constraints on $a_0,\dots,a_{n-1}$ and thus a list of polynomials likely to have all their roots in $K$ much more restricted than in the naive algorithm. In summary, the limits of the remark \ref{rmqBornes} have been refined by taking into account the value of the other coefficients already set. Thus the polynomials whose roots are calculated are much less numerous and of lower degree, which explains why this algorithm is more efficient than the previous one. Thanks to this optimization, it was possible to obtain all algebraic integers, whose minimum polynomial's degree is at most $n=7$ for segments included in $[-2.5,2.5]$. The analysis of figure \ref{complexite} corroborates this observation. While it has been calculated that the naive algorithm has a complexity of at least $C^{n^2}$, the second algorithm has an experimental complexity close to $15^{n}$.

\paragraph{}

After writing this algorithm, we tried to compare it with the state of the art. In \cite{articleAlgo}, the authors combine techniques very close to ours but more refined with algebraic properties to solve a problem close to ours. Their algorithm allows  to go up to degree 13 in a reasonable time. But unfortunately, we cannot directly apply their algorithm to our problem.

\subsubsection{Results}

\paragraph{}
Considering remark \ref{rmqInclusion}, an algebraic integer totally in a compact set $K$ is totally in all compact sets containing $K$. We can think that "large" compact sets will have an infinite number of algebraic integers while "small" compact sets will have a finite number.

\paragraph{} To verify this, we have plotted  the evolution of the number of algebraic integers totally in $K$ according to the degree of their minimal polynomial for different segments $K$ of $\R$. The results in figure \ref{superFig} show that the length 4 is the boundary between "small" and "large" segments. For segments with lengths less than 4, there are fewer and fewer new algebraic integers as the degree increases (blue, orange). On the contrary, for segments longer than 4, there are more and more (green, red, purple). This observation remains valid when translating the segments.

\paragraph{} These observations are in accordance with the theorems that we will demonstrate in this paper. The behavior of "small" segments is treated by Fekete's theorem \ref{thmFekete} and that of "large" segments by Robinson's theorem \ref{thmRobinson}.

\paragraph{} When $K=[-2.2]$, we can notice that the number of new algebraic integers totally in $K$ found at the $n$ degree seems to increase, but not exponentially (figure \ref{figMoins22}). We can verify that the algebraic integers found are of the form $2\cos(\frac{2k\pi}{m})$ with $k$ and $m$ integers. We will see in the following subsection (\ref{prop-22}) that the algebraic integers totally in $[-2.2]$ are the $2\cos(\frac{2k\pi}{m})$ with $k$ and $m$ integers. However, the algorithm does not find them by increasing $m$.

\vspace{\stretch{1}}

\paragraph{} Beyond the finiteness of the set of algebraic integers totally in a compact set, one can be interested in the distribution of these numbers. The figure \ref{figRepart} shows that the distribution is essentially the same both when translating slightly and when the length changes. Moreover, this distribution is similar to that of the roots of the Chebyshev polynomials.

\begin{figure}[ht]
   
      \begin{minipage}[c]{.4\linewidth}
       \centering
    \includegraphics[width = 1.20\linewidth]{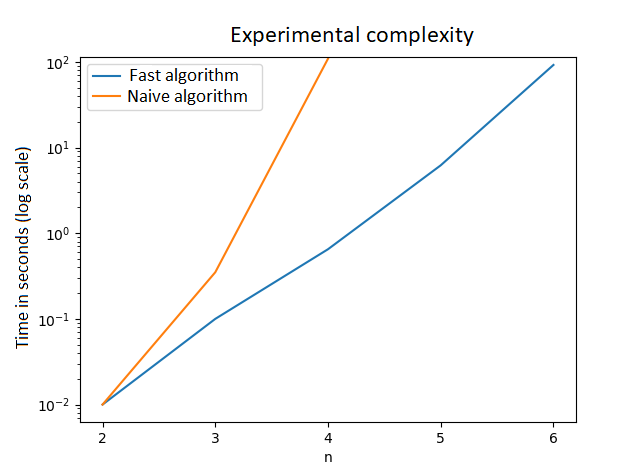}
    \caption{Experimental complexity of the two algorithms (length = 3.9)}
    \label{complexite}
   \end{minipage}
   \hfill
   \hspace{0.5cm}
   \begin{minipage}[c]{.55\linewidth}
    \centering
    \includegraphics[width=0.9\textwidth]{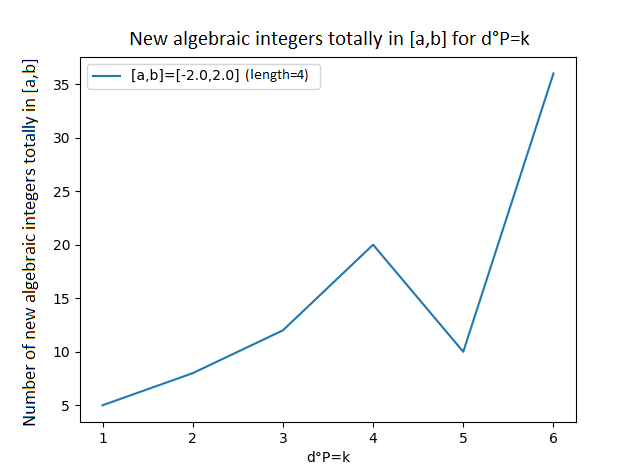}
    \caption{Number of new algebraic integers in function of the degree of the minimum polynomials for $K=[-2,2]$}
    \label{figMoins22}
   \end{minipage}
\end{figure}

\vspace{\stretch{2}}

\begin{figure}[b]
    \centering
    \includegraphics[width = 0.8\linewidth]{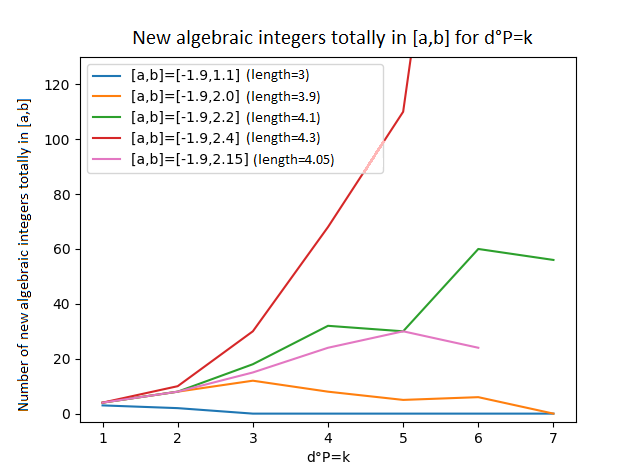}
    \caption{Number of new algebraic integers found in function of the degree of the polynomials traversed}
    \label{superFig}
\end{figure}

\begin{figure}[t]
    \centering
    \includegraphics[width = \linewidth]{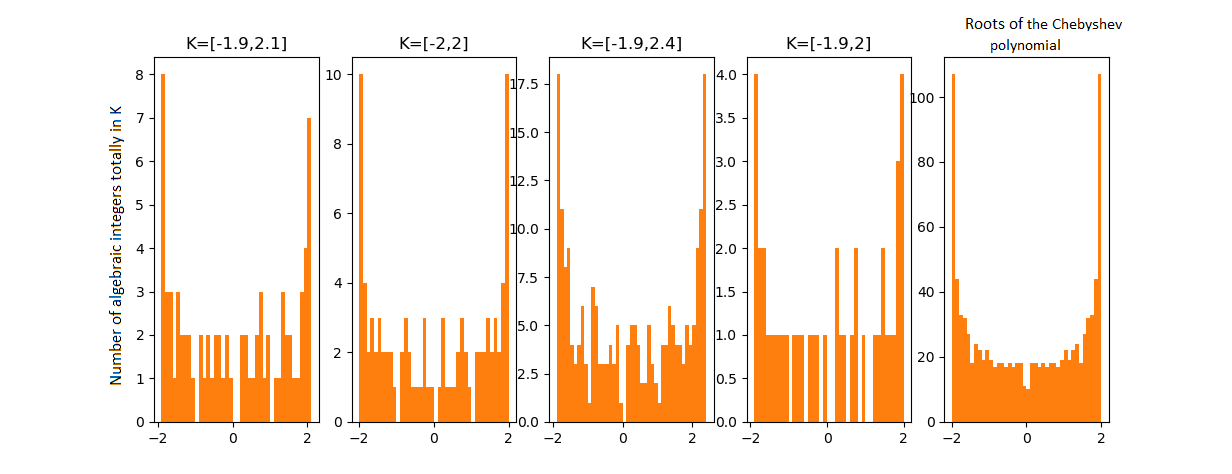}
    \caption{Distribution of algebraic integers of lower degree for different segments compared to that of the roots of Chebyshev polynomials}
    \label{figRepart}
\end{figure}

\clearpage
\subsection{Two elementary cases : $\U$ and $[-2,2]$ }\label{subsection13} 
\paragraph{}
Let us start with the well known cases of the unit circle and segments with length $4$ and integer end points.

\subsubsection{$K=\U$ : Kronecker's theorem}
\paragraph{}
When $K=\U$ is the unit circle, Kronecker's theorem gives us the following result : 

\begin{thm}[Kronecker]
Let $P\in\Z [X]$ be a monic polynomial whose complex roots are all in the unit circle $\U$, then its roots are roots of unity.
\end{thm}

Let us start with the following lemma :

\begin{lmm}\label{lmm-entier}
Let $ n, k \in \mathbb{N}$ and $Q$ be a polynomial with expression $\displaystyle\prod_{i=1}^n (X-\alpha_{i})$.\\
If $Q \in \mathbb{Z}[X]$, then $Q_k(X):=\displaystyle\prod_{i=1}^n (X-\alpha_{i}^k) \in \mathbb{Z}[X]$.
\end{lmm}

\begin{proof}

Let us consider such a polynomial $Q$ and integers $n$ and $k$. The Frobenius companion matrix $M$ of $Q$ is a matrix of $M_n(\mathbb{Z})$ whose characteristic polynomial is $Q$. Therefore the $\alpha_i$ are the eigenvalues of $M$. $M$ is similar to an upper triangular matrix with main diagonal $\alpha_1,\dots,\alpha_n$, which results in $M^k$ being similar to an upper triangular matrix with main diagonal  $\alpha_1^k,\dots,\alpha_n^k$. Thus, the $\alpha_i^k$ are the eigenvalues of $M^k$ with the same multiplicity as for $M$. It implies that $Q_k$ is the characteristic polynomial of $M^k$. Since this matrix has integer coefficients, $Q_k$ has as well.
\end{proof}

\medskip

Let us prove Kronecker's theorem:
\medskip

\begin{proof}
Let us consider $P = X^n + p_{n-1}X^{n-1} + \dots + p_0 \in \mathbb{Z}[X]$ a monic polynomial such that $Z(P)\subset \mathbb{U}$, where $Z(P) = \{\alpha_1,\dots, \alpha_n\}$ is the set of the complex roots of $P$. It follows from the Vieta's formulas that:
\[p_{n-k}=(-1)^{k} \displaystyle\sum_{J \in \mathcal{P}_k([1,\dots,n])}\displaystyle\prod_{i=1}^{k}\alpha_{j_{i}}.\]

Since the $\alpha_i$ are of magnitude $1$, the triangle inequality gives us that $\vert p_{n-k}\vert \leq \binom{n}{k}$. The coefficients of $P$ being integers, we deduce that $\mathcal{E}_n$, the set of polynomials of degree $n$  which verify the hypotheses of the theorem, is a finite set. It implies that $\mathcal{R}_n$, the set of roots of polynomials in $\mathcal{E}_n$ is also finite.

Let $\alpha_i$ be a root of $P$ and let us consider the multiplicative group generated by $\alpha_i$, $G=\lbrace\alpha_i^k, k \in \mathbb{Z}\rbrace$. With the notations from lemma ~\ref{lmm-entier}, for all $k\in\Z$, $\alpha_i^k$ is a root of $Q_k \in \mathcal{E}_n$ according to the lemma. Therefore, $G$ is a finite group since $G\subset \mathcal{R}_n$. As a result, there exists integers $k<k'$ such that $\alpha^k=\alpha^{k'}$, which means $\alpha^{k'-k}=1$.
\end{proof}

\begin{rmq}\label{rmqCombineSymmEntier}
Let us give another proof of lemma~\ref{lmm-entier} using symmetric polynomials. 
Let us consider $P = \displaystyle\sum_{j=1}^n p_j X^j =\displaystyle \prod _{i=1}^n (X-\alpha_i) \in \Z[X]$ and $Q = \displaystyle\sum_{j=1}^n q_j X^j = \displaystyle \prod _{i=1}^n (X-\alpha_i^k)$.
Let $S_j$ denote the $j$-th elementary symmetric polynomial in $n$ variables. It follows from the Vieta's formulas that : $p_{n-j} = (-1)^j S_j(\alpha_1, \dots, \alpha_n)$ and $q_{n-j} = (-1)^j S_j(\alpha_1^k, \dots, \alpha_n^k)$.  According to the fundamental theorem of symmetric polynomials,

there exists a polynomial $R\in\Z[X_1,\dots,X_n]$ such that $$S_j(X_1^k,\dots,X_n^k) = R(S_1,\dots,S_n)$$ 
Therefore, the coefficients of $Q$ 
$$q_{n-j} = (-1)^j S_k(\alpha_1^k, \dots, \alpha_n^k) =(-1)^j R(-p_{n-1},\cdots,(-1)^n p_0)$$ are integers.

Actually, the fundamental theorem of symmetric polynomials applies to all the symmetric polynomials, so that all the symmetric combinations of $\alpha_1,\dots,\alpha_n$, with integer coefficients, are integers.
\end{rmq}

\begin{rmq}
Note that the use of companion matrices in the proof of lemma \ref{lmm-entier} allows us to give a direct proof of this following particular case of the fundamental theorem of symmetric polynomials: "there exists  
$ R\in\Z[X_1,\dots,X_n],$ such that $  S_j(X_1^k,\dots,X_n^k) = R(S_1,\dots,S_n)$".
\paragraph{}
Indeed, using the previous notations, let $M$ be the companion matix of $P$.  $$\forall \,1\leqslant i,j\leqslant n,\quad M_{i,j}\in\{S_j(\alpha_1,\dots,\alpha_n))\,|\,0\leqslant j\leqslant n\}\cup\{0,1\}.$$
We deduce from it that :
\[\forall \, (i,j),\quad M_{i,j}\in \Z [S_0(\alpha_1,\dots,\alpha_n),\dots,S_n(\alpha_1,\dots,\alpha_n)].\]

Thus the coefficients of the characteristic polynomial of $M^k$ are in $\Z [S_0(\alpha_1,\dots,\alpha_n),\dots,S_n(\alpha_1,\dots,\alpha_n)]$. However, these coefficients turn out to be, regardless of the sign, the $(S_i(\alpha_1^k,\dots,\alpha_n^k))_{0\leqslant i\leqslant n}$.\\
We conclude that :
\[\forall\, 0\leqslant i\leqslant n,\quad S_i(\alpha_1^k,\dots,\alpha_n^k) \in\Z [S_0(\alpha_1,\dots,\alpha_n),\dots,S_n(\alpha_1,\dots,\alpha_n)].\]
\end{rmq}

\paragraph{Conclusion of the case $K=\U$ :} algebraic integers totally in $\U$ are the roots of unity, and there is an infinite number of them. We have an example of a compact set with an infinite number of algebraic integers totally in it. 

\subsubsection{$K = [-2,2]$}
\begin{prop}\label{prop-22}
Let us consider $P \in \mathbb{Z}[X]$ a monic polynomial such that $Z(P)\subset [-2;2] $.\\
The roots of $P$ all have the following expression: $z+\overline{z}$, where $z$ is a root of unity.
\end{prop}
\begin{figure}[ht]
    \centering
    \includegraphics[width = 9cm]{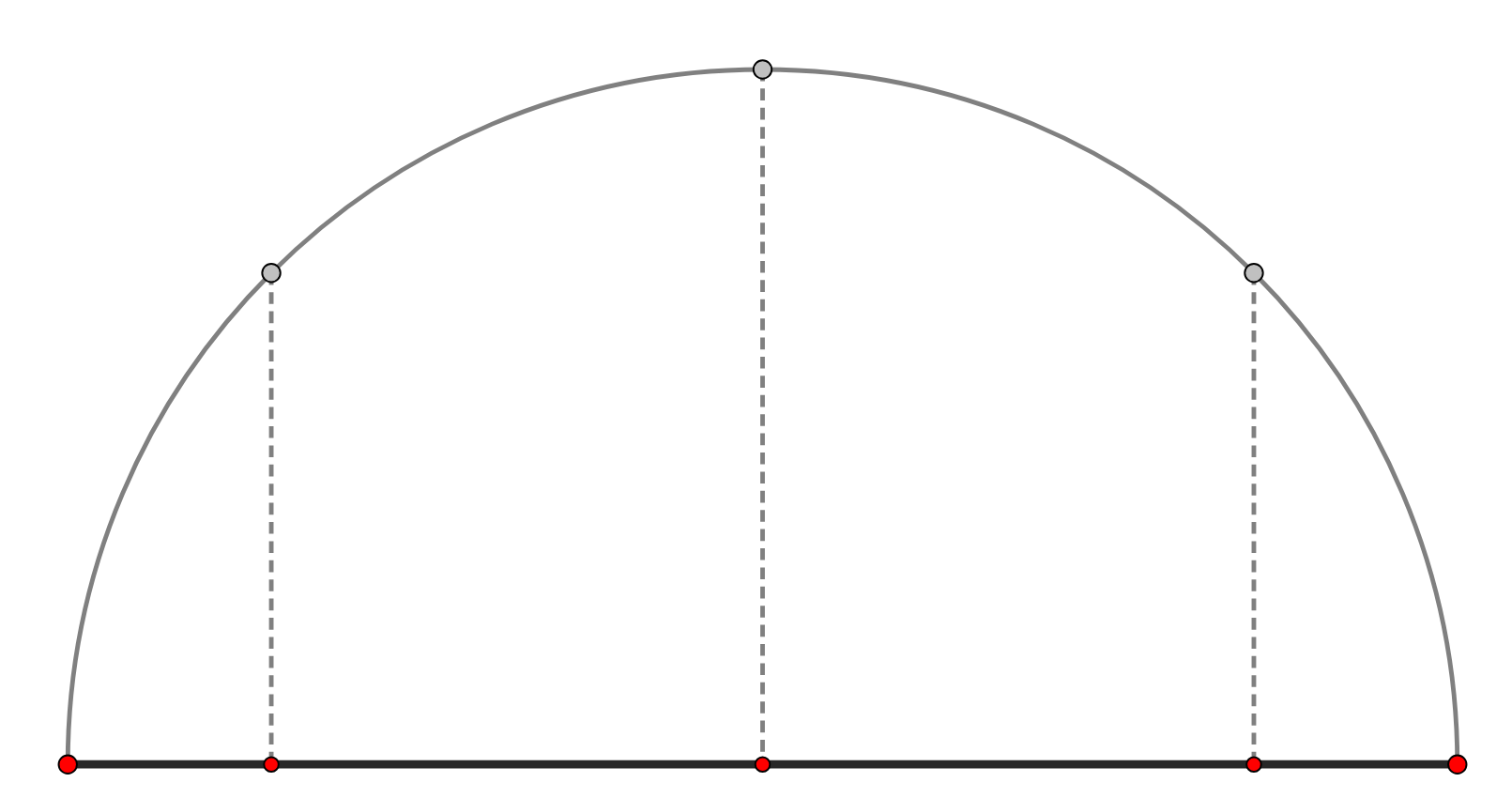}
    \caption{examples of points totally in $K=[-2,2]$}
    \label{fig:PtsTotalementDansPlusMoinsDeux}
\end{figure}

\begin{proof}
Let $P$ be a polynomial of degree $n$ verifying the assumptions of the proposition. Let us consider $Q= X^n P(X+\frac{1}{X})$. Using the binomial formula, we show that $Q$ is a monic polynomial of degree $2n$ with integer coefficients. Let us note that $0$ is not a root of $Q$ since the polynomial's constant coefficient equals $1$. Therefore, for all $z\in\C$,
\[Q(z)=0  \Leftrightarrow P(z+\frac{1}{z})=0.\]
    Let $z$ be a complex root of $Q$. Then, $z+\frac{1}{z}$ is a root of $P$, and $z+ \frac 1z  \in [-2,2]$. In particular, $z+ \frac 1z$ is a real number. If we write the obvious equality between this real number and its complex conjugate, we get that $\rho \sin(\theta)=\frac{1}{\rho}\sin(\theta)$, where $\rho$ and $\theta$ are the modulus and an argument of $z$ respectively. Two cases are possible:\\
\begin{itemize}
\item[$\bullet$] if $\sin(\theta)=0$, then $z \in \mathbb{R}$. In order to verify $z+\frac{1}{z} \in [-2,2]$, studying the sign of this expression shows us that $z=\pm 1$.
\item[$\bullet$] if $\rho=\frac{1}{\rho}$,  $\rho=1$, since the modulus is a non-negative number.
\end{itemize}
\vspace{1em}
In all cases, $z\in \U$. It follows from Kronecker's theorem that all the roots of $Q$ are roots of unity.

Let $r$ be a root of $P$. Since $r\in [-2,2]$, we can write $r = z +\overline{z}$, with $z = \frac r2 + i \sqrt{1-\frac{r^2}{4}}$. Since $z\in\U$, $r = z + \frac 1z $. Therefore, $z$ is a root of $Q$, which means that $z$ is a root of unity.
\end{proof}

\begin{cor}
The algebraic integers totally in $[-2,2]$ are real numbers with the following expression: $z +\overline{z}$, where $z$ is a root of unity.
\end{cor}

\begin{proof}
    The previous proposition shows us that the algebraic integers totally in $[-2,2]$ all have the following expression:  $z +\overline{z}$, where $z$ is a root of unity. Let us prove the reciprocal implication.
    
    \paragraph{}
    Let $(U_n)_n$ be the Chebyshev polynomials of the second kind defined as follows: $$U_0 = 1 \text{ et }\displaystyle U_{n+1}=2XU_{n}-U_{n-1},\ \ \forall n\geq 1.$$ 
    We use induction on $n$ to show that $P_n(X)=U_n(\frac X 2)$ are monic polynomials of degree $n$ and with integer coefficients. They verify that $\forall\theta\in\R\setminus\pi\Z\quad\frac{\sin \left((n+1)\theta\right)}{\sin\theta}=P_n(2\cos \theta)$. We deduce from it that: 
    $$ \forall n \geqslant 1, Z(P_n)=\{2\cos {\frac {k\pi }{n+1}},\quad \forall k\in \{1,\ldots ,n\}\}\subset [-2,2].$$ Therefore, for all $n\geqslant 1, Z(P_n)$ is a set of algebraic integers totally in $[-2,2]$, which implies that for any root of unity $z$, $z+\overline{z}$ is totally in $[-2,2]$.
\end{proof}

\paragraph{ Conclusion of the case $K = [-2,2]$ : } algebraic integers totally in $[-2,2]$ are the $z + \frac 1z$, with $z$ some root of unity. 
\vspace{2em}

It follows from translating the previous result, this claim for any segment of $\R$ with length $4$ and integer end points:

\begin{cor}

Let $K=[n-2, n+2]$ be a real segment with length $4$, integer endpoints and $n\in\Z$ its middle point. The algebraic integers totally in $K$ are $n+z+\frac 1z$, with $z$ a root of unity.

\end{cor}
\begin{proof}
Let $P\in\Z[X]$ be a monic polynomial such that $Z(P)\subset[n-2, n+2]$ and let us consider $Q=P(X+n)$. 
Since $n\in\Z$, $Q$ is a monic polynomial with integer coefficients, and $Z(Q)\subset [-2,2]$. The proposition~\ref{prop-22} allows us to conclude.
\end{proof}

From this particular case, we deduce a more general theorem which gives us a first piece of information:

\begin{thm}[Upper bound of the minimal length to contain an infinite number of algebraic integers totally in a segment]

Segments with length greater than or equal to 5 have an infinite number of algebraic integers totally in them.
    
\end{thm}

\begin{proof}
    Segments with length $5$ contain a segment with length $4$ and integer endpoints, and according to remark 1.3, they contain at least as many algebraic integers as the smaller segments, therefore containing an infinite number, according to corollary 1.2.
\end{proof}

A first stage of the work was done with general remarks, an algorithmic approach to convince ourselves of the result and a first particular case. Let us now build a general framework to tackle the general case of the theorem.

\clearpage
\section{Capacity theory} \label{section2}
\paragraph{}
First and foremost, let us focus our interest on the notion of \emph{capacity}. It is a non-negative number describing the size of a set, but not a geometric size related to its measure : the terminology comes from physics, more specifically, from the \emph{capacity of a capacitor}, which describes the ability of a set to contain electric charges. We shall see that it is an "adequate " definition of the size of
a compact set when it comes to algebraic integers. This notion gives us a criterion which allows us to determine whether or not a compact set has an infinite number of algebraic integers totally in it (cf. \ref{section3}), with Fekete's and Fekete-Szegö's theorems. 

\medskip

The goal of this section is to give three equivalent definitions of the notion of capacity : transfinite diameter (\ref{ssec-diamTrans}), logarithmic capacity (\ref{ssec-capacitelogarithmique}) and the Chebyshev's constant (\ref{ssec-constcheb}). The main theorem of this section is theorem \ref{chebEgalCapaEgalTau} which proves the equivalence between the three definitions, and states important properties of objects which allow us to link these three approaches: Fekete's measures, equilibrium measure, equilibrium potential, Fekete's polynomials, Chebyshev polynomials... In section \ref{ssec-potentielsemiharmonic}, we will also prove a formula to calculate a type of capacity thanks to semi harmonic functions: this section, which is quite technical, can be skipped at first reading.

\paragraph{}

We will first introduce the notion of transfinite diameter of a compact set thanks to Fekete points ; then we will study the notion of logarithmic capacity and equilibrium measure before unifying these two notions. Finally we will give a third definition of the capacity using Chebyshev polynomials. Thanks to this last approach we will be able to calculate the capacity of segments. 

\medskip

\subsection{Transfinite diameter, Fekete points}\label{ssec-diamTrans}
\paragraph{}
To get a first intuition, let us consider a geometric problem from electrostatics: let us place electrons in a bounded domain. These electrons will tend to maximize their mutual distances in order to minimize the overall energy. In $3$ dimensions, the force is proportional to the square of the inverse of the distance, and the potential is proportional to the inverse of the distance. Since we are working in the two-dimensional complex plane, the force is proportional to the inverse of the distance and the potential is proportional to $\log$ of the distance.

\medskip

Given $K$ a compact subset of $\C$ and $z_1,\dots, z_n \in K$, let us define the \emph{potential}  at point $z_i$ (which can be equal to $+\infty$) as follows :
\[U_{z_1,\dots,z_n}(z_i) = -\sum _{j\neq i} \log\abs{z_j-z_i}.\]

\paragraph{}
Let us also define the energy of the configuration $(z_1,\dots,z_n)$ as the mean of the potentials at each point : 
\[E(z_1,\dots,z_n) =- \frac{2}{n(n-1)}\sum_{i<j}\log\abs{z_i-z_j}
= - \log \prod_{i<j}\abs{z_i-z_j}^{\frac {2}{n(n-1)}}\]

Let us focus on minimizing the energy of the compact set $K$ for a given number $n$ of points : 
\[E(K) = \min_{z_1,\dots,z_n\in K}E(z_1,\dots,z_n) = - \log \max_{z_1,\dots,z_n\in K}\prod_{i<j}\abs{z_i-z_j}^{\frac {2}{n(n-1)}} \]

\paragraph{}
The lower bound is reached since $K$ is compact, hence the $\min$ in the optimization formula. The points where this minimum is reached are called the \emph{Fekete points} :

\begin{defi}[Fekete points]\label{defi-PointsdeFekete}

Let $K$ be a compact subset of $\C$. Let us define
\[\delta_n(K)=\max_{z_1,\dots,z_n\in K} \prod_{i < j}\abs{z_i-z_j}^{2/n(n-1)}\]

This maximum value is reached at the \emph{Fekete points} (of degree $n$).
\end{defi}
\begin{rmq}
It matches the definition of the usual diameter for $n=2$.
\end{rmq}

\begin{lmm}\label{decroissanceDesDnK}
let $f:K\times K \to \R\cup\{-\infty\}$ be a function and for all $n\in\N^*$, let us consider
$$
m_n^f(K)=\sup_{x_1,\dots,x_n\in K}\frac{1}{n(n-1)}\sum_{i\neq j}f(x_i,x_j).
$$
The sequence $(m_n^f(K))_n$ is decreasing.
\end{lmm}

\begin{proof}
Let $n\geq 1$, and $x_1,\dots,x_{n+1}\in K$ (not necessarily distinct). We obtain the following equality:
$$
\frac{1}{n(n+1)}\sum_{i\neq j}f(x_i,x_j)=
\frac{1}{n+1}\sum_{k=1}^{n+1}\left(\frac{1}{n(n-1)}\sum_{i\neq j, i\neq k, j\neq k}f(x_i,x_j)\right);
$$
indeed, it is easily verified by calculating the coefficients of $f(x_i,x_j)$ on each side: the term $f(x_i,x_j)$ on the right is only affected by the $k\in [\![1,n+1]\!]\backslash\{i,j\}$ so there are $n-1$ choices. Given a $k$, the sum $\sum_{i\neq j, i\neq k, j\neq k}$ on the right can be considered as $\sum_{i\neq j}$ where $i,j\in [\![1,n+1]\!]\backslash\{k\}$, which is lower or equal to $m^f_n(K)$ by definition. Then, as we consider the supremum on $x_1,\dots,x_n\in K$, we obtain that : $$m_{n+1}^f(K)\leqslant\displaystyle\sup_{x_1,\dots,x_n\in K} \frac 1 {n+1} \sum^{n+1}_{k=1}m_n^f(K)=m_n^f(K).$$ 
\end{proof}

\begin{defi}[Transfinite diameter]
The sequence $(\delta_n(K))_{n\geqslant 2}$ is non-negative and decreasing. The limit \[\tau(K) = \lim_{n\to\infty} \delta_n(K) \] is called the \emph{transfinite diameter}.
\end{defi}

\begin{proof}
We deduce the fact that $(\delta_n(K))_{n\geqslant 2}$ is decreasing from lemma \ref{decroissanceDesDnK}.    
\end{proof}
\paragraph{}
\clearpage
We have given our first definition of the notion of capacity : the transfinite diameter. Let us now prove a few elementary properties and give a few examples as well.
\newline
\begin{prop}\label{propTau}
Let $K$ be a compact subset of $\C$, then
\begin{enumerate}
    \item if $a,b\in \C$, then $\tau(a K +b) = \abs{a}\tau(K)$;
    \item $\delta_n(K)=\delta_n(\partial K)$ and then $\tau(K)=\tau(\partial K)$.
\end{enumerate}
\end{prop}

\begin{proof}
(a) can be easily deduced from the definition.
\medskip\par (b) Let us consider $n\geq2$ and $h(z_1,\dots,z_n)=\prod_{1\leq i<j\leq n}(z_i-z_j)$. It is a holomorphic function of $n$ complex variables on $\C^n$, therefore $\|h\|_{K^n}=\|h\|_{(\partial K)^n}$ (we can use the maximum modulus principle of a holomorphic function of one variable), where $\delta_n(K)=\delta_n(\partial K)$.
\end{proof}

\begin{prop}[Unit circle]\label{propTauCercleUnit}
The Fekete points on $\U$ are the $n$-th roots of unity, up to a rotation. The transfinite diameter of the unit circle equals $1$.
\end{prop}

\begin{proof}
Let us consider $n\geqslant 2$ and $x_1,\dots,x_n\in \U$, let us denote by $D(x_1,\dots,x_n)$ the determinant of the Vandermonde matrix $D(x_1,\dots,x_n)=\det ((x_i^{j-1})_{1\leq i,j\leq n})=\prod_{i<j} (x_j-x_i)$. \\
We have $\delta_n(\U) = \displaystyle\max_{x_1,\dots,x_n\in \U} |D(x_1,\dots,x_n)|^{\frac 2 {n(n-1)}}$. 
Hadamard's inequality gives us $$ |D(x_1,\dots,x_n)| \leq \left|\left|\!\begin{pmatrix} 1\\ \vdots \\ 1
\end{pmatrix}\!\right|\right|_2 \, \left|\left|\!\begin{pmatrix} x_1\\ \vdots \\ x_n
\end{pmatrix}\!\right|\right|_2 
\dots\,\left|\left|\! \begin{pmatrix} x_1^{n-1}\\ \vdots \\ x_n^{n-1}
\end{pmatrix}\!\right|\right|_2 = n^\frac n 2.$$ 
Therefore, $\delta_n(\U)\leq n^{\frac 1 {n-1}}$.

The equality is reached when the row vectors form a family of orthogonal vectors that are linearly independent or contains a null vector, which can be written as follows :
\begin{align*}
    D(x_1,\dots,x_n) = n^\frac n 2
    &\Leftrightarrow \forall i<j \in [\![1,n]\!], \begin{pmatrix} x_i^0\\ \vdots \\ x_i^{n-1}
\end{pmatrix} \cdot \begin{pmatrix} x_j^0\\ \vdots \\ x_j^{n-1}
\end{pmatrix} = 0 \\
    &\Leftrightarrow \forall i<j \in [\![1,n]\!],  \sum_{k=0}^{n-1} (x_j\overline{x_i})^k=0 \\
    &\Leftrightarrow \forall i<j \in [\![1,n]\!], x_j\neq x_i \text{ et }  \frac {1-(x_j\overline{x_i})^n}{1-x_j\overline{x_i}}=0 \\
                        & \Leftrightarrow \forall i<j \in [\![1,n]\!],  x_j\overline{x_i}\in \U_n-\{1\} \\
                        & \Leftrightarrow \{x_1,\dots,x_n\} = x_1\U_n
\end{align*}
Therefore $\delta_n(\U) = n^{\frac 1 {n-1}}$, the Fekete points on $\U$ are the $n$-th roots of unity, up to a rotation, and then $\tau (\U)=\displaystyle\lim_{n\rightarrow\infty} n^\frac 1 {n-1}=1$.
\end{proof}

\begin{cor}\label{capaDisqueUnite}
   We have $\tau(\overline{B}(0,\rho))=\rho$.
\end{cor}

\begin{proof}
    It follows from Prop. \ref{propTau} and Prop. \ref{propTauCercleUnit} that
    $$\tau(\overline{B}(0,\rho))=\rho\,\tau(B(0,1))=\rho\,\tau(\partial B(0,1))=\rho\,\tau(\U)=\rho.$$
\end{proof}

\begin{prop}
Let us consider $K = \{0\}\cup \{1, \frac 12, \frac 13, \dots\}$. We have $\tau(K)=0$.
\end{prop}

\medskip\begin{proof}
    Let us consider $n\geq 2$. Let $x_n\leq y\leq x_{n-1}\leq\dots\leq x_1$ be a choice of Fekete points in $K\subset[0,1]$, then
	\begin{align*}
		\delta_{n+1}(K)^\frac{(n+1)n}{2}&=(y-x_n)\left(\prod_{i=1}^{n-1}(x_i-y)\right)\left(\prod_{1\leq i<j\leq n}(x_i-x_j)\right)\\
		&\leq(y-x_n)\left(\prod_{i=1}^{n-1}(x_i-y)\right)\delta_n(K)^\frac{n(n-1)}{2}
	\end{align*}
	We can obtain an upper bound using $x_n \geqslant 0$ and since all points are distinct, $x_i\leq\frac{1}{i}$ for $i\in[\![1,n-1]\!]$ and $y\leq\frac{1}{n}$, therefore
	\[(y-x_n)\left(\prod_{i=1}^{n-1}(x_i-y)\right)\leq\left(\frac{1}{n}-0\right)\prod_{i=1}^{n-1}\left(\frac{1}{i}-0\right)=\frac{1}{n\,!}\]
	Hence
	\[\delta_{n+1}(K)^\frac{(n+1)n}{2}\leq\frac{1}{n\,!}\delta_n(K)^\frac{n(n-1)}{2}\]
	Therefore we obtain by induction, using $\delta_2(K)=1$, the following inequality
	\[\delta_{n+1}(K)\leq\left(\prod_{i=2}^ni\,!\right)^{-\frac{2}{(n+1)n}}\]
	Finally, let us notice that this last term tends to 0 when $n\to\infty$. Indeed, we obtain for $n\geq8$
	\begin{align*}
	\log\left(\prod_{i=2}^n i\,!\right)&=\sum_{i=2}^n\sum_{j=2}^i\log j=\sum_{j=2}^n (n+1-j)\log j\\
	&\geq\sum_{j=[\frac{n}{4}]}^{3[\frac{n}{4}]}\dots\geq\sum_{j=[\frac{n}{4}]}^{3[\frac{n}{4}]}\frac{n}{4}\log\left[\frac{n}{4}\right]=\Theta(n^2\log n),\quad n\to\infty
	\end{align*}
\end{proof}

We will be able to give another proof later (cf. Ex. \ref{exCapaSetCountable}, Thm. \ref{thmTauEgalCapa}).

\clearpage

\subsection{Passage from discrete to continuous : equilibrium measure and logarithmic capacity}\label{ssec-capacitelogarithmique}
\paragraph{}
In this section, we will use probability measures on a compact set and a few results on measures. Notions from the theory of measures are provided in the appendix \ref{appendixmeasure}.

\medskip

Fekete points describe an equilibrium when dealing with a finite number of particles. The transfinite diameter describes the ability of a compact set to have electric charges (like in a capacitor). When the number of charges tends to infinity, their distribution becomes continuous. 

\medskip

This passage to the limit can be formalized using the concept of weak convergence-$*$ of measures. Let us remind ourselves of the definition here:

\begin{defi}[Weak-$*$ convergence]
    Let us consider $D\subset\C$ and $(\mu_n)_{n\in\N}$ a sequence of measures on $D$. Let $\mu$ be a measure on $D$. We say that $(\mu_n)_{n\in\N}$ converges weakly to $\mu$, written $\mu_{n} \overset{*}{\longrightarrow} \mu  $, if and only if 
    $$\forall f\in C_c^0(D,\R),\ \lim_{n\to\infty}\int_D f\dif \mu_n = \int_D f\dif\mu$$.
\end{defi}

\begin{ex}[Riemann integral]
The Riemann integral can be considered as the limit measure of a sequence of counting measures: indeed, if $f:[0,1]\to\R$, we have $$\lim_{n\to\infty}\frac 1n \displaystyle\sum_{k=0}^n f\left(\frac kn\right) = \int_0^1 f\ \dif x$$ If we denote by $\nu_n$ the counting measure with respect to the points $\{\frac kn\ | \ k=0,\dots,n\}$, the continuous linear form $\displaystyle\int_0^1$ is the limit measure of the sequence $(\nu_n)$ for the weak convergence-$*$.
\end{ex}

\begin{ex}
If $\nu_n$ is the counting measure with respect to the set of the $n$-th roots of unity $\{e^{\frac{2\pi i k}{n}}\ |\ k=0,\dots,n-1\}$, then 
$$\lim_{n\to\infty}\int f\dif \nu_n = \frac{1}{2\pi}\int_0^{2\pi}f(e^{i\theta})\dif \theta$$
\end{ex}

The space of probability measures on a compact set can be equipped with a topology associated with this notion of convergence: the weak-* topology. The main result is that this space is sequentially compact. 

\begin{thm}[Banach-Alaoglu-Bourbaki]\label{thmBanachAlaoglu}
Let $K$ be a compact subset of $\C$. Let us denote by $ \mathcal{P}(K) $ the set of probability measures on $K$. $\mathcal{P}(K) $ is sequentially a compact set for the weak topology-*.

\medskip

In other words, for any sequence of probability measures $(\mu_n)_n$, there exists a sub-sequence $(\mu_{\varphi(n)})_n$ and a probability measure $\mu$ such that $\mu_{\varphi(n)} \overset{*}{\longrightarrow} \mu  $, \emph{i.e.} 
$$
\forall \phi \in C(X), \int_X \phi ~ \dif \mu_{\varphi(n)} {\longrightarrow}   \int_X \phi ~ \dif \mu
$$
\end{thm}

The proof of this claim is given in the appendix \ref{appendixmeasure}.

\medskip

Let $K\subset\mathbb{C}$ be a compact set and use again the electrostatic model introduced in sub-subsection~\ref{ssec-diamTrans}. We can consider a probability measure $\mu\in\mathcal{P}(K)$ as a distribution of positive electric charges in $K$. For example, a punctual charge at point $z$ can be modelled by $\delta_z$ a Dirac measure centered on $z$, and a discrete distribution can be modelled by a counting measure.

\begin{defi}[Potential, energy]
Let $K\subset\C$ be a compact set. Let us consider $\mu\in\mathcal{P}(K)$ a probability measure on $K$.

  Let us define \emph{the potential of $\mu$ in $z$}
\[U^\mu(z)=\int_K\log\frac{1}{\abs{z-t}}\dif\mu(t)\]

    and the energy of the compact set $K$, with respect to $\mu$
\[I(\mu)=\int_K U^\mu(z)\dif \mu(z)=\iint_{K^2}\log\frac{1}{\abs{z-t}}\dif \mu(z)\dif \mu(t)\]
\end{defi}

At equilibrium, the distribution of the charges tends to minimize the overall energy. We therefore have an energy minimization problem on the space of measures on $K$.
\begin{defi}[Logarithmic capacity, equilibrium measure]

Let us define the \emph{Robin constant}
\[V_K=\inf_{\mu\in\mathcal{P}(K)}I(\mu)\]
and the \emph{logarithmic capacity}
\[\capa(K)=e^{-V_K}.\]

If the infimum $V_K=\inf_{\mu\in\mathcal{P}(K)}I(\mu)$ is reached for a measure $\mu_K$, this measure is called \emph{equilibrium measure}. In that case, we have $\capa(K)=e^{-I(\mu_K)}$.

The potential $U^{\mu_K}$ associated with the equilibrium measure is called \emph{the equilibrium potential}.
\end{defi}

\begin{defi}
	The capacity of a borel set $B\subset\C$ is defined as \[\capa(B)=\sup\left\{\capa(K):K\subset B, K\textrm{ compact}\right\}\]
\end{defi}

\begin{defi}["quasi-almost everywhere"]
Let $K$ be a compact set.
We write that \emph{some property holds quasi-almost everywhere (q-a. e.) in a set $K$}, if there exists a compact subset $S$ of capacity zero such that the property holds for all $\zeta\in K\backslash S$.     
\end{defi}

\begin{ex}
Let us consider $\mu\in\mathcal{P}(K)$; if there exists a point $z\in K$ such that $\mu(\{z\})>0$, then $I(\mu)=+\infty$ by definition. Furthermore, if for all $\mu\in\mathcal{P}(K)$ there exists such a point, then $\capa(K)=0$.
\end{ex}

\begin{ex}\label{exCapaSetCountable}
A countable set has a capacity of zero. In fact, a probability measure on such a set is atomic. The previous example gives us the expected result.
\end{ex}

\begin{rmq}[Heuristic approach]
If we consider $\nu_F$ the averaged counting measure with respect to $F=\{z_1,\dots,z_n\}$, by neglecting the divergent terms, we find the discrete definition given in sub-section~\ref{ssec-diamTrans}
$$I(\nu_F)=-\frac{1}{n(n-1)}\sum_{i\neq j} \log \abs{z_i-z_j} = - \log\prod_{i\neq j} \abs{z_i-z_j}^{\frac{2}{n(n-1)}}$$
To determine the measure that minimizes the energy, it becomes natural to consider the limit measure of the counting measures with respect to the Fekete points.

\medskip

This intuition will be justified later.
\end{rmq}

\begin{lmm}\label{lmmPrincpDesc}
Let $(\mu_n)$ be a sequence of measures of $\mathcal{P}(K)$ such that $\mu_n\stackrel{*}{\to}\mu$, then for all $z\in\C$,
\[U^\mu(z)\leqslant\liminf_{n\to\infty}U^{\mu_n}(z)\]
hence
\[I(\mu)\leqslant\liminf_{n\to\infty}I(\mu_n)\]
\end{lmm}

\begin{proof}
This lemma is a consequence of proposition~\ref{prop-smi-cvf} applied to $t\mapsto\log\frac{1}{\abs{z-t}}$ which is l.s.c. : 
$$\int_K \log\frac{1}{\abs{z-t}}\dif\mu(t) \leqslant \liminf_{n\to\infty}\int_K\log\frac{1}{\abs{z-t}}\dif\mu_n(t)$$
for all $z\in K$, hence the first inequality. The second inequality follows by integrating with respect to $\!\dif \mu(z)$ and by applying Fatou's lemma.
\end{proof}

\begin{thm}[Equilibrium measure]\label{thmExistUniqMesrEqlbr}
There always exists an equilibrium measure. Moreover, if $V_K<+\infty$ (or equivalently $\capa(K)>0$), the equilibrium measure is unique. 
\end{thm}

\begin{proof}
The uniqueness follows directly from the fact that $\mu\mapsto I(\mu)$ is strictly convex on $\{\mu\in\mathcal{P}(K):I(\mu)<\infty\}$  (\cite{saff2013logarithmic}, Chap. I, Thm. 1.3(b), Lem. 1.8). 
The existence results from theorem \ref{thmBanachAlaoglu} which states that the space of probability measures on a compact set is a compact set. In fact, 
$$V_K=\inf_{\mu\in\mathcal{P}(K)}I(\mu)$$ 
Let $(\mu_n)$ be a sequence of measures such that $I(\mu_n)\to V_K$. By compactness, there exists $\mu\in\mathcal{P}(K)$ and a sub-sequence $(\mu_{n_k})$ such that $\mu_{n_k}\cvf\mu$. By lemma \ref{lmmPrincpDesc} :
$$I(\mu)\leqslant\liminf_{k\to +\infty}I(\mu_{n_k})=V_K\leqslant I(\mu)$$ since $V_K$ is the infimum. Hence $V_K=I(\mu)$.
\end{proof}
\clearpage
\paragraph{}
We now have all the keys to prove the equivalence between the two definitions of the capacity which were previously defined.
\begin{thm}\label{thmTauEgalCapa}
Let $K$ be a compact subset of $\C$. We have $$\tau(K)=\capa(K).$$
Moreover, if $\capa(K)>0$ and we denote by $\mu_K$ the equilibrium measure of $K$ and by $\nu_n$ the counting measures with respect to the Fekete points, then 
$$\nu_n \cvf \mu_K$$
The Fekete points are said to be equidistributed with respect to the equilibrium measure of $K$.
\end{thm}

\begin{proof}
Let us rather compare $V_K =\log\frac{1}{\capa(K)}$ and $\log\frac{1}{\tau(K)}$. First of all, let us show that $ \log\frac{1}{\capa(K)} \geqslant \log\frac{1}{\tau(K)}$.
\medskip

Let us define for $z_1,\dots,z_n \in K$ :
\[F(z_1,\dots,z_n) := \sum_{i<j}\log \frac{1}{\abs{z_i-z_j}}\]
Let us remind ourselves of the definition of $\delta_n(K)$ in definition~\ref{defi-PointsdeFekete} : 
\[\delta_n(K)=\max_{z_1,\dots,z_n\in K} \prod_{i < j}\abs{z_i-z_j}^{2/n(n-1)}\]
and that $\tau(K)=\lim \delta_n(K)$. Let us also define the minimal energy associated with a $n$-point configuration:
$$\mathcal{E}_n :=\min_{z_1,\dots,z_n\in K} F(z_1,\dots,z_n)=\frac{n(n-1)}{2}\log \frac{1}{\delta_n(K)} $$
Let $\mu_K$ be an equilibrium measure on $K$.  Let us consider
\begin{align*} 
 J &:= \int\cdots\int F(z_1,\dots,z_n) \dif\mu_K(z_1)\dots\dif\mu_K(z_n) \\
  &= \sum_{ i<j} \int\cdots\int \log \frac{1}{\abs{z_i-z_j}}\dif\mu_K(z_1)\dots\dif\mu_K(z_n)  \\ 
  &= \sum_{ i<j} \int \int \log \frac{1}{\abs{z_i-z_j}}\dif\mu_K(z_i)\dif\mu_K(z_j)\\
  & = \sum_{1 \leqslant i<j \leqslant n} V_K \\
  & =  \frac{n(n-1)}{2}V_K
\end{align*}

On the other hand 
\begin{align*}
    J &\geqslant \int\cdots\int \mathcal{E}_n \dif\mu_K(z_1)\dots\dif\mu_K(z_n) \\
      &= \mathcal{E}_n \\
      &= \frac{n(n-1)}{2}\log \frac{1}{\delta_n(K)}
\end{align*}
Hence $ V_K =\log\frac{1}{\capa(K)}\geqslant \log\frac{1}{\tau(K)}$.

\vspace{1em}
Let us denote by $\nu_n$ the counting measures with respect to the Fekete points $\{z_1,\dots,z_n\}$. By compactness, there exists a sub-sequence $(\nu_{n_k})$ and $\mu\in\mathcal{P}(K)$ so that $\nu_{n_k}\cvf \mu$. Then
\begin{align*}
    I(\mu) &= \iint \log\frac{1}{\abs{z-t}}\dif\mu(z)\dif\mu(t) \\
           &= \lim_{M\to +\infty} \iint \min[M,\log\frac{1}{\abs{z-t}}]\dif\mu(z)\dif\mu(t) \\
           &= \lim_{M\to +\infty}\lim_{k\to +\infty}  \iint \min[M,\log\frac{1}{\abs{z-t}}]\dif\nu_{n_k}(z)\dif\nu_{n_k}(t) \\
           &= \lim_{M\to +\infty}\lim_{k\to +\infty} \frac{1}{n_k^2}\sum_{1\leqslant i,j\leqslant n_k}\min[M,\log\frac{1}{\abs{z_i-z_j}}] \\
           &\leqslant \lim_{M\to +\infty}\lim_{k\to +\infty}\frac{1}{n_k^2} (\sum_{i=1}^{n_k} M +
           2\sum_{1\leqslant i< j\leqslant n_k}\log\frac{1}{\abs{z_i-z_j}})\\
            &\leqslant \lim_{M\to +\infty}\lim_{k\to +\infty}\frac{1}{n_k^2} (n_k M +
           n_k(n_k-1)\log \frac{1}{\delta_{n_k}(K)})\\
           &= \log\frac{1}{\tau(K)}
\end{align*}
Then
$$ I(\mu) \leqslant  \log\frac{1}{\tau(K)} \leqslant  \log\frac{1}{\capa(K)} = V_K \leqslant I(\mu)  $$
Therefore, $\tau(K) = \capa(K)$.

Furthermore, $I(\mu)=V_K=I(\mu_K)$. In the case where $\capa(K)>0$, $\mu = \mu_K$ the unique equilibrium measure on $K$. $(\nu_n)$ is a sequence of elements of a compact set with  $\mu_K$ as its unique accumulation point, therefore
$$\nu_n\cvf\mu_K$$

\end{proof}

\begin{ex}\label{propMesPotentCercleUnit}
The equilibrium measures $\mu_{\U}$ and $\mu_{\overline{B}}$ of the unit circle $\U$ and the unit disk $\overline{B}$ are both $\frac{\mathrm{d}\theta}{2\pi}$ on $\U$, and the corresponding equilibrium potential is \[U^\mu(z)=-\log^+\abs{z}=-\max(0,\log\abs{z})\]
\end{ex}
\begin{figure}[h!]
    \centering
    \includegraphics[width = 13cm]{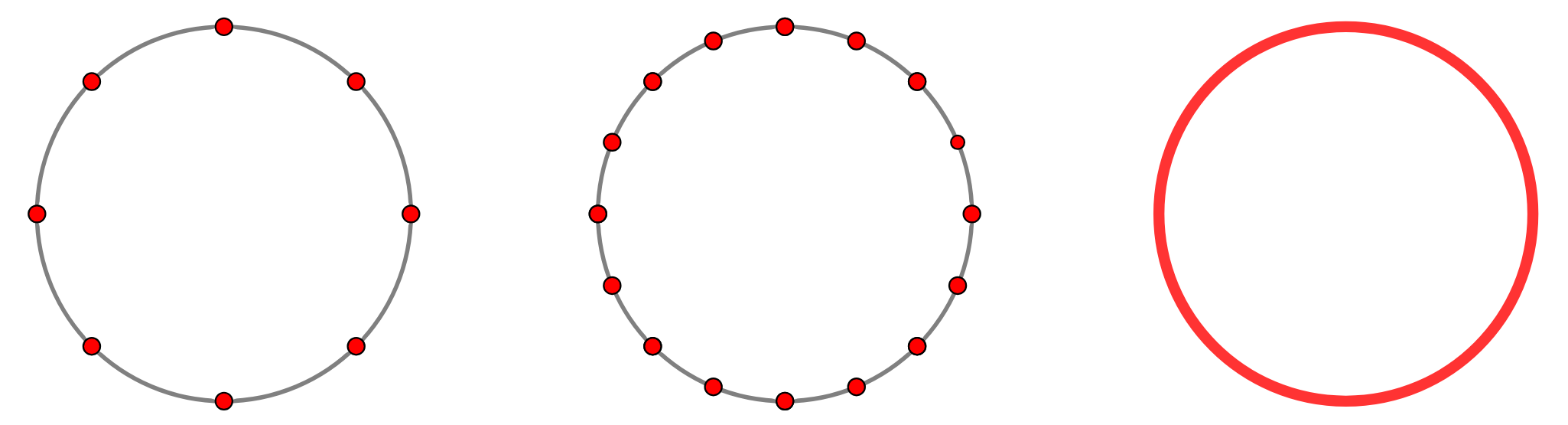}
    \caption{Fekete points for $n\in\{8,16\}$ and equilibrium measure with $K=\U$}
    \label{fig:racineUniteVersCercle}
\end{figure}

\medskip
\vspace{5em}

\begin{proof}
\begin{itemize}[label=$\bullet$]
    \item   Let $f$ be a continuous function. By applying the convergence theorem of Riemann sums to $\theta \mapsto f(e^{i \theta})$, we obtain :
    $$\nu(\U_n)(f)=\frac 1 n \sum^{n-1}_{k=0} f(\exp{\frac{2ik\pi}{n}})\underset{n \rightarrow \infty}{\rightarrow}\int^{2\pi}_0 f(e^{i\theta})\frac{\text{d}\theta}{2\pi}$$
    But $\U_n$ is a set of Fekete points of $\U$ (Prop. \ref{propTauCercleUnit}) therefore $\nu(\U_n)\cvf \mu_\U$ (Thm. \ref{thmTauEgalCapa}). Hence $\boxed{\mu_\U=\frac{\text{d}\theta}{2\pi}}$.
    
   \item    By making the substitution $t:=\theta+\phi$ :
    \[U^{\mu_\U}(z)=-\int^{2\pi}_0 \log(\abs{z-e^{i\theta}})\frac{\text{d}\theta}{2\pi}=-\int^{2\pi}_0\log(\abs{ze^{i\phi}-e^{it}})\frac{\text{d}t}{2\pi} = U^{\mu_\U}(\abs{z})\]
    Let us consider $ r\in\R_+$
    \[
        U^{\nu(\U_n)}(r) = -\frac 1 n \log\abs{\prod^{n-1}_{k=0} (r-e^{\frac{2ik\pi}{n}})}
        = -\frac 1 n \log\abs{r^n-1}
    \]
    \begin{itemize}
        \item  If $r<1$, $\abs{r^n-1}\rightarrow 1$, therefore $U^{\nu(\U_n)}(r)=-\frac 1 n \log\abs{r^n-1}\underset{n \rightarrow \infty}{\rightarrow} 0$
    \item If $r>1$, $U^{\nu(\U_n)}(r)=-\frac 1 n \log\abs{r^n-1} = -\log(r)-\frac 1 n \log\abs{\frac 1 r ^n-1}\underset{n \rightarrow \infty}{\rightarrow}-\log(r)$
    \end{itemize}
    When $r\neq 1$, the function $t\in \U \mapsto -\log\abs{z-t}\in \C$ is continuous, therefore \\ $U^{\mu_\U}(r)=\displaystyle\lim_{n \rightarrow \infty} U^{\nu(\U_n)}(r)=\begin{cases}
    0 \text{ if } r<1 \\ \log\frac 1 r \text{ if } r>1
    \end{cases}$\\
    When $r=1$, by using \ref{propTauCercleUnit} and \ref{thmTauEgalCapa} we obtain :  $0=-\log(\tau(\U))=V_\U=\int_\U U^{\mu_\U}(z)\dif\mu_\U(z)=U^{\mu_\U}(1)$\\
    Hence $\boxed{U^{\mu_\mathbb{U}}(z)=-\log^+\abs{z}}$
    \item From the proposition \ref{propCapa}(c), it follows that $\capa(\U)=\capa(\overline{B})$ and $V_{\overline{B}}=I(\mu_{\U})\neq +\infty$ thanks to the uniqueness of the equilibrium measure, hence $\mu_{\overline{B}}=\mu_{\U}$ and $U^{\mu_{\overline{B}}}=U^{\mu_{\U}}$.
    \end{itemize}
    \end{proof}

\paragraph{}
Now that we have unified the two notions of capacity previously defined, let us introduce a few properties of our capacity.

\begin{prop}\label{propCapa}
Let us consider $(K_n)$ compact subsets of $\C$.
\begin{enumerate}[label=(\alph*)]
    \item Let us consider $K_1\subset K_2$, then $\capa(K_1)\leq\capa(K_2)$.
    \item Let us consider $\alpha,\beta\in\C$, then $\capa(\alpha K+\beta)=\abs{\alpha}\capa(K)$.
    \item $\capa(K)=\capa(\partial K)$.
    \item Let us consider $(K_n)_{n}$ a decreasing sequence and let us consider $K=\bigcap_n K_n$, then $\capa(K)=\lim_n\capa(K_n)$.
    \item Let us consider $(K_n)_{n}$ an increasing sequence and let us assume that $K=\bigcup_n K_n$ is a compact set, then $\capa(K)=\lim_n\capa(K_n)$.
    \item Assume that $\capa(K)=0$. Then $\nu(K)=0$, for all finite measures with compact support $\nu$ on $\C$ such that $I(\nu)<+\infty$. In particular, Lebesgue measure $m_2(K)=0$.
\end{enumerate}
\end{prop}

\begin{proof}
(a) and (b) are immediate consequences of the definition.
\medskip\par (c) results from theorem \ref{thmTauEgalCapa} and proposition \ref{propTau}.
    \medskip\par (d) According to (a), it is enough to show that $\capa(K)\geq\lim_n\capa(K_n)$, and moreover we can assume that $\capa(K_n)>0$ for all $n$. Let $\mu_{K_n}$ be the equilibrium measure of $K_n$, then $\mu_{E_n}\in\mathcal{P}(K_0)$. According to \hyperref[thmBanachAlaoglu]{Banach-Alaoglu-Bourbaki} theorem, we can extract a sub-sequence $(\mu_{E_{\varphi(n)}})_n$ which converges weakly-* to a measure $\mu^\star\in\mathcal{P}(K_0)$. Then lemma \hyperref[lmmPrincpDesc]{2.2} gives us
\begin{align*}
    I(\mu^\star)\leq\liminf_nI(\mu_{\varphi(n)})=\liminf_n\left(-\log\capa(K_{\varphi(n)})\right)=-\lim_n\log\capa(K_n)
\end{align*}
On the other hand, Prop. \ref{propSuppDecroi} implies that 
\[\supp\mu^\star\subset\bigcap_n\supp\mu_{E_{\varphi(n)}}\subset\bigcap_nK_{\varphi(n)}=\bigcap_nK_n=K\]
Therefore $\mu^\star\in\mathcal{P}(K)$, then
\[\capa(K)\geq e^{-I(\mu^\star)}\geq \exp\left(\lim_n\log\capa(K_n)\right)=\lim_n\capa(K_n)\]
\medskip\par (e) According to (a), it is enough to show that $\capa(K)\leq\lim_n\capa(K_n)$, and moreover we can assume that $\capa(K)>0$. Let us consider $\mu\in\mathcal{P}(K)$. Since $\lim_n\mu(K_n)=\mu(K)=1$, we have $\mu(K_n)>0$ for a sufficiently large $n$ . Then for such a $n$, we consider $\frac{\mu|_{K_n}}{\mu(K_n)}\in\mathcal{P}(K_n)$. We have
\[-\log\capa(K_n)\leq I\left(\frac{\mu|_{K_n}}{\mu(K_n)}\right)=\frac{1}{\mu(K_n)^2}\int_{K_n}\int_{K_n}\log\frac{1}{\abs{z-t}}\dif\mu(t)\dif\mu(z)\]
Since the domain of integration $K_n\times K_n$ increases with $n$ and tends to $K\times K$, because $(K_n)$ is increasing and tends to $K$, we obtain that
\[-\log\left(\lim_n\capa(K_n)\right)\leq \frac{1}{\lim_n\mu(K_n)^2}\int_{K}\int_{K}\log\frac{1}{\abs{z-t}}\dif\mu(t)\dif\mu(z)=I(\mu)\]
where we apply the monotone convergence theorem, after verifying that on $K\times K$, $\log\frac{1}{\abs{z-t}}$ is lower bounded by $\log(1/\delta_2(K))$. In particular, when $\mu=\mu_K$ the equilibrium measure of $K$, $I(\mu_K)=-\log\capa(K)$, we obtain that $\capa(K)\leq\lim_n\capa(K_n)$.
\medskip\par (f) Let $\nu$ be a measure verifying the theorem's assumptions. If $\nu(K)>0$, then $\frac{\nu|_K}{\nu(K)}\in\mathcal{P}(K)$, therefore since $\capa(K)=0$, we have
\[I\left(\frac{\nu|_K}{\nu(K)}\right)=\frac{1}{\nu(K)^2}\int_K\int_K\log\frac{1}{\abs{z-t}}\dif\nu(t)\dif\nu(z)=+\infty\]
Then
\begin{align*}
    I(\nu)&=\iint\log\frac{1}{\abs{z-t}}\dif\nu(t)\dif\nu(z)\\
    &=\left(\int\int_{(K\times K)^c}+\int_K\int_K\right)\log\frac{1}{\abs{z-t}}\dif\nu(t)\dif\nu(z)\\
    &\geqslant \left(1-\nu(K)^2\right)\log\frac{1}{\delta_2(\supp(\nu))}+\nu(K)^2I\left(\frac{\nu|_K}{\nu(K)}\right)=+\infty
\end{align*}
which is contradicted by $I(\nu)<+\infty$.
\medskip\par If $K\subset B(0,r)$, we have $m_2|_{B(0,r)}(\C)=\pi r^2$ and $I(m_2|_{B(0,r)})<+\infty$. In fact, for all $z\in B(0,r)$,
\begin{align*}
    \int_{B(0,r)}\log\frac{1}{\abs{z-t}}\dif m_2(t)&\leq\int_{B(z,2r)}\log\frac{1}{\abs{z-t}}\dif m_2(t)\\
    &=\int_{0}^{2\pi}\int_{0}^{2r}\left(\log\frac{1}{\rho}\right)\rho\dif\rho
    \leq 4\pi r\left\|\rho\log\rho\right\|_{L^\infty(]0,r])}<+\infty
\end{align*}
Therefore we can apply what we have just proved to obtain that $m_2(K)=0$.
\end{proof}

\begin{cor}\label{cor-continutecapacitesegment}
    Let $a=(a_1 < \dots < a_{2n})$ be a set of strictly increasing real numbers. Let us consider $E_a = \displaystyle\bigcup_{k=0}^{n-1} [a_{2k+1},a_{2k+2}]$. The capacity of $E_a$ continuously varies with $a$.
\end{cor}

\begin{proof}
    Let us consider $d=\min_i\{a_{i+1}-a_i\}$ and $\varepsilon\in]0, d/2[$. Let
    \[a_{\varepsilon,+} = (a_1-\varepsilon, a_2+\varepsilon,\dots a_{2n-1}-\varepsilon, a_{2n}+\varepsilon),\ a_{\varepsilon,-} = (a_1+\varepsilon, a_2-\varepsilon,\dots a_{2n-1}+\varepsilon, a_{2n}-\varepsilon), \]
    We have $E_{a_{\varepsilon,-}}\subset E_a \subset E_{a_{\varepsilon,+}}$. 
    \medskip
    
    The left (resp. right) continuity of $\capa(E_a)$ with respect to $a$ is equivalent to 
    $\capa(E_a) = \sup \capa(E_{a_{\varepsilon,-}})$ 
    (resp. $\capa(E_a) = \inf \capa(E_{a_{\varepsilon,+}})$). These two conditions are insured by points (d) and (e) of the previous proposition.
    
\end{proof}

\begin{defi}\label{defEssentialClosure}
    Let $K\subset\C$ be a compact subset. Let us denote by $\Omega_K$ the unbounded connected component of $\C\backslash K$.
    Let us define \emph{the external boundary} of a compact set $K\subset\C$ as the boundary of $\Omega_K$, subset of $\C$:
    \[\partial_\infty K:=\partial\Omega_K\]
    and \emph{the essential closure} as the complement of $\Omega_K$ with respect to $\C$:
    \[K^{\ce}:=\C\backslash\Omega_K\]
\end{defi}

\begin{cor}\label{corPropCapa}
Let $K\subset\C$ be a compact set, then:
\begin{enumerate}[label=(\alph*)]
    \item If $\mathring{K}$ is non-empty, then $\capa(K)>0$; in other words, $\capa(K)=0$ implies that $K=\partial K$;
    \item $\Omega_{K}=\Omega_{\partial K}=\Omega_{\partial_\infty K}=\Omega_{K^{\ce}}$;
    \item $\partial_\infty K=\partial K^{\ce}$, $K^{\ce}\supset K\supset\partial_\infty K$;
    \item $\capa(K^{ce})=\capa(\partial_\infty K)=\capa(K)=\capa(\partial K)$.
    \item Let $\mu_K$ be the equilibrium measure of $K$. If $\capa(K)>0$, then $\supp(\mu_K)\subset\partial_\infty K$ and $\mu_K=\mu_{K^{\ce}}=\mu_{\partial_\infty K}$.
    \item Let $(K_n)$ be an increasing sequence of compact subsets of $\C$, and $B=\bigcup_nK_n$. Then $\capa(B)=\lim_n\capa(K_n)$.
\end{enumerate}
\end{cor}

\begin{proof}
(a) results from Prop. \ref{propCapa} (f) using Lebesgue measure.
\medskip\par (b) and (c) result from the definition.
\medskip\par (d) can be deduced from (c) and Prop. \ref{propCapa} (a),(c). And (e) results from it, given the uniqueness (Thm. \ref{thmExistUniqMesrEqlbr}).
\medskip\par (f) Let us suppose $K_n$ of finite capacity. Let $K\subset B$ be a compact. Then $(K\cap K_n)$ is an increasing sequence that converges to $K$, then we have
\[\capa(K)=\lim_n\capa(K\cap K_n)\]
hence $\capa(B)\leq\lim_n\capa(K_n)$. The other inequality is trivial.
\end{proof}

\subsection{Potentials and semi harmonic functions}\label{ssec-potentielsemiharmonic}
\paragraph{}
For some probability measure $\mu$, the associated potential $U^\mu$ is a function of a complex variable belonging to a specific family of functions: \emph{semi-harmonic functions}. These are semi-continuous functions which verify the maximum/minimum modulus principle. Notions on these functions are provided in the appendix \ref{appendixeSemiharmonic}. 

\medskip

By using properties of these functions, we will prove several important results on potentials and capacities. First, we will prove Frostman's theorem (Thm. \ref{thmFrostman}) which states that the equilibrium potential of a compact set $K$ has the shape of a platter on $K$. Secondly we will prove theorem \ref{thmAppliHolom} which states that under a few assumptions of regularity, a holomorphic function transforms the capacity of a compact set according to its monomial of highest degree. This theorem and its two corollaries \ref{cor-biholo}, \ref{corCapaciteImageRecipPolynome}, provide a useful tool to calculate some capacities. 

\medskip

A few technical proofs are gathered in the appendix \ref{appendixthmPrincpMaxPotent}.

\subsubsection{Frostman's theorem}

To begin with, let us remind ourselves of the definition of semi-harmonic functions and give two important examples of super-harmonic functions.
\begin{defi}[(super-,sub-)harmonic functions]
    Let $D\subset\C$ be an open set. A function $f:D\to\R$ is called \emph{harmonic} (\resp. \emph{super-harmonic, sub-harmonic}) if it is continuous (\resp. lower semi-continuous, upper semi-continuous) and if it verifies the \emph{mean value property} (\resp. \emph{super-mean, sub-mean}): for all $z\in D$, if the disk $\{\abs{\zeta-z}\leq r\}\subset D$, we have
    \[f(z)=\frac{1}{2\pi}\int_0^{2\pi}f(z+re^{i\theta})\dif\theta\quad(\resp.\;\geq,\leq)\]
\end{defi}

\begin{ex}\label{exLogSuperharmonic}
$g(z)=\log\frac{1}{\abs{z-t}}$ is superharmonic and harmonic at points $z\neq t$. \qed
\end{ex}

\begin{ex}\label{exPotentialSuperharmonic}
Let $\mu$ be a positive measure with compact support $K$, then the potential
\[U^\mu(z)=\int\log\frac{1}{\abs{z-t}}\dif\mu(t)\]
is superharmonic on $\C$ and harmonic on $\C\backslash K$.
\end{ex}

\begin{proof}
    $U^\mu$ is lower semi-continuous since $\log\frac{1}{\abs{z-t}}$ is for all $t$. By using Fubini-Tonelli theorem, we obtain
    \[\frac{1}{2\pi}\int_0^{2\pi}U^\mu(z+re^{i\theta})\dif\theta=\int\frac{1}{2\pi}\int_0^{2\pi}\log\frac{1}{\abs{z+re^{i\theta}-t}}\dif\theta\dif\mu(t).\]
    The example \ref{exLogSuperharmonic} gives us that this last integral is smaller than $\displaystyle\int\log\frac{1}{\abs{z-t}}\dif\mu(t)=U^\mu(z)$. Therefore, $U^\mu$ is superharmonic on $\C$. By considering $z\notin K$ and $r<\dist(z,K)$, it follows from example \ref{exLogSuperharmonic} that this last integral is equal to $U^\mu(z)$, hence the harmonicity.
\end{proof}

\begin{thm}[Maximum principle for potentials]\label{thmPrincpMaxPotent}
    Let $\mu$ be a finite positive measure with compact support. If $U^\mu(z)\leq M$ for all $z\in\supp(\mu)$, then the same goes for all $z\in\C$.
\end{thm}

The proof is given in the appendix \ref{appendixeSemiharmonic}, Corollary \ref{appendixthmPrincpMaxPotent}.

\medskip

Let us state and prove an important theorem on equilibrium potentials : Frostman's theorem, which claims that the equilibrium potential has the shape of a platter and is upper bounded every where by Robin's constant (cf. Figure \ref{fig-frost}).

\begin{thm}[Frostman]\label{thmFrostman}
Let $K$ be a compact subset of $\C$ such that $\capa(K)>0$, then
\begin{enumerate}
    \item $ U^{\mu_K}(z)\leqslant V_K$ for all $z\in\C$
    \item $U^{\mu_K}(z) = V_K$ for all $z\in K-S$ where $\capa(S)=0$
    \item $U^{\mu_K}(z) < V_K$ for all $z\in\Omega$ where $\Omega$ is the non bounded connected component of $\C-K$
\end{enumerate}
\end{thm}

\begin{figure}[h!]
    \centering
    \includegraphics[scale = 0.8]{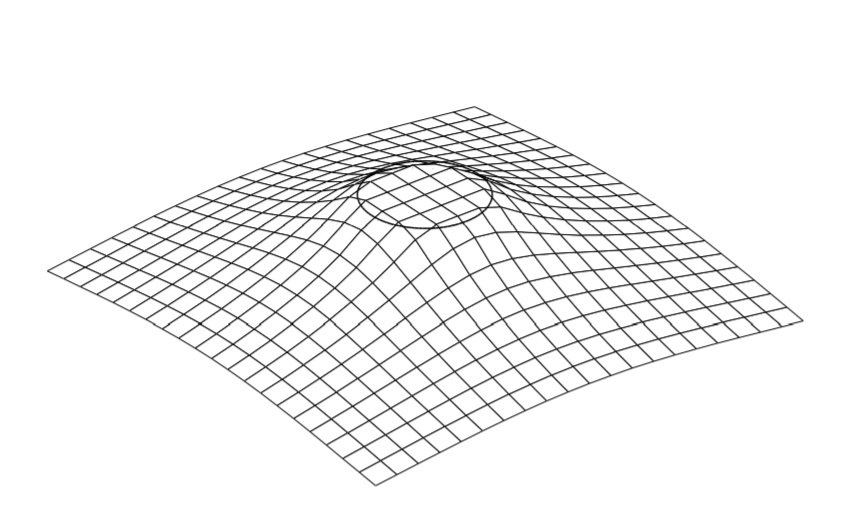}
    \caption{Frostman's theorem}
    \label{fig-frost}
\end{figure}

\begin{proof}
The formal argument proceeds in three steps. First of all, let us define for $n \geq  1 $ the following sets: $ K_{n}=\left\{z \in K | U^{\mu_{K}} \leq V_{K}-1 / n\right\} $ and $ L_{n}=\left\{z \in \operatorname{support}\left(\mu_{K}\right) | U^{\mu_{K}}(z)>V_{K}+1 / n\right\}$.
We will first show that the $K_n$ have capacity zero, then that the $L_n$ are empty, and we will finally conclude, using the maximum principle for potentials (\ref{thmPrincpMaxPotent}).
\\

Let us argue by contradiction to prove that $\operatorname{Cap}(K_n) = 0$. Assume there exists $n \geq 1$ such that $\operatorname{Cap}(K_n) > 0$. We remind that $
V_{K}=I\left(\mu_{K}\right)=\int U^{\mu_{K}} d \mu_{K}$, therefore, there exists some $z_0$ in the support of $\mu_K$ verifying $U^{\mu_{K}}\left(z_{0}\right) \geq V_{K}$. By using the lower semi-continuity of $U^{\mu_{K}}$, there exists a closed ball denoted by $B$ with radius $r>0$ and center $z_0$, on which we have $U^{\mu_{K}}(z)>V_{K}-1 / 2 n$. Then, we have $B \cap K_{n}=\emptyset$, and $\mu_K(B) = a > 0$ because $z_0$ is in the support of $\mu_K$. To get a contradiction, we want to build a measure that contradicts the fact that $V_K$ is minimal. To do so, we use the assumption that $\operatorname{Cap}(K_n) > 0$. In fact, we can consider a measure $\mu$ on $K_n$ such that $I(\mu)$ is finite. We then build the following measure: 
$$
\sigma=\left\{\begin{array}{ll}{\mu,} & {\text { on } K_{n}} \\ {-\mu_{K} / a,} & {\text { on } B} \\ {0,} & {\text { on the rest }}\end{array}\right.
$$

Let us then consider the family of probability measures $\mu_{t}=\mu_{K}+t \sigma $ for $t \in ]0;a[$. Given how we built $B$, we obtain : $$ I\left(\mu_{K}\right)-I\left(\mu_{t}\right) \geq 2 t\left(V_{K}-1 / 2 n-V_{K}+1 / n\right)+O\left(t^{2}\right) $$ Therefore for all $t$ sufficiently close to $0$ (we remind that here $n$ is fixed), we have $I\left(\mu_{K}\right) = V_K \geq I\left(\mu_{t}\right) $, which contradicts the fact that $\mu_K$ is minimal since here , $\mu_t \neq \mu_K$.
\\

Secondly, let us show by contradiction that $L_n = \emptyset$. If one of the $L_n$ is non-empty, using the same argument as earlier, there exists a closed ball $\tilde{B}$ on which $U^{\mu_{K}}(z)>V_{K}+1 / n$. As we did earlier, we consider $\tilde{a} = \mu_K(\tilde{B}) > 0$. According to the previous point, we obtain for all $n$, $\mu_K(K_n) = 0$, therefore, $U^{\mu_{K}}(z) \geq V_{K}$ except on a set with measure zero. We then have :
$$
\begin{aligned} V_{K}=I\left(\mu_{K}\right) &=\int_{K} U^{\mu_{K}} d \mu_{K} \\ &=\int_{\tilde{B}} U^{\mu_{K}} d \mu_{K}+\int_{K \backslash \tilde{B}}  U^{\mu_{K}} d \mu_{K} \\ & \geq\left(V_{K}+\frac{1}{n}\right) \tilde{a}+V_{K}(1-\tilde{a}) >V_{K} \end{aligned}
$$ We obtain a contradiction, and the $L_n$ are empty.
\\

These two facts allow us to prove the three claims of the theorem : since all the $L_n$ are empty, we have the inequality of the first claim on the support of $\mu_K$, therefore, on $\mathbb{C}$, thanks to the maximum principle. The second claim results from the fact that $\operatorname{Cap}(\bigcup_{n} K_{n}) = 0$ which gives us that $U^{\mu_{K}}(z) \geq V_{K}$ on $K \backslash \bigcup_{n} K_{n}$, and from the first claim which states that $U^{\mu_{K}}(z) \leq V_{K}$. Finally, the last claim results from the harmonicity of $U^{\mu_K}$ on $\Omega$.
 
\end{proof}

\begin{ex}
When $K=\U$ (Prop. \ref{propMesPotentCercleUnit}), the potential $U^{\mu_\U}(z)=
\begin{cases}
0=V_\U \text{ if }\abs{z}\leqslant 1 \\
\log(\frac{1}{ \abs{z}}) \text{ if } \abs{z} >1
\end{cases}$ verifies Frostman's theorem with $S=\emptyset$.
\end{ex}
\begin{cor}\label{corValeurInt(K)}
Let $K\subset\C$ be a compact set with $\capa(K)>0$ and $\mu_K$ its equilibrium measure. Then $U^{\mu_K}\equiv V_K$ in the interior of $K$. Moreover, since $\mu_K=\mu_{K^{\ce}}$ (Cor. \ref{corPropCapa}(e)), we have  $U^{\mu_K}\equiv V_K$ in the interior of $K^{\ce}$.
\end{cor}

\begin{proof}
	\hyperref[thmFrostman]{Frostman's theorem} implies that
	\[U^{\mu_K}(z)=V_K,\quad q-a.e. \text{~in~} \partial K\]
	On the other hand, $U^{\mu_K}$ is superharmonic and lower bounded by $\log(1/d_2(K))$ in $K$. Then the \hyperref[thmGenerlPrincpMin]{maximum principle} can be applied to $U^{\mu_K}$ on each connected component of $\mathring{K}$ to obtain $U^{\mu_K}\geq V_K$ in $\mathring{K}$. However, \hyperref[thmFrostman]{Frostman's theorem} already gives us an upper bound, $V_K$. Therefore $U^{\mu_K}= V_K$ in $\mathring{K}$.
	
\end{proof}

\subsubsection{Calculation of capacities}
\paragraph{}
The goal of this section is theorem \ref{thmAppliHolom}, which states that a holomorphic function transforms the capacity according to its monomial of highest degree, and two corollaries which serve as tools to calculate some capacities.

\medskip

The main tool of this proof is Green's function with respect to a compact set.

\begin{thm}\label{thmGreen}
Let $K\subset\C$ be a compact set with $\capa(K)>0$ and $\Omega_K$ the non bounded connected component of $\C\backslash K$. Then there exists a \emph{unique} function $g_K(\cdot,\infty):\Omega_K\to\R$ \emph{characterized} by the following properties:
\begin{enumerate}
    \item $g_K(\cdot,\infty)$ is harmonic on $\Omega_K$ and bounded outside of all the neighborhoods of $\infty$;
    \item $g_K(\cdot,\infty)-\log\abs{\cdot}$ is bounded on a neighborhood of $\infty$;
    \item $\lim_{z\to\zeta}g_K(z,\infty)=0$ \emph{q-a.e.} in $\partial\Omega_K$.
\end{enumerate}
\end{thm}

\begin{proof}
The existence is given by \hyperref[thmFrostman]{Frostman's theorem} and example \ref{exPotentialSuperharmonic} if we consider
\begin{equation}\label{relationPotentielG_K}
    g_K(z,\infty)=V_K-U^{\mu_K}(z).
\end{equation}
\par Let $g, g'$ be two functions verifying these three properties, then $g-g'$ is harmonic and bounded on $\Omega_K$, and therefore can be extended to a harmonic function on $\Omega_K\cup\{\infty\}$. According to the second point, we have $g(z)>0$ (\resp. $g'(z)>0$) for $R\gg0$ so that $K\subset B(0,R)$; the \hyperref[thmGenerlPrincpMin]{minimum principle} applied on $\Omega_K\cap B(0,R)$ gives us that $g(z)>0$ (\resp. $g'(z)>0$) on $\Omega_K\cap B(0,R)$, and then on $\Omega_K$. The uniqueness results from the third property and the \hyperref[thmGenerlPrincpMin]{minimum principle} applied twice on $\Omega_K\cup\{\infty\}$ to $g(z)-g'(z)$ and $g'(z)-g(z)$ respectively.
\end{proof}

\begin{defi}
    Let us denote by $g_K(\cdot,\infty):\Omega_K\to\R$ \emph{the Green's function with respect to $K$, with a pole at infinity}.
\end{defi}

\begin{cor}\label{corPositiveFonctGreen}
    Let $K\subset\C$ be a compact set such that $\capa(K)>0$. We then have $g_K(\cdot,\infty)>0$ on $\Omega_K$.
\end{cor}

\begin{cor}\label{corGreenDependOnClotureEss}
If two compact sets $K_1,K_2\subset\C$ are such that $\capa(K_1)>0$ or $\capa(K_2)>0$, and that $\Omega_{K_1}=\Omega_{K_2}$, then
\[g_{K_1}(\cdot,\infty)=g_{K_2}(\cdot,\infty).\]
In particular, let $K\subset\C$ be a compact set with $\capa(K)>0$, then
\[g_{K}(\cdot,\infty)=g_{\partial K}(\cdot,\infty)=g_{\partial_\infty K}(\cdot,\infty)=g_{K^{\ce}}(\cdot,\infty).\]
\end{cor}

\begin{proof}
Note that the three claims in Thm. \ref{thmGreen} only depend on $\Omega_K$, the non bounded connected component of $\C\backslash K$ ; therefore the result immediately follows from the definition and Cor. \ref{corPropCapa}. 
\end{proof}

\begin{thm}\label{thmPotent&Green}
Let $K\subset\C$ be a compact set with $\capa(K)>0$ and $g_K(\cdot,\infty)$ Green's function with respect to $K$, with a pole at infinity. Then
\begin{enumerate}
    \item $g_K(z,\infty)=\log\abs{z}+V_K+o_{z\to\infty}(1)$;
    \item The equilibrium potential of $K$ is $U^{\mu_K}=V_K-g_K(\cdot,\infty)$; then (Thm. \ref{thmFrostman}) $g_K(\cdot,\infty)>0$ on $\Omega_K$.
\end{enumerate}
\end{thm}

\begin{proof}
    It immediately results from the formula (\ref{relationPotentielG_K}) given in the proof of theorem \ref{thmGreen} and in the definition of the potential $U^{\mu_K}$.
\end{proof}

Let us define the notion of regularity of a compact set. We will see later that it is the right assumption to link the capacity of two compact sets and thanks to Prop \ref{propKepsilonConditionCone}, it is a special case that will be useful in a more general approach.
\begin{defi}
	Let $K\subset\C$ be a compact set such that $\capa(K)>0$. We call $\zeta_0\in\partial\Omega_K$ a \emph{regular point} if
	\[\lim_{\Omega_K\ni z\to\zeta_0}U^{\mu_K}(z)=V_K ;\]
	otherwise, we write that $\zeta_0$ is an \emph{irregular point}.
	We write that $\Omega_K$ is a \emph{regular set} if all the points of $\partial\Omega_K$ are regular.
\end{defi}

According to the relation between the potential and Green's function (Thm. \ref{thmPotent&Green}), we have the following property :

\begin{cor}
    A point $\zeta_0\in\partial\Omega_K$ is regular if and only if $\lim_{z\to\zeta_0}g_K(z,\infty)=0$.
\end{cor}

\begin{prop}\label{propConditEquivRegulier}
	Let $K\subset\C$ be a compact set such that $\capa(K)>0$. Then $U^{\mu_K}(\zeta_0)=V_K$ implies the continuity of $U^{\mu_K}$ at point $\zeta_0$. The reciprocal implication is true if $\zeta_0\in\supp(\mu_K)$.
\end{prop}

\begin{proof}
	It follows from the semi-continuity of $U^{\mu_K}$ and \hyperref[thmFrostman]{Frostman's theorem} that:
	\[U^{\mu_K}(\zeta_0)\leq\liminf_{z\to\zeta_0}U^{\mu_K}(z)\leq\limsup_{z\to\zeta_0}U^{\mu_K}(z)\leq V_K.\]
	Therefore if $U^{\mu_K}(\zeta_0)=V_K$, $U^{\mu_K}$ is continuous at point $\zeta_0$.
	\medskip\par Conversely, if $U^{\mu_K}$ is continuous at point $\zeta_0$ and $U^{\mu_K}(\zeta_0)<V_K$, then there exists $r>0$ such that $U^{\mu_K}<V_K$ on $\overline{B(\zeta_0,r)}$, therefore according to \hyperref[thmFrostman]{Frostman's theorem}, $\capa(\overline{B(\zeta_0,r)})\cap K)\leq\capa(S)=0$, then according to Prop. \ref{propCapa}(f), we have
	\[\mu_K(\overline{B(\zeta_0,r)}\cap K)=0.\]
	Therefore $\mu_K(\overline{B(\zeta_0,r)})=0$ and we deduce from it that $\zeta_0\notin\supp(\mu_K)$.
\end{proof}

\begin{cor}\label{corSuppCapaNonNullLocally}
	Let us consider $\zeta_0\in\supp(\mu_K)$ then for all $r>0$, $\capa(B(\zeta_0,r)\cap K)>0$.
\end{cor}

\begin{proof}
    It is indeed the contrapositive of the last argument of the previous proof.
\end{proof}

\begin{cor}
	The set of all irregular points has capacity zero.
\end{cor}

\begin{proof}
    Let us assume that $\capa(K)>0$. Then the corollary results from Frostman's theorem.
\end{proof}

To get a sufficient condition of regularity, let us introduce the following notion:

\begin{defi}\label{defConditionCone}
Let us consider $\Omega\subset\C$. We write that $\Omega$ verifies the \emph{cone condition} if for all $\zeta\in\partial\Omega$, there exists $\zeta_1\neq\zeta$ such that the segment $[\zeta,\zeta_1]\subset\C\backslash\Omega$.
\end{defi}

\begin{prop}\label{propConditionConeReg}
	Let $K\subset\C$ be a compact set with $\capa(K)>0$. If $\Omega_K$ verifies the cone condition, then $\Omega_K$ is a regular set.
\end{prop}

The proof is given in appendix \ref{propConditionConeRegProof}.
\medskip

The following proposition will be useful in order to consider the domain regular.

\begin{prop}\label{propKepsilonConditionCone}
    Let $K\subset\C$ be a non-empty compact set and let $K^\epsilon=\{z\in\C:\dist(z,K)\leq\epsilon\}$. Then $\capa(K^\epsilon)>0$ and $\Omega_{K^\epsilon}$ verifies the cone condition, therefore $\Omega_{K^\epsilon}$ is regular.
\end{prop}

\begin{proof}
    Since the interior of $K^\epsilon$ is non-empty, we have $\capa(K^\epsilon)>0$. Moreover, by definition, all points $\zeta$ in $\partial K^\epsilon$ are exactly at a distance of $\epsilon $ from $K$, therefore there exists $\zeta_1\in K$ such that $[\zeta,\zeta_1]\subset K^\epsilon\subset\C\backslash\Omega_K$.
\end{proof}

\clearpage

We will now prove a few results that allow us to compare the capacity of two compact sets thanks to holomorphic functions and the notion of regularity previously defined.
\begin{lmm}
	Let $f$ be a non locally constant holomorphic function defined at $\infty$ by $f(\infty)=\infty$. Then there exists some integer $n\in\N^*$ and some $A\neq0$ such that
	\[f(z)\sim Az^n,\quad z\to\infty\]
	\[\log\abs{f(z)}=n\log\abs{z}+\log\abs{A}+o(1),\quad z\to\infty\]
	Let us define $A_f=A,n_f=\ord_\infty(f)=n$. \qed
\end{lmm}

\begin{thm}\label{thmAppliHolom}
	Let $K_1,K_2\subset\C$ be two non-empty compact sets. Let $f:\Omega_{K_1}\cup\{\infty\}\to\Omega_{K_2}\cup\{\infty\}$ be a non constant holomorphic function such that $f(\infty)=\infty$.
	Then,
	\begin{enumerate}[label=(\alph*)]
		\item $\abs{A_f}\capa(K_1)^{n_f}\geq\capa(K_2)$.
		\item Moreover assume that :
			\begin{enumerate}[label=(\roman*)]
				\item $\capa(K_2)>0$;
				\item $f^{-1}(\infty)=\{\infty\}$;

				\item $\Omega_{K_2}$ is regular;
				\item $f$ can be continuously extended to the boundaries $\partial\Omega_{K_1}\to\partial\Omega_{K_2}.$
			\end{enumerate}
			Then, $\capa(K_1)$>0, $\Omega_{K_1}$ is a regular set and
			\[g_{K_1}(\cdot,\infty)=\frac{1}{n_f}g_{K_2}(f(\cdot),\infty)\]
			\[\abs{A_f}\capa(K_1)^{n_f}=\capa(K_2).\]
	\end{enumerate} 
\end{thm}

The complete proof is provided in the appendix \ref{appendixthmPrincpMaxPotent}. 

\medskip

To understand how to apply this theorem, let us see when the conditions in $(b)$ are verified:
	\medskip\par (ii) holds when $f$ is biholomorphic, or when $f$ is a polynomial;
	\par (iii) holds when $\Omega_K$ verifies the cone condition (Prop. \ref{propConditionConeReg}).

\noindent We therefore have the two following corollaries (by applying the same technique as in Prop. \ref{propKepsilonConditionCone} if necessary):

\begin{cor}\label{cor-biholo}
	If $f$ is biholomorphic, we have $n_f=n_{f^{-1}}=1$ and $A_fA_{f^{-1}}=1$; Let us apply (a) twice to obtain (with notation $f'(\infty)=A_f$)
	\[\abs{f'(\infty)}\capa(K_1)=\capa(K_2).\]
	\qed
\end{cor}

\begin{cor}\label{corCapaciteImageRecipPolynome}
	Assume that $f$ is a polynomial function of degree $n$:
	$$f(z)=\sum_{i=0}^na_ix^i,\quad a_n\neq0.$$ Let $K\subset\C$ be a compact set. Then
	\[\abs{a_n}\capa(f^{-1}(K))^{n}=\capa(K).\]
\end{cor}

\begin{proof}
    Assume that $K\neq\emptyset$. Let us try to apply theorem \ref{thmAppliHolom} (b).
    Since $f$ is a polynomial function, (ii) is verified for all compacts in $\C$.
    For $\epsilon>0$, let us consider $K^\epsilon=\{z\in\C:\dist(z,K)\leq\epsilon\}$. Let us show that the conditions (i),(iii),(iv) of the theorem are verified for the compact sets $f^{-1}(K^\epsilon)$ and $K^\epsilon$ and that $f(\Omega_{f^{-1}(K^\epsilon)})\subset\Omega_{K^\epsilon}$; It will result in
	\[\abs{a_n}\capa(f^{-1}(K^\epsilon))^{n}=\capa(K^\epsilon).\]
	When $\epsilon\to0$, it allows us to conclude.
	\medskip\par (i) $K^\epsilon$ has a non-empty interior, therefore $\capa(K^\epsilon)>0$ (Cor. \ref{corPropCapa}).
	\medskip\par (iii) results from Prop. \ref{propConditionConeReg} and Prop. \ref{propKepsilonConditionCone}.
	\medskip\par (iv) We will actually prove this claim for all compact sets $K$ (instead of $K^\epsilon$). A topological analysis shows that $f(\partial f^{-1}(K))=\partial K$. We have left to prove that $f(\partial\Omega_{f^{-1}(K)})\subset\partial\Omega_{K}$.		
	We have $f(\infty)=\infty$ and $K$ a compact set, therefore $f^{-1}(K)$ is a compact set; likewise, the inverse image of each (open) bounded connected component of $\C\backslash K$ is bounded, therefore included in one of the bounded component of $\ f^{-1}(K)$. Therefore $\Omega_{f^{-1}(K)}\subset f^{-1}(\Omega_K)$, then \[f(\Omega_{f^{-1}(K)})\subset\Omega_K\]
	\[f(\partial\Omega_{f^{-1}(K)})\subset\overline{\Omega_K}.\]
	We finally obtain that
	\[f(\partial\Omega_{f^{-1}(K)})\subset\overline{\Omega_{K}}\cap\partial K=\partial\Omega_K,\] which allows us to apply theorem \ref{thmAppliHolom} (b) to obtain the final result.
\end{proof}

\subsection{Chebyshev constant}\label{ssec-constcheb}

\paragraph{}
Let us now give a third equivalent definition of the capacity, defined thanks to Chebyshev polynomials. This point of view is very useful to calculate the capacity of segments of $\R$.

\medskip

First, let us state the equioscillation theorem : 

\begin{thm}[Equioscillation]\label{thm-equioscillation}
Let us consider $n\in\N$. Let $f$ be a continuous function on $[a,b]$. Let us consider $P\in\R_n[X]$. $P$ minimizes $||f-P||_{\infty,[a,b]}$ if and only if there exists $n+2$ points $a\leqslant x_0 < x_1 < \dots x_{n+1} \leqslant b$ such that $f(x_i)-P(x_i)=\pm (-1)^i||f-P||_{\infty,[a,b]}$
\end{thm}

The proof of this theorem is not that hard but is quite long. In order not to overfill this paper, we advise the reader to read article \cite{equioscillation} for a complete and illustrated proof.

\begin{defi}[The Chebyshev constant]
    Let us consider some $n\in\N$ and $K$ a compact subset of $\C$, let us denote by $||\cdot||_K$ the uniform norm on $K$ and :
    $$t_n(K) = \displaystyle\inf_{P\in\C_{n-1} [X]} || X^n+P||_K.$$
    If $K$ contains an infinite number of points, there exists a unique monic polynomial $T_n\in\R_n[X]$, called Chebyshev polynomial such that $t_n(K)=||T_n||_K$. \\
    
    Let us define the Chebyshev constant as follows :
    $$\cheb(K)=\lim_{n\rightarrow+\infty} t_n(K)^{\frac 1 n}.$$
    \end{defi}
    
\begin{proof}
    \begin{itemize}
        \item If $K$ is an infinite set, $||\cdot||_K$ is a norm on $\C[X]$. The existence of $T_n$ follows from the fact that the distance to a closed vector subspace is reached. The uniqueness follows from the reciprocal implication of the equioscillation theorem (\ref{thm-equioscillation}).

        \item Let us consider $n,m\in \N$. Since $T_nT_m$ is a monic polynomial of degree $n+m$, we have \[||T_{n+m}||_K=\displaystyle\inf_{\substack{P\in \C_{n+m}[X] \\ \text{monic}}} ||P||_K \leqslant ||T_nT_m||_K \leqslant ||T_n||_K||T_m||_K\] 
        Therefore $t_{n+m}(K)\leqslant t_n(K)t_m(K)$ and $(\log(t_n(K)))_{n\in \N}$ is sub-additive. It follows from the sub-additivity lemma that $ (\frac 1 n \log(t_n(K)))_{n\in \N}$ converges, which proves that $\cheb (K)$ is well defined.
    \end{itemize}
\end{proof}

The Chebyshev constant is equivalent to the logarithmic capacity and the transfinite diameter :

\begin{thm}\label{chebEgalCapaEgalTau}
    Let us consider $K$ a compact subset of $\C$, then 
    \[\mathrm{\tau}(K)=\capa (K) = \cheb (K).\]
    Let $\mathcal{F}_n=\{z_1^{(n)},\dots,z_n^{(n)}\}$ be a set of $n$ Fekete points and $F_n(X)=\displaystyle\prod_{i=1}^n (X-z_n^{(i)})$ the associated Fekete polynomial. We have
    \[\lim_{n\rightarrow + \infty} ||F_n||_K^{1 / n}=\cheb (K).\]
    Moreover, if $\capa(K)>0$, noting $\mu_K$ the equilibrium measure of $K$, we have 
    \[\forall z\in\C\backslash K, \lim_{n\to+\infty}F_n(z)^{\frac 1n} = \exp(-U^{\mu_K}(z))\] The convergence is uniform on all compact subsets of $\C\backslash K$.

\end{thm}

\begin{proof}
     We argue using theorem \ref{thmTauEgalCapa}.
    \begin{itemize}
        \item Let us consider $n\in\N$, $z\in \C$ and $\mathcal{F}_n=\{z_1^{(n)},\dots,z_n^{(n)}\}$ a set of $n$ Fekete points. \\
        Let us consider $\{y_1,\dots,y_{n+1}\}=\{z_1^{(n)},\dots,z_n^{(n)},z\}$, then
        \[\delta_{n+1}(K)^{n(n+1)/2} = \max_{\{y_1,\dots,y_{n+1}\}\subset K} \prod_{1\leqslant i< j \leqslant n+1} |y_i-y_j| \geqslant \prod_{i=1}^n |z-z_i^{(n)}| \prod_{1\leqslant i< j \leqslant n} |z_i^{(n)}-z_j^{(n)}| \]
         Let $F_n=\prod_{i=1}^n(X-z^{(n)}_i)$. Since $\mathcal{F}_n$ is a set of Fekete points : 
         \[\delta_{n+1}(K)^{n(n+1)/2}\geqslant |F_n(z)| \delta_n(K)^{(n-1)n/2}.\]
         By considering $\sup_{z\in K}$, we obtain $||F_n||_K\leqslant \left(\frac{\delta_{n+1}(K)}{\delta_n(K)}\right)^{(n-1)n/2}\delta_{n+1}(
         K)^n$. However, the $\delta_n(K)$ are decreasing, therefore $||F_n||_K^{1/n}\leqslant \delta_{n+1}(K)$. Moreover by definition, $t_n(K)\leqslant ||F_n||_K^{1/n}\leqslant \delta_{n+1}(K)$. Finally, by considering the limit, we have
         \[\tau(K)\geqslant \limsup_{n\rightarrow \infty} ||F_n||_K^{1/n} \geqslant \liminf_{n\rightarrow \infty} ||F_n||_K^{1/n} \geqslant \cheb (K).\]
         
         \item To conclude, it is enough to show that $\cheb (K) \geqslant \mathrm{\tau (K)}$. If $\tau (K)=0$, it is trivial, otherwise $K$ contains an infinite number of points and we can consider $T_n=\displaystyle\prod_{i=1}^n (X-x_i)$, the $n$-th Chebyshev polynomial with respect to $K$ and  $\nu(T_n)$ the counting measure with respect to $x_1,\dots,x_n$. 
         \\Since $x\in\R_+^*\mapsto\log\frac 1 x$ is a decreasing function, 
         \[\frac 1 n \log\frac 1 {t_n(K)} = \inf_{z\in K} \frac 1 n \log\frac 1 {|T_n(z)|} = \inf_{z\in K} \frac 1 n \sum_{i=1}^n \log\frac 1 {|z-x_i|} = \inf_{z\in K} U^{\nu(T_n)}(z).\]
         It follows from prop \ref{propIntegralSciPositive} that:
         \[\inf_{z\in K} U^{\nu(T_n)}(z)=\mu_K(\inf_{z\in K} U^{\nu(T_n)}(z))\leqslant \mu_K( U^{\nu(T_n)}).\]
         Then,
         \[\mu_K( U^{\nu(T_n)})= \mu_K(\frac 1 n \sum_{i=1}^n \log\frac 1 {|z-x_i|})=\frac 1 n \sum_{i=1}^n \mu_K(\log\frac 1 {|z-x_i|}) = \nu(T_n)(U^{\mu_K})\]
         Frostman's theorem \ref{thmFrostman} gives us that  
         \[\frac 1 n \log\frac 1 {t_n(K)}\leqslant \nu(T_n)(U^{\mu_K})\leqslant V_K.\]
         Finally, passing to the limit $n\rightarrow +\infty$ and composing by $x\mapsto e^{-x}$, we obtain :
         \[ \cheb (K)\geqslant\capa (K) = \tau (K),\]
         which concludes the first claim of the theorem.
         \item Let $f_n(z)=\log(|F_n(z)|^{1/n}) = \frac{1}{n}\sum_{k=1}^n \log|z-t_k|$ where $t_1,\dots,t_n$ are the roots of $F_n$. Let $L$ be a compact subset of $\C - K$, then :
         \[\forall t\in K, \forall z \in L, \log d(L,K)\leqslant \log|z-t|  \leqslant \log \sup_{z'\in L, t'\in K} |z'-t'|.\]
         However, by compactness of $L$ and $K$, $d(L,K)\geqslant 0$ and $\sup_{z'\in L, t'\in K} |z'-t'| < \infty$. \newline
         Moreover, we have
         \begin{align*}
         \forall n\in \N, \forall z\in L , \,|f_n(z)| &\leqslant  \frac{1}{n}\sum_{k=1}^n \left|\log\left(\frac 1 {d(L,K)}+ \sup_{z'\in L, t'\in K} |z'-t'|\right)\right| \\ 
         &\leqslant \left|\log\left(\frac 1 {d(L,K)}+ \sup_{z'\in L, t'\in K} |z'-t'|\right)\right|<\infty.
         \end{align*}
         therefore the $f_n$ are uniformly bounded on $L$.
         \medskip
         Therefore, $f_n$ is a sequence of holomorphic functions, uniformly bounded on all compact sets, which converges pointwise on $\C \backslash K$. It follows from Vitali theorem that the convergence is uniform on all compact sets.
         
    \end{itemize}
\end{proof}

This new definition of capacity allows us to calculate the capacity of a segment in a simple manner.

\begin{ex}[Capacity of a segment]\label{capasegment}
The capacity of a segment with length $2l$ is equal to $\frac{l}{2}$.
\end{ex}

\begin{proof}
\begin{itemize}

\item[$\bullet$] Let us first prove the claim when the segment can be written as  $[-l,l]$.\\
Let us denote by $\tilde{T}_{n}$ the unique polynomial verifying:
\[ \forall \theta \in \mathbb{R}, \tilde{T}_n(l\cos(\theta))=lcos(n\theta) \] 
The uniform norm of $\tilde{T}_n$ is obviously equal to $l$. And like in the well-known case $[-1,1]$, we show easily that our polynomial reaches alternatively  $\pm l $ at $n+1$ distinct points. (1) \\
Using De Moivre's formula, we show that the leading coefficient of $\tilde{T}_n$ is equal to $c:=(\frac{2}{l})^{n-1}$. \\
Minimizing $\Vert P \Vert_{[-l,l]}$, for $P\in \C_n[X]$ a monic polynomial, is equivalent to minimizing $\Vert  X^n - Q \Vert_{[-l,l]}$ for $Q$ living in $\mathbb{C}_{n-1}[X]$. In other words, we are trying to calculate the distance of $X^n$ to $\mathbb{C}_{n-1}[X]$. Let us consider $Q^{\star}:=X^n-\frac{1}{c}\tilde{T}_n(X)$. We obviously have that $Q^{\star} \in \mathbb{C}_{n-1}[X]$. Moreover, $X^n-Q^{\star}=\frac{1}{c}\tilde{T}_n(X)$, which \textit{equioscillates} at $n+1$ points according to (1). Therefore, it follows from lemma 2.1 that $Q^{\star}$ minimizes the distance of $X^n$ from polynomials of $\mathbb{C}_{n-1}[X]$. Therefore $\frac 1 c \tilde{T}_n$ is the Chebyshev polynomial with respect to $[-l,l]$.\\
Hence $t_n([-l,l])=\Vert \frac{1}{c} \tilde{T}_n(X) \Vert = \frac{l}{c}= l (\frac{l}{2})^{n-1}$.\\
Therefore, $\cheb ([-l,l])=\lim\limits_{n\rightarrow \infty} t_n([-l,l])^{\frac{1}{n}}=\frac{l}{2}$.

\vspace{1em}
\item[$\bullet$] In the general case of a segment of the form $[a,b]$, with length $2l$, it is enough to consider the following polynomial function: $\tilde{T}_n(X-(a+l))$ which, on $[a,b]$, has the same behavior as $\tilde{T}_n$ on $[-l,l]$. Therefore, the capacity is not changed by the translation. Hence $\cheb ([a,b])=\frac{l}{2}$.
\end{itemize}
\end{proof}
\vspace{2em}

\paragraph{}
Let us now calculate the capacity of the union of two segments, which is symmetric with respect to the origin $0$.
\begin{ex}
    Let us consider $0\leqslant b<a$. Then $\cheb ([-a,-b]\cup[b,a])=\frac{\sqrt{a^2-b^2}} 2$.
\end{ex}

\begin{proof}
Let us consider $F$ an infinite compact subset of $\C$ and denote by $T_{n}^{F}$ the Chebyshev polynomial of degree $n$ with respect to $F$. Let us consider $E = [-a,-b]\cup[b,a]$.
    \begin{itemize}
        \item Let us consider $ k\in \N$. Let us show that $T_{2k}^{E}$ is even.
        \[\left\Vert \frac{T_{2k}^{E}(X)+T_{2k}^{E}(-X)}{2} \right\Vert_E\leqslant \frac 1 2 (\left\Vert T_{2k}^{E}(X) \right\Vert_E+\left\Vert T_{2k}^{E}(-X) \right\Vert_E)\]
        But $E$ is symmetric with respect to the origin, therefore 
        \[\left\Vert T_{2k}^{E}(-X) \right\Vert_E = \sup_{z\in E} \left|T_{2k}^{E}(-z)\right|=\sup_{z\in E} \left|T_{2k}^{E}(z)\right|=\left\Vert T_{2k}^{E}(-X) \right\Vert_E.\]
        Thus $\left\Vert \frac{T_{2k}^{E}(X)+T_{2k}^{E}(-X)}{2} \right\Vert_E \leqslant \left\Vert T_{2k}^{E} \right\Vert_E$ and by the uniqueness of the Chebyshev polynomial of degree $2k$, $ T_{2k}^{E}= \frac{T_{2k}^{E}(X)+T_{2k}^{E}(-X)}{2}$. 
        Therefore $T_{2k}^{E}$ is even.
        \item We deduce from it that $T_{2k}^{E}$ can be written as follows: $q(X^2)$ with $q\in \mathcal{P}_k$, the set of monic polynomials of degree $k$. As a result :
        \begin{align*}
            t_{2k} (E) &= \inf_{q\in \mathcal{P}_k} \sup_{x\in E} |q(x^2)| \\
            &= \inf_{q\in \mathcal{P}_k} \sup_{x^2\in [b^2,a^2]} |q(x^2)|\\
            &= \inf_{q\in \mathcal{P}_k} \sup_{y\in [b^2,a^2]} |q(y)|.
        \end{align*} 
        Therefore $T_{2k}^{E} = T_k^{[b^2,a^2]}(X^2)$. But in the example \ref{capasegment} 
        we saw that $(\frac 2 {l})^{k-1} T_k^{[-l,l]}(l\cos(\theta))=l\cos(k\theta)$. 
        Therefore, by considering $x=l\cos(\theta)$, we have :$$T_k^{[-l,l]}(x) = \frac{l^k}{2^{k-1}}\cos\left(k\arccos\left(\frac x l\right)\right).$$ Then by translation we obtain :
        \[T_k^{[b^2,a^2]}(x) = \frac{(a^2-b^2)^k}{2^{2k-1}} \cos \left(k\arccos \left(2\frac {x-b^2}{a^2-b^2}-1\right)\right).\]
        And finally $T_{2k}^{E}(x) = \frac{(a^2-b^2)^k}{2^{2k-1}} \cos \left(k\arccos \left(2\frac {x^2-b^2}{a^2-b^2}-1\right)\right)$.
        \item $t_{2k}(E)=\left\Vert T_{2k}^{E}\right\Vert_E  = 2 \left(\frac{\sqrt{a^2-b^2}}{2}\right)^{2k}$. Hence 
        \[\cheb (E) = \lim_{k\rightarrow\infty }t_{2k}(E)^{\frac 1 {2k}}=\lim_{k\rightarrow\infty } \left(2 \left(\frac{\sqrt{a^2-b^2}}{2}\right)^{2k}\right)^{\frac 1 {2k}} = \frac{\sqrt{a^2-b^2}}{2} .\]
    \end{itemize}
\end{proof}

\clearpage

\section{Fekete's theorem. Fekete-Szegö's theorem}\label{sectionthFEk}

\paragraph{}
Now that we have introduced the notion of capacity and showed the equivalence between the three definitions (transfinite diameter/potential/Chebyshev polynomials), let us now focus on how this notion gives us information on the fact that there is (or not) an infinite number of algebraic integers totally in a given compact set. The goal of this section is to prove Fekete's theorem (Thm. \ref{thmFekete}) and Fekete-Szegö's theorem (Thm. \ref{thmFeketeSzego}) : the former states that if the capacity of a compact set is strictly smaller than 1, there is a finite number of algebraic integers totally in it; the latter states that if the capacity of the compact set is greater than or equal to 1, we can find an infinite number of algebraic integers totally in any neighborhood (with respect to $\C$) of the compact set.

\medskip

\subsection{Fekete's theorem}

This section is dedicated to the proof of Fekete's theorem. 

\begin{thm}[Fekete]\label{thmFekete}
    Let $K\subset\C$ be a compact set with $\capa(K)<1$. Then there exists an open neighborhood $U$ of $K$ such that the set of the algebraic integers totally in $U$ is finite. In particular, there is a finite number of algebraic integers totally in $K$.
\end{thm}
    The proof mainly comes from the fact that $\tau(K)=\capa(K)$.
\begin{proof}
    It is enough to prove the last claim of the theorem; in fact, since $\capa(K)<1$, we have $\capa(K^\epsilon)<1$ for some $\epsilon>0$ sufficiently small, where $K^\epsilon=\{z\in\C:\dist(z,K)\leq\epsilon\}$; for such a $\epsilon$, $U=\overset{\circ}{K^\epsilon}$ verifies the assumption of the theorem. Assume that $E_K$, the set of algebraic integers totally in $K$, is infinite and let us argue by contradiction.
    \par\medskip Let us start with the case where the degrees of all the minimal polynomials $p$ of $\zeta\in E_K$ are bounded by an integer $N>0$. Then the coefficient before $X^{\deg(p)-i}$ is bounded by
    \[\binom{\deg(p)}{i}\|K\|^i_\infty\leq\binom{N}{i}\|K\|^i_\infty,\quad\textrm{where }\|K\|_\infty=\max_{z\in K}\abs{z}.\]
    There is only a finite number of such minimal polynomials (we remind that the coefficients are integer), hence $\#E_K<\infty$, which contradicts our initial assumption. 
    
    \par\medskip Otherwise, there exists a sequence $\zeta_n\in E_K$ such that the sequence $d_n=\deg(p_n)$ is strictly increasing, where $p_n\in\Z[X]$ is the (monic) minimal polynomial of $\zeta_n$.

    Let us remind ourselves of the notion of $n$-distance (Definition~\ref{defi-PointsdeFekete}) : 
    \[\delta_n(K)=\max_{z_1,\dots,z_n\in K} \prod_{i < j}\abs{z_i-z_j}^{2/n(n-1)}\]
    And the capacity as the transfinite diameter : 
    \[\tau(K) = \lim_{n\to+\infty}\delta_n(K)\]
   
    Let us notice that
    \[\delta_{d_n}(K)^{d_n(d_n-1)}\geq\abs{\prod_{\xi\neq\eta\in\gal(\zeta_n)}(\xi-\eta)}=\abs{\prod_{\xi\in\gal(\zeta_n)}p_n'(\xi)}\geqslant1\]
    
    The equality in the middle results from the fact that $p_n$ only has simple roots. Moreover, as a symmetric combination with integer coefficients of the $\xi$ (the roots of $p_n$), $\prod_{\xi\in\gal(\zeta_n)}p_n'(\xi)$ is an integer, according to remark \ref{rmqCombineSymmEntier}, non-zero since the roots are simple, therefore greater than 1. 
    
    Therefore $\delta_{d_n}(K)\geq1$, hence $\tau(K)\geq 1$ by passing to the limit $n\to+\infty$, which leads to a contradiction because $\tau(K)=\capa(K)<1$.
\end{proof}

\subsection{Fekete-Szegö's theorem}
\paragraph{}
This section is dedicated to the proof of Fekete-Szegö's theorem \ref{thmFeketeSzego}. Let us first define a set which plays a major part in the proof, the \emph{Hilbert lemniscate} : 

\begin{defi}
 Let $p\in\C[X]$ be a \emph{monic} polynomial of degree $d>0$ and $\rho\in\R_{+}^{\star}$ a real constant. The \emph{lemniscate of polynomial $p$ and constant $\rho$} is the following set
    \[L=L_{p,\rho}=p^{-1}\big(\overline{B}(0,\rho^d)\big)=\{z\in\C | \abs{p(z)}\leq\rho^d\}.\]
\end{defi}

\begin{rmq}
     According to corollary \ref{corCapaciteImageRecipPolynome}, we have $\capa(L_{p,\rho})=\capa\big(\overline{B}(0,\rho^d)\big)^{1/d}$, which is equal to $\rho$ thanks to corollary \ref{capaDisqueUnite} and theorem \ref{thmTauEgalCapa}.
\end{rmq}

The idea of the proof of Fekete-Szegö's theorem is to find an infinite number of algebraic integers totally in a set. Hilbert lemniscates provide us with examples of such sets : 

\begin{prop}\label{prop-lemniscateContientEAT}
Let $P\in\Z[X]$ be a monic polynomial. The lemniscate $L_{P,1}=\{z\in\C\  | \ \abs{P(z)}\leqslant 1\}$ contains an infinite number of algebraic integers totally in it.
\end{prop}

\begin{proof}
Let us consider $S=\bigcup_{n\in\N^*}\{z\in\C\ |\ P(z)^n=1\}$. The function $
   \left \{
   \begin{array}{r c l}
      S  & \rightarrow & \U \\
      z  & \mapsto & P(z) \\
   \end{array}
   \right .
$
is well defined, surjective since $\C$ is algebraically closed. Therefore $S$ is infinite. Moreover, the elements of $S$ are algebraic integers totally in $S$ (since they are the roots of $P^n-1\in\Z[X]$ monic, for some $n$). We conclude by noticing that $S\subset L_{P,1}$.

\end{proof}

The main step of the proof is the following theorem which states that, under some assumptions, we can always find a Hilbert lemniscate in any neighborhood of the compact set.

\begin{thm}[Hilbert's lemniscate]\label{thmLemniscHilbert}
    Let $K\subset\C$ be a compact set with $\capa(K)>0$ and $U\supset K$ an open neighborhood of $K$ such that $\C\backslash U$ is connected. Then there exists a monic polynomial $p\in\C[X]$ of degree $d>0$ and a constant $\rho>\capa(K)$ such that
    \[K\subset L_{p,\rho}\subset U.\]
\end{thm}

\begin{proof}
    This proof relies on theorem \ref{chebEgalCapaEgalTau} (Chebyshev constant) and theorem \ref{thmFrostman} (Frostman).
    
    Even if it means restricting $U$, we can assume that $U$ is bounded. Let us consider $R>0$ such that 
    \[K\subset U\subset B(0,R)\]
    It follows from Frostman's theorem that $U^{\mu_K}(z)<V_K$ for all $z\in \overline{B}(0,R)\backslash U$. Therefore there exists $\varepsilon>0$ such that
    \[\forall z\in \overline{B}(0,R)\backslash U, \  U^{\mu_K}(z)\leqslant V_K-\varepsilon\]
    This upper bound is uniform since $\overline{B}(0,R)\backslash U$ is compact.

    According to the last claim of theorem \ref{chebEgalCapaEgalTau}, $\log\frac{1}{F_n^{\frac{1}{n}}}$ uniformly converges to $U^{\mu_K}$ : there exists $n_0\in \N$ such that for all $n\geqslant n_0$, 
    
    \[\forall z\in \overline{B}(0,R)\backslash U, \ 
    \abs{
    \log
    \frac{1}{ \abs{F_n(z)}^{\frac 1n} }-U^{\mu_K}(z)}<\varepsilon/2
    \]
  
    which implies that 
    \[\forall z\in \overline{B}(0,R)\backslash U,   \log
    \frac{1}{ \abs{F_n(z)}^{\frac 1n}} < V_K -\varepsilon/2 \]
    
    in other words,
    
    \[\forall z\in \overline{B}(0,R)\backslash U,\ 
    \abs{F_n(z)}^{\frac 1n} > e^{-V_K+\varepsilon/2} = \capa(K) e^{\varepsilon/2}
    \]

    But $z\mapsto \frac{1}{ F_n(z)}$ is holomorphic on $\C \backslash U$, therefore it follows from the maximum modulus principle that the previous inequality holds for all $z\in \C\backslash U$.
    Let us consider $\rho :=  \capa(K) e^{\varepsilon/2} $, we then have $L_{F_n,\rho}\subset U$ for all $n\geqslant n_0$.
    \medskip
     On the other hand, according to theorem \ref{chebEgalCapaEgalTau}, there exists $n_1\in \N$ such that for all $n\geqslant n_1$,
     \[
     ||F_n||_K^{\frac 1n} < \capa(K)e^{\varepsilon/2} = \rho
     \]
     By considering $d =\max(n_0,n_1)$, we have $K\subset L_{F_d,\rho}\subset U$.
     
\end{proof}

Let us now prove Fekete-Szegö's theorem :

\begin{thm}[Fekete-Szeg\"o]\label{thmFeketeSzego}
    Let $K\subset\C$ be a compact set that is symmetric with respect to complex conjugation and such that $\capa(K)\geq1$. If $U$ is an open set containing $K$ such that $C\backslash U$ is connected, then $U$ contains an infinite number of algebraic integers totally in $U$.
\end{thm}

\begin{proof}
   The idea is to find a known set included in $U$, which contains an infinite number of algebraic integers totally in it. According to proposition \ref{prop-lemniscateContientEAT}, the lemniscate of a monic polynomial with integer coefficients seems to be a good contender. The theorem \ref{thmLemniscHilbert} gives us a lemniscate $L_{p,\rho}$ included in $U$, but for a polynomial $P$ with complex coefficients (and $\rho>1$). We must now bring its coefficients to $\Z$.
   
   \medskip
   
   Even if it means considering $U\cup U^{\star}$ (where $U^\star$ is the set of the complex conjugates of the elements of $U$), we can assume that $U$ is symmetric with respect to the real axis. Under this assumption, we have 
   \[\forall z\in \C\backslash U,\  \abs{p(z)}>\rho^d>1 ,\]
   then
    \[\forall z\in \C\backslash U,\  \abs{p(z)\,\overline{p}(z)}>\rho^{2d}>1 ,\]
   Therefore $L_{p\,\overline{p},\rho}\subset U$ and $p\,\overline{p}\in\R[X]$. We keep on denoting $p$ this polynomial with real coefficients.
   
   \medskip
   
   Then assume that $U\subset \overline{B}(0,R)$, even if it means restricting the open set $U$. We have 
      \[\forall z\in \overline{B}(0,R) \backslash U,\  \abs{\frac{1}{p(z)}}<\rho^{-d}.\]
   It follows from the density of $\Q$ in $\R$ and the continuity of the roots (since $\abs{z}\leqslant R$) that we can choose a polynomial $q\in\Q[X]$ whose roots are still all in $U$ and such that
     \[\forall z\in \overline{B}(0,R) \backslash U,\  \abs{\frac{1}{q(z)}}<\rho^{-d}.\]
    And since $\frac 1q$ is holomorphic on $\C \backslash U$, the maximum modulus principle gives: 
 \[\forall z\in \C \backslash U,\  \abs{\frac{1}{q(z)}}<\rho^{-d},\]
   which implies that $L_{q,\rho}\subset U$.
   
   \medskip
   
   We now have to bring the coefficients of $q$ to $\Z$, which is the goal of the two following lemmas (\ref{lmmlmmFeketeSzego}, \ref{lmmFeketeSzego}). 
\end{proof}

\begin{lmm}\label{lmmlmmFeketeSzego}
    Let us consider $p\in\Q[X]$ of degree $d\geq1$ and $n\in\N^*$ such that \[p(X)=X^d+\frac{1}{n}\gamma(x),\quad\gamma(X)\in\Z[X].\]
    Let us consider $\mu\in\N^*$, $\sigma=\mu\,d$ and $\nu=\sigma!\,n^\sigma$. Then there exists a monic polynomial $\Gamma(X)\in\Z[X]$ of degree $\nu\,d=\deg(p^\nu)$ such that $r(X):=p^\nu(X)-\Gamma(X)$ is a polynomial which can be written as follows:
    \[r(X)=\sum_{l=0}^{\nu-\mu-1}p^{l}(X)q_l(X)\]
    where the $(q_l)_{0\leq l\leq\nu-\mu-1}$ are polynomials of degree not greater than $d-1$, with coefficients in $\Q\cap[0,1[$.
\end{lmm}

\begin{proof}
    Let us consider the following decomposition:
    \[p^\nu(X)=E(X)+R(X)\]
    where
    \[E(X)=\sum_{i=0}^\sigma\binom{\nu}{i}\frac{1}{n^i}X^{(\nu-i)d}\gamma^{i}(X)\quad\textrm{(monic)}\]
    \[R(X)=\sum_{i=\sigma+1}^\nu \binom{\nu}{i}\frac{1}{n^i}X^{(\nu-i)d}\gamma^{i}(X).\]
    
   Let us study the coefficients of $E(X)$: for $0\leq i\leq\sigma$, since $i!\,n^i$ divides $\sigma!\,n^\sigma=\nu$, then $\binom{\nu}{i}\frac{1}{n^i}\in\Z$. Therefore $E(X)\in\Z[X]$.

   Let us now take a look at $R(X)$: for $i\geq\sigma+1$, we have
    $\deg({X^{(\nu-i)d}\gamma^i(X)})\leq (\nu-i)d+i(d-1)\leq\nu d-\sigma-1=(\nu-\mu)d-1$, therefore $\deg(R)\leq(\nu-\mu)d-1$.
    
    Since $\{p^l(X)X^k\ |\ 0\leq l\leq\nu-\mu-1,0\leq k\leq d-1\}$ forms a basis of $\Q_{d(\nu-\mu)-1}[X]$, $R(X)$ is a  $\Q$-linear combination of this basis. Moreover, $p^l(X)X^k$ are monic polynomials which are, pairwise, of different degrees, therefore we can find  by induction $c_{l,k}\in\Q\cap[0,1[$ in the order $k=d-1,\ldots,0$ and $l=\nu-\mu-1,\ldots,0$, such that
    \[p^\nu(X)-\sum^{l,k}c_{l,k}p^l(X)X^k\in\Z[X].\]
    Let us finally consider
    \[q_l=\sum_{k=0}^{d-1}c_{l,k}X^k\]
    \[r(X)=\sum_{l=0}^{\nu-\mu-1}p^{l}(X)q_l(X).\]
    Then $\Gamma(X):=p^\nu(X)-r(X)\in\Z[X]$ is the monic polynomial we were looking for.

\end{proof}

\begin{lmm}\label{lmmFeketeSzego}
    Let $L$ be a lemniscate of polynomial $p\in\Q[X]$ of degree $d\geq1$ and constant $\rho>1$, then there exists a lemniscate of polynomial $\Gamma\in\Z[X]$ and constant $1$ included in $L$.
\end{lmm}

\begin{proof}
    It is obvious that $\partial L =\{z\in\C \ | \ \abs{p(z)}=\rho^d\}$.
    
    Let us consider $M=\sup_{z\in\partial L}(1+\abs{z}+\cdots+\abs{z}^{d-1})$, 
    $\mu\in\N^*$ such that $\frac{M}{\rho^{\mu d}(\rho^d-1)}\leq1/2$ and $\frac{\rho^{\mu d}}{2}>1$. Let $\Gamma$ be the polynomial in the previous lemma, which corresponds to our choice of $\mu$, and let us use the same notations as in the lemma.
    
    \medskip

    For all $z\in\partial L$, we have
    \[
    \frac{\abs{p^{\nu}(z)-\Gamma(z)}}{\abs{p^{\nu}(z)}}
    =
    \frac{\abs{r(z)}}{\rho^{\nu d}}
    =
    \frac{\displaystyle\abs{\sum_{l=0}^{\nu-\mu-1}p^{l}(z)q_l(z)}}{\rho^{\nu d}}
    \leqslant
    \frac{M}{\rho^{\nu d}}\sum_{l=0}^{\nu-\mu-1}\rho^{ld}\leqslant\frac{M(\rho^{(\nu-\mu)d}-1)}{\rho^{\nu d}(\rho^d-1)}\leq\frac{M}{\rho^{\mu d}(\rho^d-1)}\leq\frac{1}{2}.\]

    It follows from Rouché's theorem that $\Gamma^{-1}(0)\subset L$. On the other hand, this inequality implies that for all $z\in\partial L$
    \[\abs{\Gamma(z)}\geq\abs{p^\nu(z)}-\abs{r(z)}\geq\frac{1}{2}\abs{p^\nu(z)}\geq\frac{\rho^{\mu d}}{2}>1.\]
    
    But $\frac{1}{\Gamma}$ is holomorphic on $\overline{C}\backslash L$ because $\Gamma^{-1}(\{0\})\subset L$, therefore, according to the maximum modulus principle, $\abs{\Gamma(z)}>1$ for all $z\notin L$.
    
    Therefore $L_{\Gamma,1}\subset L$.
\end{proof}

\clearpage

\section{Robinson's theorem}\label{section3}

\paragraph{}
In this section, we shall state and prove the main theorem of this article: Robinson's theorem.
Let us consider some segment $K$ of $\R$. The second section of this piece (\ref{capasegment}) allows us to calculate its capacity. If it is strictly smaller than 1, Fekete's theorem \ref{thmFekete} gives us that the number of algebraic integers totally in $K$ is finite. On the contrary, when the capacity is greater than or equal to 1, Fekete-Szeg\"o's theorem \ref{thmFeketeSzego} gives us that there is an infinite number of algebraic integers arbitrarily close to $K$ with respect to $\C$. However, when $K$ is a segment of $\R$, this property is not enough to show that there is an infinite number of algebraic integers totally in $K$. Hence the potential theory is not enough to fully understand the case of compact subsets of $\R$. To prove Robinson's theorem, we will have to use algebraic curves.
\medskip

In section \ref{section-courbes algébriques}, we will first define a few notions on algebraic curves: regular/rational functions on a curve, group of divisors... Then we will build a smooth completion of a hyperelliptic curve. Finally, section \ref{section-robinson} is dedicated to the proof of Robinson's theorem. We will cleverly use hyperelliptic curves to prove a theorem which, at first sight, has nothing to do with them. 

\medskip
We shall also use charts, holomorphic and meromorphic forms on a Riemann surface. An introduction to this notions is provided in \cite{Bost1992}, in particular in sections I.1 and B.2.

\subsection{A few notions on algebraic curves}
\label{section-courbes algébriques}
\paragraph{}
An \emph{algebraic curve} is an object of dimension 1 locally defined by an \emph{algebraic} (\emph{i.e.} polynomial) equation. Since all polynomials are analytic functions, the algebraic approach gives us less information than the analytic approach.
\medskip

We assume that the field $\mathbb{K}$ is $\C$, but we shall use $\mathbb{K}$ or $\C$ according to the approach (algebraic or analytic). Several definitions or results where the field is referred as $\mathbb{K}$ can be generalized, but we will not focus on this here.

\medskip

We will give a visual example at the end of section \ref{sectionExVisuel}, which can be useful to understand the paragraphs that precede.

\clearpage

\subsubsection{Algebraic curve: affine case}

\begin{defi}
    Let $F\in \mathbb{K}[X,Y]$ be a non-constant polynomial in two variables which is \emph{irreducible}, in other words which cannot be written as the product of two non-constant polynomials of $\mathbb{K}[X,Y]$.
    An \emph{algebraic curve} $C$ in the affine space $\mathbb{A}^2_\mathbb{K}$ (which is isomorphic to $\mathbb{K}^2$) is the set of all the points $(x,y)\in K^2$ verifying the equation $F(x,y)=0$. It is represented as follows:
    \[C: F(x,y)=0.\]
\end{defi}
\vspace{2em}

$\bullet$ \textbf{Algebraic approach}:

\medskip \emph{The ring of regular functions} on curve $C$ is by definition the ring
\[A(C) := \mathbb{K}[X,Y]/(F).\]

\medskip

In other words, it is the ring of polynomials \emph{modulo an equivalence relation}, for which two polynomials are equivalent if and only if they are equal on the points of curve $C$. $A(C)$ is an \emph{integral domain} since $F$ is \emph{irreducible}; its field of fractions is called \emph{the field of rational functions} of the curve and is referred as $\mathscr{R}(C)$.

\medskip $\bullet$ \textbf{Analytic approach} (more intuitive):

\medskip It follows from the implicit function theorem that, if $\nabla F=\left(\partial_xF,\partial_yF\right)\neq\mathbf{0}$ at a point $(x_0,y_0)\in C(\C)$, then in a neighborhood of $(x_0,y_0)$, the curve is locally the graph of some holomorphic function $y=y(x)$ or $x=x(y)$.
The curve has a good behavior in the neighborhood of such a point $(x_0,y_0)$ : we say that $(x_0,y_0)$ is a \emph{smooth point}.
More precisely, the curve has a 1-dimensional complex manifold structure in the neighborhood of $(x_0,y_0)$. That is why it is called a curve (1-dimensional). On the contrary, some point of $C$ with $\nabla F=\mathbf{0}$ is called a \emph{singular point}. A curve without any singular point is called \emph{smooth}. A smooth curve has the natural structure of a Riemann surface.

\medskip
Let us consider (affine) hyperelliptic curves, whose definition is given in a restricted manner in order to fit our problem.
\begin{defi}
    An \emph{hyperelliptic curve} is an algebraic curve
    \[C: y^2=D(x)\]
    where $D\in\C[X]$ is a \emph{monic} polynomial in $X$ of degree $2g+2$ ($g\geq 0$) whose roots are \emph{distinct}.
    We write that $C$ is a \emph{real} hyperelliptic curve if $D\in\R[X]$.
\end{defi}

With this definition, we check that a hyperelliptic curve is \emph{smooth}.

\subsubsection{Divisors}
\paragraph{}
Let us consider $F(X,Y)\in\C[X,Y]$ non constant and irreducible, as well as the complex algebraic curve associated with $C$. The \emph{group of divisors} on $C$, referred as $\Div(C)$, is the free abelian group generated by the basis $\{P\}_{P\in C}$. More precisely, a \emph{divisor} $D$ on $C$ is a formal linear combination with integer coefficients of a finite number of points $P\in C$:
\[D=\sum_{i=1}^nn_iP_i, \quad\textrm{where~} n\in\N,n_i\in\Z,P_i\in C.\]

Different points are considered linearly independent. The group of divisors $\Div(C)$ refers to the set of divisors on $C$ equipped with a structure of group with the addition of coefficients (point by point). We have a natural group homomorphism called \emph{degree of a divisor} :
\[\deg:\;\Div(C)\longrightarrow\Z,\quad\sum n_iP_i\longmapsto\sum n_i.\]

Assume curve $C$ is \emph{smooth}. According to the implicit function theorem, $C$ has a 1-dimensional complex manifold structure. Let us consider $f\in\mathscr{R}(C)^\times$. $f$ can be seen as a \emph{meromorphic function} on $C$ and therefore, $\ord_P(f)$ is well defined for all $P\in C$ as the order of the meromorphic function $f$ at point $P$. Let us define the \emph{principal divisor} associated with $f\in\mathscr{R}(C)$ as
\[\Div(f):=\sum \ord_P(f)P.\]
\paragraph{}
We have $\Div(f)\in\Div(C)$ since there is a finite number of non-zero coefficients. We verify that we have just defined a group homomorphism
\[\Div:\left(\mathscr{R}(C)^\times,\cdot\right)\longrightarrow\left(\Div(C),+\right).\]
\paragraph{}
Let us define the \emph{jacobian variety} $J(C)$ \emph{of the (smooth) curve $C$} as the cokernel of the homomorphism above, i.e.
\[J(C):=\Div(C)/\Div\mathscr{R}(C)^\times.\]
\paragraph{}
We say that two divisors $D_1$ and $D_2$ are \emph{linearly equivalent} if $D_1-D_2$ is a principal divisor, in other words, if $D_1=D_2$ in $J(C)$.

\medskip

Notice that the smoothness assumption was necessary to define the notions above. That is why we will look for a \emph{smooth} completion of the hyperelliptic curve in the following sections.

\subsubsection{A naive compactification}\label{sectionCompactificationNaive}
\paragraph{}
When we talk about the compactness of a curve, we talk about the subjacent set of the curve equipped with the induced topology of $\C^2$. A compact curve has interesting properties which motivates us to compactify curves. Moreover, we want the compactified curve to be \emph{smooth} in order to have a Riemann surface structure.

\medskip The hyperelliptic curve $y^2=D(x)$ is not compact : in fact, intuitively, there is (are) a point(s) at infinity which is (are) not on the affine curve. 
A first possible approach is to use the projective plane $\mathbb{P}^2_\mathbb{K}$, which compactifies the affine plane $\mathbb{A}^2_\mathbb{K}=\mathbb{K}^2$ by adding a projective line at infinity. The compactification induced on the curve is the closure of $C$ in $\mathbb{P}^2_\mathbb{K}$. More precisely, the steps of this compactification are : 
\begin{itemize}
    \item Write the homogeneous form of the equation of the curve $z^{2g}y^2=z^{2g+2}D(x/z)$;
    \item Define the compactification as the algebraic curve defined on $\mathbb{P}^2_\mathbb{K}$ by this equation. 
\end{itemize}

\paragraph{}
We will not detail this approach. However, we state that this compactified curve \emph{is not} smooth at point $\infty = {[0 : 1 : 0]}\in\mathbb{P}^2_\mathbb{K}$ when $g>0$. Therefore we need another compactification.

\subsubsection{Interlude: adjunction of two affine algebraic curves}

\paragraph{}
Before constructing a smooth completion, let us focus on the adjunction of two affine algebraic curves.
\[C_1: F(x,y)=0,\quad C_2: G(u,v)=0\]
on the affine spaces $\mathbb{K}^2_{x,y}$ and $\mathbb{K}^2_{u,v}$ respectively (with coordinates $x,y$ or $u,v$).
\begin{defi}
    Let $C: F(x,y)=0$ be an affine algebraic curve. A \emph{closed algebraic} subset of $C$ is the set of the roots \emph{in $C$} of a finite number of polynomials $H_i\in \mathbb{K}[X,Y]$ for $i\in I$ where $I$ is a finite set. We call it $Z_C(H_i,i\in I)$, or simply  $Z(H_i,i\in I)$, if there is no confusion.
\end{defi}

\begin{defi}
    Let $C_1$ (\resp. $C_2$) be an algebraic curve on $\mathbb{K}^2_{x,y}$ (\resp. $\mathbb{K}^2_{u,v}$).
    A function $\varphi: C_1\to C_2$ is called a \emph{rational homomorphism} if $\varphi=(u(x,y),v(x,y))$ where $u,v\in \mathscr{R}(C_1)$.
\end{defi}

In other words, $\varphi$ is rational if it can be written as $u(x,y)=H_{11}/H_{12},v(x,y)=H_{21}/H_{22}$ where $H_{ij}\in \mathbb{K}[X,Y]$ and $H_{12}, H_{22}$ are non zero on $C_1$.
Let us notice that such a $\varphi$ is <<\`well defined>> (in terms of subjacent function) on $C_1\backslash Z(H_{12},H_{22})$.

\begin{defi}
    An \emph{Adjunction} between two Riemann surfaces $C_1$ and $C_2$ is given by:
    \begin{itemize}
        \item[(i)] two closed algebraic subsets $\Sigma_1$ and $\Sigma_2$ respectively of $C_1$ and $C_2$;
        \item[(ii)] a couple of rational homomorphisms
        \[\begin{tikzcd}
        C_1 \arrow[r, shift left=0.5ex,"\varphi"] & C_2 \arrow[l, shift left=0.5ex,"\psi"]
        \end{tikzcd}\]
        inducing a well-defined isomorphism
        \[\begin{tikzcd}
        \varphi:\;C_1\backslash\Sigma_1 \arrow[r, shift left=0.5ex,"\sim"] & C_2\backslash\Sigma_2\;:\psi \arrow[l, shift left=0.5ex]
        \end{tikzcd}\]
    \end{itemize}
\end{defi}

With the notion of adjunction, one can obtain a new algebraic curve $C'$ by considering the disjoint union $C_1\dot\cup C_2$ and identifying the points of $C_1\backslash\Sigma_1$ and $C_2\backslash\Sigma_2$ using $\varphi$ and $\psi$. When considering a point coming from $C_i$, we can look at its neighborhood in $C'$ which can be identified to $C_i$ with the corresponding homomorphism. However, the new curve is not necessarily an affine curve.

\medskip

Notions of field of rational functions, smoothness, divisors, principal divisors, \emph{etc.} can be generalized to the case of curves obtained by the adjunction of two affine curves, by considering the definitions on affine curves $C_1$ and $C_2$. For example, $f$ is a rational function on $C'$ if $f|_{C_1}$ and $f|_{C_2}$ are both rational functions, on $C_1$ and $C_2$ respectively. We will not get into more details here.

\subsubsection{Smooth completion}\label{sectionCompactificationLisse}
\paragraph{}
Let $D\in \mathbb{K}[x]$ be a monic polynomial of degree $2g+2$ whose roots are all distinct. We have an affine hyperelliptic curve
\[C:y^2=D(x).\]
\paragraph{}
Its compactification, referred as $\widetilde{C}$ for now, can be obtained by adding two points at infinity $\infty_+$ and $\infty_-$. It is the curve that we will study in the following sections.

\medskip This compactification can be built as follows : $\widetilde{C}$ $=U_0\cup U_\infty$, where 

$$U_0: y^2=D(x),\quad U_\infty:v^2=u^{2g+2}D(1/u),$$
the adjunction of this two affine curves being given by

\begin{equation}\label{recollementCourbeHyperellip}
    \begin{tikzcd}[row sep=tiny]
        \qquad U_0\qquad \arrow[r, shift left=0.5ex] & \qquad U_\infty\qquad \arrow[l, shift left=0.5ex]\\
        U_0\backslash\{x=0\} \arrow[r, shift left=0.5ex,"\sim"] & U_\infty\backslash\{u=0\} \arrow[l, shift left=0.5ex] \\
        (x,y) \arrow[r, mapsto] & (1/x,y/x^{g+1}) \\
        (1/u,v/u^{g+1})   & (u,v). \arrow[l, mapsto] \\
    \end{tikzcd}
\end{equation}

The points $\infty\pm$ correspond to $(0,\pm1)$ on $U_\infty$, \emph{i.e.} on the curve $v^2=u^{2g+2}D(1/u)$. By restricting our study to $U_0$ or $U_\infty$, we check immediately that the obtained curve $\widetilde{C}$ is smooth ; it is equipped with a natural complex manifold structure. Moreover, we have

\begin{prop}\label{genus=g}
    $\widetilde{C}$ is a compact Riemann surface of genus equals to $g$. \qed
\end{prop}

Once we have the compactification $\widetilde{C}$ , we will refer to it as $C$ with no risk of confusion.

\begin{prop}
    On a (smooth) compactified hyperelliptic curve, the degree of all the principal divisors is zero. In other words, we have $(\deg\circ\Div) \mathscr{R}(C)^\times=0$.
\end{prop}

\begin{proof}
    It follows from the residue theorem on a compact Riemann surface (\cite{Bost1992}, Prop. B.2.2) that the sum of residues of some meromorphic form is zero. Let $f$ be a rational function, by applying the previous result to the form $\dif f/f$, the argument principle gives us that $f$ has as many zeros as it has poles, counted with multiplicity.
\end{proof}

\medskip The compactified curve is equipped with an involution $\iota$ induced by the function $(x,y)\mapsto(x,-y)$ on the affine space, which  swaps $\infty_+$ and $\infty_-$. We call $\iota$ the \emph{conjugation} of this hyperelliptic curve. This conjugation induces an involution on the group of divisors on the curve: $\sum n_iP_i\mapsto\sum n_i\iota(P_i)$.

\begin{ex}[Calculation of a principal divisor on $C$]  Let us denote by $\alpha_1,\ldots,\alpha_{2g+2}$ the roots of $D(x)$ (which are assumed to be distinct) and by $P_i\in C$ their corresponding points, then
\begin{equation}\label{CalculDiv1/y^2}
    \Div({1}/{y^2})=-2\sum_{i=1}^{2g+2}P_i+(2g+2)\left(\infty_++\infty_-\right).
\end{equation}
Its degree is zero, as expected.
\end{ex}

\begin{proof}
Let us first define the function $1/y^2$: It is defined as $1/y^2$ on $U_0$ and $u^{2g+2}/v^2$ on $U_\infty$; We check that they coincide on $U_0\cap U_\infty$  by using the substitution (\ref{recollementCourbeHyperellip}), therefore this rational function is well defined.

Then to simplify, let us consider the case where $g=1,D(x)=x(x^3-1)$ and calculate the multiplicity from a "complex analysis" point of view (i.e with charts, etc.). The equation of the curve on $U_\infty$ becomes $v^2=1-u^3$.

Let us calculate its multiplicity at point $\infty\pm$. Let us consider $U_\infty$. Then $\infty_+=(0,1)$, and in a neighborhood of $\infty_+$, we can use the chart ($V_1$ being some neighborhood of  0)
$$\varphi_1:\C\supset V_1\longrightarrow C,\quad z\longmapsto (u,v)=(z,\sqrt{1-z^3}).$$
For this chart, $\varphi_1(0)=\infty_+$ and
$$1/y^2=u^4/v^2=z^4/(1-z^3),$$
therefore
$$\ord_{\infty_+}(1/y^2)=4=2g+2.$$
Likewise we have
$$\ord_{\infty_-}(1/y^2)=4=2g+2.$$

Let us calculate the multiplicity of $1/y^2$ at point $P_1=(x,y)=(0,0)\in U_0$, for example.
Previously $u$ was a local (holomorphic) parameter of $C$ near $\infty_\pm$. \emph{Now, the variable $y$ will be used as a local parameter of $C$ near $P_1$.} Therefore
$$\ord_{P_1}(1/y^2)=-2.$$
More precisely, since $D(x)=x(x^3-1)$ has a simple root at $x=0$, the holomorphic function $\psi:x\mapsto w=x(x^3-1)$ is locally biholomorphic at points $x=0$ and $w=0$, in other words, on sufficiently small neighborhoods of $x=0$ and $w=0$, $\psi$ has a holomorphic inverse function which we will call $\psi^{-1}(w)$. Therefore, by considering the following chart ($V_2$ being some neighborhood of 0):
$$\varphi_2:\C\supset V_2\longrightarrow C,\quad z\longmapsto (x,y)=(\psi^{-1}(z^2),z),$$
we have 
$$1/y^2=1/z^2,$$
therefore
$$\ord_{P_1}(1/y^2)=-2.$$
The same goes for all the $P_i=(\alpha_i,0)\in U_0$ where $\alpha_i$ is a root of $D(x)$. Note that $1/y^2$ has no other pole nor root, hence the formula (\ref{CalculDiv1/y^2}).
\end{proof}

\begin{ex}
Let $A\in\R[x]$ be a monic polynomial and $\eta_A=\frac{A(x)\dif x}{y}$ the differential form associated with $A$. If $\deg A=g$, then $\eta_A$ has simple poles at points $\infty_+$ and $\infty_-$,with residues $-1$ and $+1$ respectively; if $\deg A<g$, then $\eta_A$ is a holomorphic form.
\end{ex}

\begin{proof}
    It is enough to make the calculations using the two charts, like in the previous example.
\end{proof}

\subsubsection{Hyperelliptic curves and Pell-Abel equation}

\paragraph{}
Let us prove an important result : A link between the existence of solutions to Pell-Abel equation and hyperelliptic curves. 

\begin{defi}
Let $D\in \mathbb{K}[X]$ be a monic polynomial of degree $2g+2$ with distinct roots. 
We call \emph{Pell-Abel equation} the equation with unknowns $(P,Q)\in \mathbb{K}[x]^2$ of the form :
\[ P^2-DQ^2=c\text{, with } c\in \mathbb{K}^*.\]
The degree of a solution $(P,Q)$ is, by definition, the degree of $P$.
\end{defi}

\begin{thm} \label{thm-PellianEquivautTorsion}
The following claims are equivalent:
\begin{enumerate}
    \item The Pell-Abel equation has a solution of degree $r$ in $\mathbb{K}[x]$.
    \item The divisor $r((\infty_-) - (\infty_+))$ on the curve $C$ (cf. section \ref{sectionCompactificationLisse}) is linearly equivalent to $0$.
\end{enumerate}
\end{thm}
\begin{proof}
Let $C$ be the curve defined by $y^2=D(x)$. The ring of regular functions on $C$ is a $\mathbb{K}[x]$-module with basis $\{1,y\}$, with $y^2=D(x)$. 

\medskip

Assume there exists $P,Q\in \mathbb{K}[x]$, $c\in \mathbb{K}^*$ such that $P^2-DQ^2=c$. Let us consider $\varphi_+ = P+yQ$, $\varphi_- = P-yQ$. Since $\varphi_+$ and $\varphi_-$ are regular functions on the affine hyperelliptic curve, their poles are in the set $\{\infty_+, \infty_-\}$. However $\varphi_+ \varphi_- = P^2-DQ^2 = c$, therefore their zeros are also in the set $\{\infty_+, \infty_-\}$. Therefore $\Div(\varphi_+) = a(\infty_-)+b(\infty_+)$. However $\deg(\Div(\varphi_+))=0$, therefore $b=-a$ and the degree of $P$ gives us that $a=r$. Therefore $ r\left((\infty_-)-(\infty_+)\right)$ is principal (therefore linearly equivalent to $0$).

\medskip
Conversely, assume that $ r((\infty_-)-(\infty_+))$ is linearly equivalent to $0$. Let $\psi$ be a rational function such that $\Div(\psi) =  r((\infty_-)-(\infty_+))$. Since $\psi$ has no pole on the affine space, $\psi$ is regular. Therefore $\psi$ can be written as $\psi = P+yQ$ with $P,Q \in \mathbb{K}[x]$, and $\deg P = r$. Its conjugate is  $\overline{\psi} = P-yQ$. However conjugation is involutory and swaps $\infty_+$ and $\infty_-$, therefore $\Div(\overline{\psi}) = (-r)((\infty_-)-(\infty_+))$, hence $\Div(\psi\overline{\psi})=0$. Therefore $P^2-DQ^2 = \psi\overline{\psi}$ is a non-zero constant of $\mathbb{K}$. 
\end{proof}

\subsubsection{A visual example}\label{sectionExVisuel}

\paragraph{}
Let us introduce the hyperelliptic curve that we will use in the proof of Robinson's theorem. Let $$a_0<b_0<a_1<\dots<a_g<b_g$$ be distinct real numbers. Let us consider $E_j =[a_j,b_j]$, $E= \displaystyle\bigcup_{j=0}^g E_j$ and the polynomial $$D = \displaystyle\prod^g_{j=0} (X-a_j)(X-b_j)$$ and let us call $C$ the compactified hyperelliptic curve associated with $D$ (cf. section \ref{sectionCompactificationLisse}).

\medskip

Let us provide an illustration of this curve in the case where $g=2$.

\begin{figure}[ht!]
    \centering
    \includegraphics[scale = 0.1]{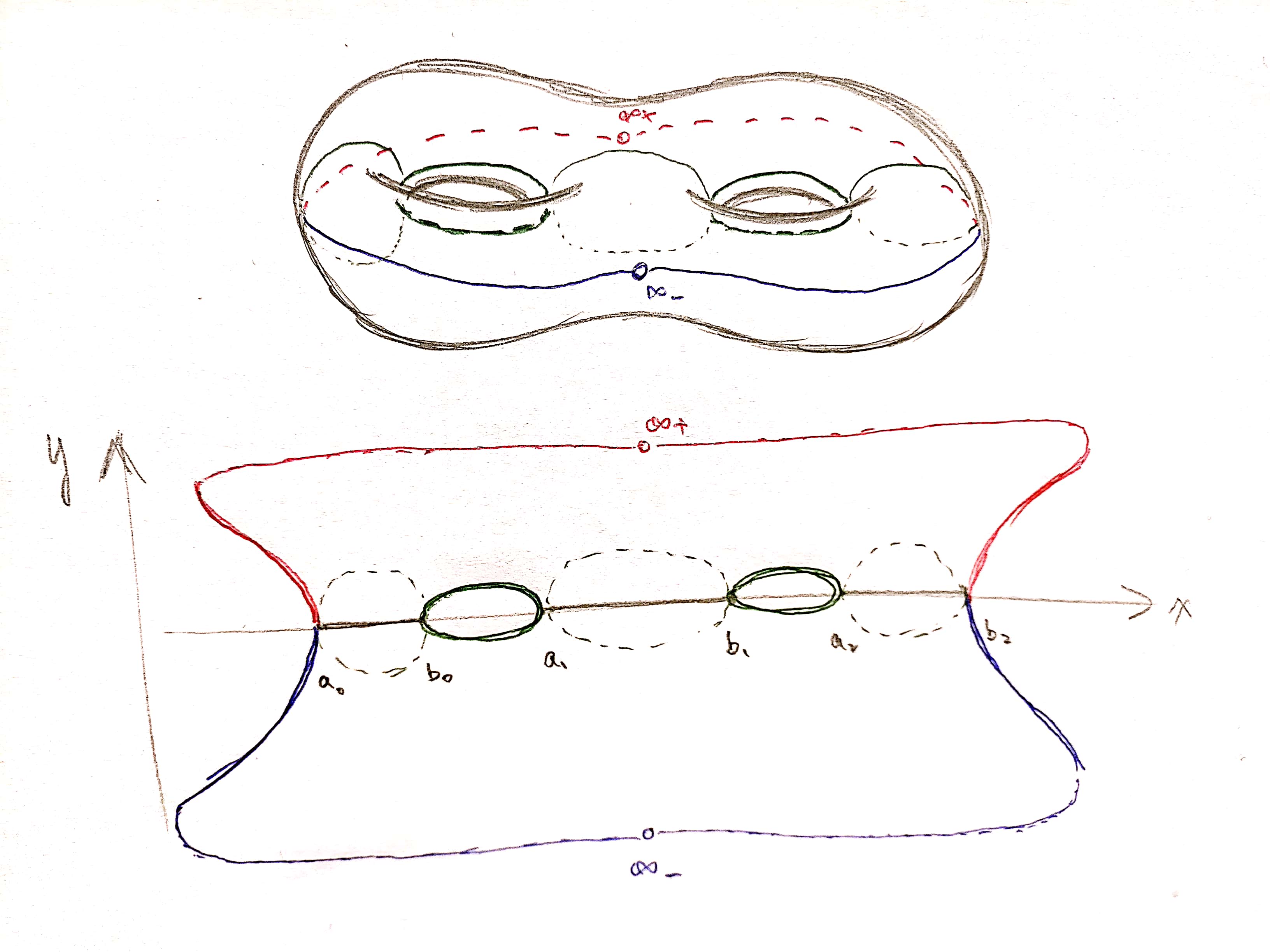}
    \caption{Case $g=2$: $C$ as a Riemann surface (above) or a curve on the affine plane (bellow)}.
    \label{ExempleVisuel1}
\end{figure}

Figure \ref{ExempleVisuel1} shows the connection between these two representations of curve $C$.
The figure below, on the affine plane (solid lines represent real points) can be obtained by intersecting the surface above with a horizontal plane (and by removing points " at infinity " $\infty\pm$). 
We clearly see the position of $\infty\pm$, the two points at infinity added to the affine hyperelliptic curve.
We also understand the singularity of the naive compactification (cf. section \ref{sectionCompactificationNaive}) : it can be found again by attaching the two points $\infty\pm$ in the figure above !

\clearpage

 Let us denote by $\alpha_j$ the cycle on the curve which covers the interval $[b_{j-1},a_j]$.
Then there exists a cycle $\beta_j$ crossing the $j$-th <<hole>> of the surface, which only intersects cycle $\alpha_j$ and with multiplicity $1$.
We say that $\{\alpha_1,\ldots,\alpha_g,\beta_1,\ldots,\beta_g\}$ forms a \emph{symplectic basis of the singular homology} $H_1(C,\Z)$.
This basis is illustrated, when $g=2$, in Figure \ref{ExempleVisuel2}.

\begin{figure}[ht!]
    \centering
    \includegraphics[scale = 0.1]{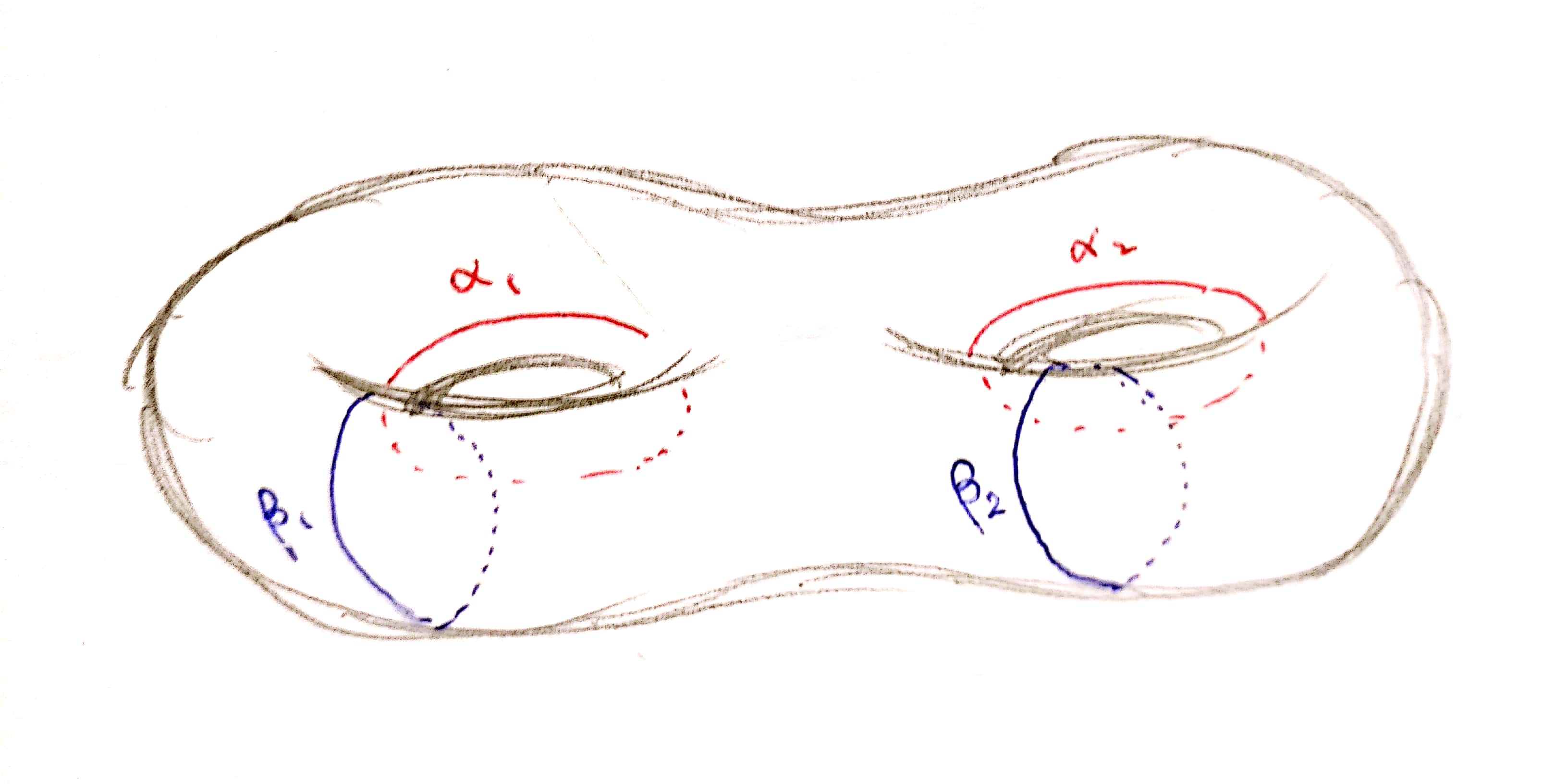}
    \caption{Case $g=2$: a symplectic basis of $H_1(C;\Z)$}.
    \label{ExempleVisuel2}
\end{figure}

The genus $g$, which intuitively corresponds to the number of holes in the surface, also corresponds to the dimension of the space of holomorphic forms.

\begin{prop}
    The complex dimension of the space of holomorphic 1-forms on some Riemann surface of genus $g$ is equal to $g$.
\end{prop}

The integral of a holomorphic form along a cycle provides a \emph{perfect coupling}, as explained below : 

\begin{thm}\label{thm-accouplementParfait}
    Let $\{\alpha_1,\ldots,\alpha_g,\beta_1,\ldots,\beta_g\}$ be a symplectic basis of $H_1(C,\Z)$.
    Let $\{\omega_1,\ldots,\omega_g\}$ be a basis of the space of holomorphic 1-forms on $C$.
    Then the $g$-by-$g$ square matrix with coefficients $A_{i,j}$ defined as follows, is invertible: 
    \[A_{ij}=\int_{\alpha_j}\omega_i.\]
\end{thm}

The proof of this theorem is provided in \cite{Bost1992}, Theorem III.1.2.

\clearpage

\subsection{Proof of Robinson's theorem}
\label{section-robinson}

\paragraph{}
The goal of the following section is to prove what follows :

\begin{thm}[Robinson]\label{thmRobinson}
Let $E$ be a finite union of intervals of $\R$ such that $\capa (E)>1$, then there exists an infinite number of algebraic integers totally in $E$.
\end{thm}
\begin{rmq}
It is enough to prove the theorem for a union of disjoint segments, since the union of two non-disjoint segments forms a segment, and the capacity of an open interval is defined as the supremum of the capacities of segments included in the interval.
\end{rmq}

Let us consider $a_0<b_0<a_1<\dots<a_n<b_n$, $E_j =[a_j,b_j]$ and $E= \displaystyle\bigcup_{j=0}^g E_j$. 

Let us consider $D = \displaystyle\prod^g_{j=0} (X-a_j)(X-b_j)$. The proof of Robinson's theorem is based on the geometry of the hyperelliptic curve $y^2=D(x)$. We prove it first for the \emph{Pell-Abel case}, then we generalize the result using density.

\subsubsection{Pell-Abel case}
\paragraph{}
Assume the Pell-Abel equation with respect to polynomial $D$ has a solution, \emph{i.e.} there exists $P,Q\in\R[X]$ such that 
\[P^2-DQ^2=c\]
where $c$ is a non-zero real number. Let us consider $r:=\deg(P)$.

\medskip

Note first that $c>0$, in fact $c=P(a_0)^2$ since $a_0$ is a root of $D$. Let us then consider $M = \sqrt{c}$ and write Pell-Abel equation as follows : 

\[P^2-DQ^2 = M^2.\]
Let us notice immediate properties of polynomials $P$ and $Q$: 
\begin{prop}\label{prop-observationP}
Let us consider $x\in\R$,
\begin{enumerate}
    \item $\abs{P(x)}\leqslant M \Leftrightarrow x\in E$ or $Q(x)=0$.
    \item $\abs{P(x)}= M \Leftrightarrow x$ is a root of $Q$ or one of the $a_i$, $b_j$.
\end{enumerate}

\end{prop}

\begin{proof}
These two properties directly follow from the Pell-Abel equation, by noticing that $D(x)\leqslant 0 \Leftrightarrow x\in E$. Indeed, $D$ tends to $+\infty$ when $x$ tends to $\infty$, and changes sign at points $a_i$, $b_j$.
\end{proof}

\medskip
\clearpage
The key to the proof of Robinson's theorem is based on the following property of the roots of polynomials $P$ and $Q$, illustrated in figure \ref{fig:P_genre1}.

\begin{prop} \label{prop-cle}
 Let us denote by $r_j$ the number of roots of $P$ in $E_j$.
\begin{enumerate}
\item The roots of $P$ and $Q$ are simple, interlaced, and all belong to $\overset{\circ}{E} =\bigcup_{j=0}^g ]a_j, b_j[ $.
\item The roots of $Q$ in $E_j$ divide $E_j$ into $r_j$ sub-intervals; in each of them, the polynomial $P$ is either strictly increasing, or strictly decreasing, with extreme values $M$ and $-M$.
\end{enumerate}
\end{prop}

These properties are illustrated in figure \ref{fig:P_genre1} (case $g=1$, $r_1=4$, $r_2=6$). The $p_i$ are the roots of $P$ and $q_i$ are the roots of $Q$.

\begin{figure}[h!]
    \centering
    \includegraphics[scale=0.2]{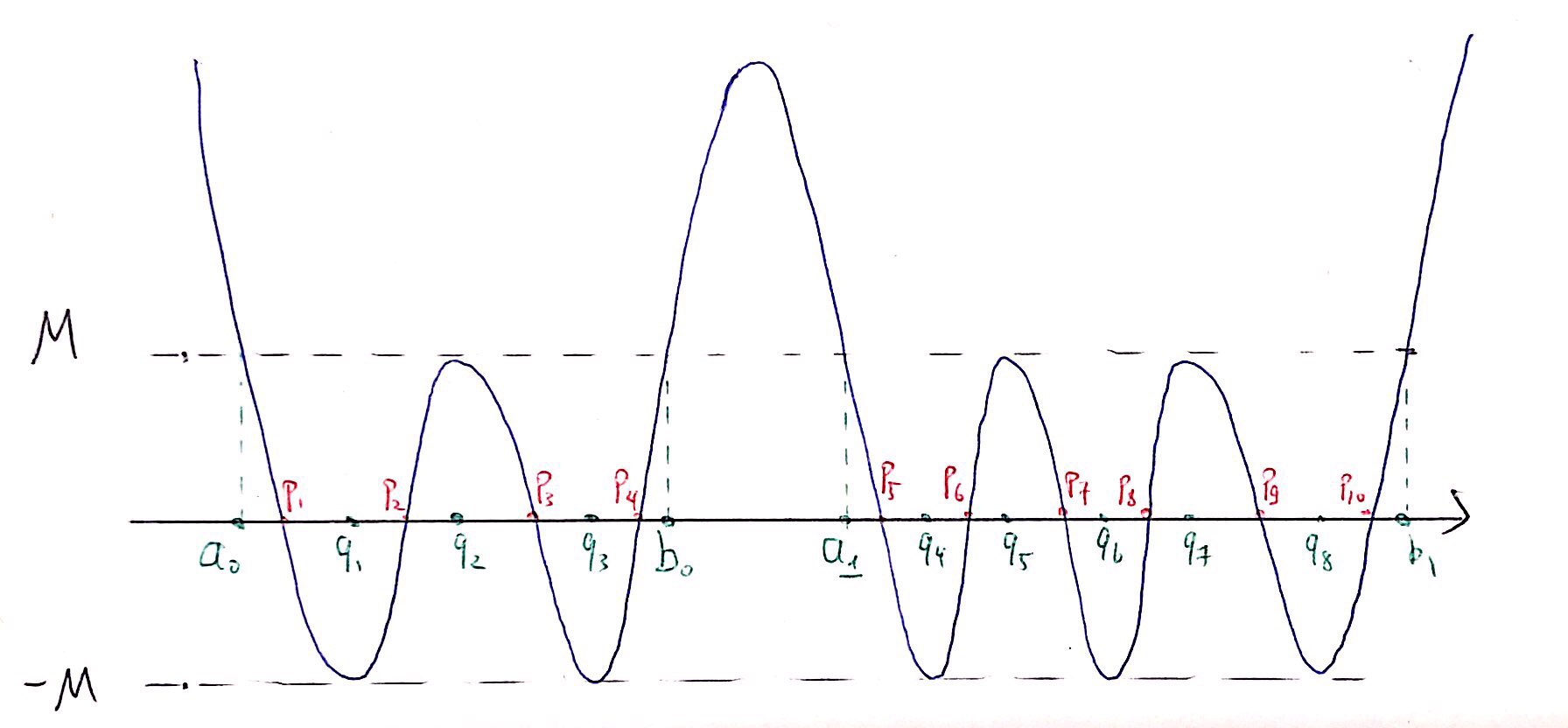}
    \caption{Shape of $P$, case $g=1$, $r_1=4$, $r_2 =6$}
    \label{fig:P_genre1}
\end{figure}

This proposition is the very key to the proof of Robinson's theorem. Because of its importance, it is very enlightening to prove it in the case $g=0$ (in other words, $E=[a,b]$ is a single segment). It allows to show the major steps of the proof, by dissociating the difficulties. 

\subsubsection{Proof of proposition \ref{prop-cle} in the case $g=0$}

We focus here on the case where $E=[a,b]$ is a segment of $\R$, therefore $g=0$ and $D(X)=(X-a)(X-b)$.
\medskip
Let us consider $f=P+yQ$, regular function on the curve $y^2=D(x)$ completed as in the beginning of the chapter (the curve is not hyperelliptic anymore but only quadratic). Let us start with the calculation of $\dif f/f$ : 

\begin{prop}\label{prop-df/f}
We have $\dif f/f = r\dif x/y$, where $r=\deg P$.
\end{prop}
\begin{proof}
  We know that $\Div(f)=r((\infty_-)-(\infty_+))$ (\emph{cf.} the proof of theorem \ref{thm-PellianEquivautTorsion}), therefore, $\infty_-$ is a zero of $f$ with multiplicity $r$ and $\infty_+$ is a pole with the same multiplicity. Hence, $\dif f/f$ has two simple poles at points $(\infty_+)$, $(\infty_-)$ with residues $-r$ and $r$. Let us calculate the poles of the form $\dif x/y$ and their multiplicity. The only potential poles of $\dif x/y$ are either  $(\infty_+)$, $(\infty_-)$, or the roots of polynomial $D$.
\medskip
\begin{itemize}
    \item If $\alpha$ is a root of $D$ in the affine space, since the roots of $D$ are simple, we can write $D(X)=(X-\alpha)H(X)$, with $H(\alpha)\neq 0$. By differentiating $y^2=(X-\alpha)H(X)$, we obtain in a neighborhood of $\alpha$ :
    $\dif x \sim \frac{2y}{H(\alpha)}\dif y$. 
    
    Therefore $\dif x/y = \frac{2}{H(\alpha)} \dif y $ has no pole at point $\alpha$.
    \item At points $\infty\pm$, with coordinates $(u=\frac 1x, v= \frac yx)$, $\infty_+ = (0,1)$, and the equation of the curve is $v^2=u^2D(\frac 1u)$. Therefore $\frac{\dif x}{y} = - \frac{u}{v}\frac{\dif u}{u^2} = -\frac{\dif u}{\pm\sqrt{u^2D(\frac 1u)}u}\sim \mp\frac{\dif u}{u}$ since the constant coefficient of $u^2D(\frac 1u)$ is $1$. Hence $\infty_+$ is a simple pole of $\dif x/y$, and the same goes for $\infty_-$.

\end{itemize}

Therefore $\frac{\dif f/f}{\dif x/y}$ has no zero nor pole neither on the affine space, nor at infinity : It is a constant function. But according to the previous calculation, the residues of $\dif x/y$ at $\infty\pm$ are $\mp1$, hence the constant is equal to $r$.
\end{proof}

\medskip

Let us now prove proposition \ref{prop-cle} in the case where $E$ is a segment, illustrated in figure \ref{fig:P_genre0} :

\begin{prop}\label{prop-clesegment}
\begin{enumerate}
\item The roots of $P$ and $Q$ are simple, interlaced, and all belong to $\overset{\circ}{E}$.
\item The roots of $Q$ divide $E$ into $r$ sub-intervals; in each of them, the polynomial $P$ is either strictly increasing, or strictly decreasing, with extreme values $M$ and $-M$.
\end{enumerate}
\end{prop}

\begin{figure}[h!]
    \centering
    \includegraphics[scale=0.2]{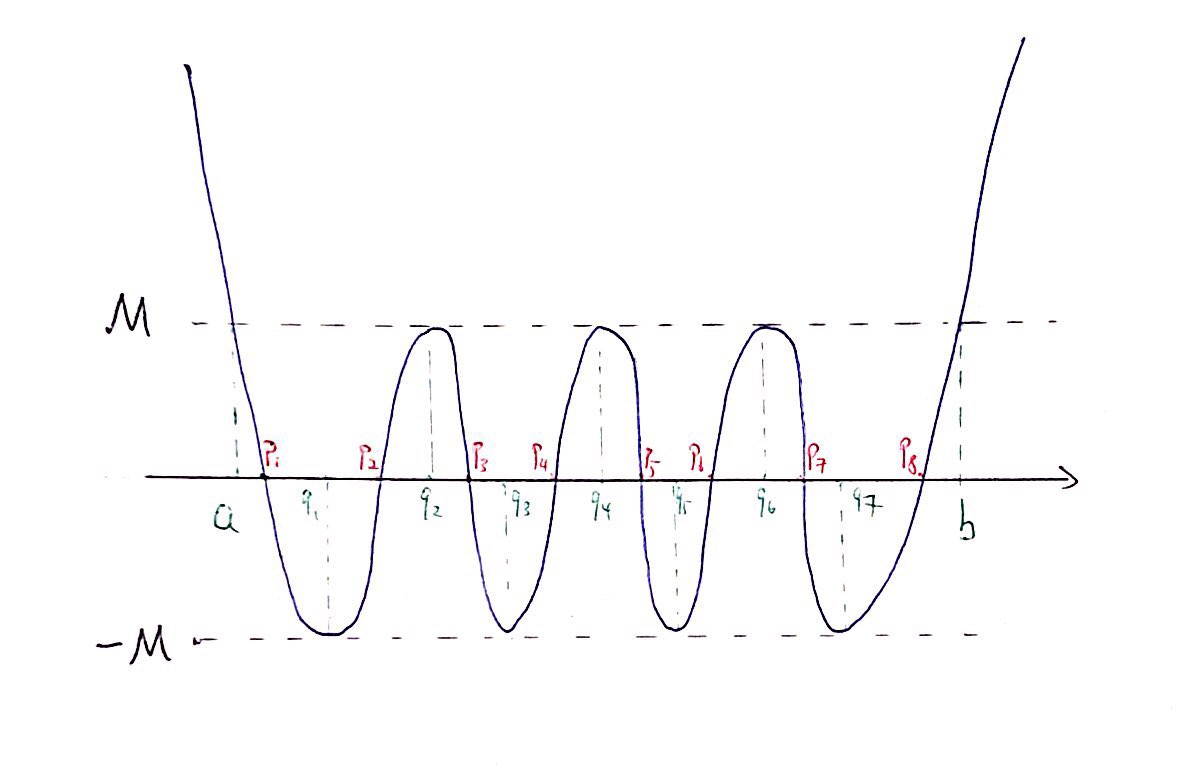}
    \caption{Shape of $P$, in the case $g=0$, $r=8$ }
    \label{fig:P_genre0}
\end{figure}

\begin{proof}
A representative of $f=P+yQ$ in $E$ is $f(x)=P(x)+iy_1(x)Q(x)$, with $y_1(x)=\sqrt{-D(x)}$ since $D(x)\leqslant 0$ for all $x\in E$. The Pell-Abel equation implies that $\abs{f(x)}^2 = P(x)^2-D(x)Q(x)^2 = M^2$, therefore $f$ has a constant magnitude which is equal to $M$ on $E$. 

\medskip

We can then write that $f(x)=M e^{i\theta(x)}$, with $\theta : E \rightarrow \R$ of class $C^1$. Hence, we have $P(x)=M\cos(\theta(x))$, $y_1(x)Q(x)=M\sin(\theta(x))$. And since $f(a)=\pm M$, $f(b)=\pm M$, we have $\theta(a)=c_0 \pi$,  $\theta(a)=c_1 \pi$, where $c_0,c_1\in \Z$. Moreover, we have that $\int_a^b\frac{\dif f}{f} = i\int_a^b \dif \theta = i(c_1-c_0)\pi$. However $\frac{\dif f}{f}=r\frac{\dif x}{y}$, and $\int_a^b \frac{\dif x}{y}= \pm i\pi $ (the residue of $\frac{\dif x}{y}$ at infinity is $\pm 1$ and the curve $y^2=D(x)$ is a covering map of degree $2$ of $E=[a,b]$. The sign $\pm$ depends on the orientation.) We deduce from it that $\abs{c_1-c_0} =r$.

\medskip

However, according to proposition \ref{prop-df/f}, $f'$ is never zero on $E$, therefore $\theta'$ is never zero on $E$. Hence, $\theta$ is strictly monotonic on $E$, taking values in the range $c_0\pi$ to $c_1\pi$, therefore $\cos(\theta(x))$ is equal to zero $\abs{c_1-c_0}=r$ times, and $\sin(\theta(x))$ is equal to zero $\abs{c_1-c_0}-1=r-1$ times (end points not included). Since $\deg P =r$, $\deg Q = r-1$, we found all the roots of $P$ and $Q$. The other claims of the proposition directly follow from the sinusoidal shapes of $P$ and $Q$.
\end{proof}

The case where $E$ is a single segment already allows us to use the main steps of the proof in the general case while omitting a major difficulty : when $E$ is the union of segments $E_j$, we do not know \textit{a priori} how many roots of $P$ are in each $E_j$, and how we can state that we have found all the roots. To solve this problem in the general case, we will use the \emph{periods} of a differential form on the hyperelliptic curve of genus $g>0$.

\subsubsection{Proof of proposition \ref{prop-cle} when $g>0$}
\paragraph{}
Let us consider the general case: $E=\displaystyle\bigcup_{i=0}^g[a_i,b_i]$.

As previously stated, the difficulty is to determine how many roots are in $E_j=[a_j, b_j]$ and to check that we have indeed found all the roots of $P$. To do so, we shall use the periods of a particular differential form on the curve $y^2 =D(x)$, called \emph{canonical form}.

\begin{thm}\label{thm-formecanonique}
There exists a unique polynomial $R\in \R[X]$ of degree $g$ such that, for $j = 1, \dots, g$, 
\[\int_{b_{j-1}}^{a_j}\frac{R(x)}{\sqrt{D(x)}}\dif x = 0.\]
The differential form $\eta = \frac{R(x)\dif x}{y}$ is said to be of \emph{the third kind}.
\end{thm}

\begin{proof}
 Let us consider the curve $y^2=D(x)$ as a Riemann surface thanks to the covering map $(x,y)\mapsto x$. The dimension of the space of holomorphic forms is $g$. A residue calculation with coordinates $(u=\frac 1x, v= \frac{y}{x^{g+1}})$ gives us that for $i =0, \dots ,g-1$, the $x^i\dif x/y$ are holomorphic forms, therefore form a basis of holomorphic forms. 
 
 \medskip
 
 Let us denote by $\alpha_j$ the cycle on the curve which covers (quadratically) the interval $[b_{j-1},a_j]$. A monic polynomial of degree $g$, $R(X)= X^g + \displaystyle\sum_{j=0}^{g-1} c_j X^j$ verifies the wanted condition if and only if its coefficients $c_0,\dots c_{g-1}$ verify the following system of $g$ equations with $g$ unknowns:
 \[\sum_{i=0}^{g-1}c_i\int_{\alpha_j} \frac{x^i\dif x}{y} = - \int_{\alpha_j} \frac{x^g\dif x}{y},\qquad j= 1,\dots g.\]
 It follows from theorem \ref{thm-accouplementParfait} that the determinant associated with this system is non-zero and therefore the system has a unique solution. 
\end{proof}
\begin{rmq}
    By a calculation with coordinates $(u=1/x, v=y/x^{g+1})$, we can show that $\eta$ is a meromorphic form which has two simple poles at $\infty_+$ and $\infty_-$, with residues $-1$ and $1$ respectively. 
\end{rmq}

\begin{rmq}\label{rmq-Rnesannulepas}
    The relations $\int_{b_{j-1}}^{a_j}\frac{R(x)}{\sqrt{D(x)}}\dif x = 0 $ require that $R$ has at least one zero in each $[b_{j-1}, a_j]$, for $j=1,\dots,g$. Since $R$ is of degree $g$, these are all its roots. In particular, $R$ has a constant sign on $E_j = [a_j, b_j]$.
\end{rmq}

Let us focus on the periods of $\eta$ on the cycles that cover segments $[a_j,b_j]$ (it is a quadratic covering space, ramified in points $a_j$, $b_j$). These periods are equal to $2\eta_j$, with $\eta_j = \displaystyle \int_{a_j}^{b_j} \frac{R(x)}{i\sqrt{-D(x)}}\dif x$.
\newpage
\begin{prop}\label{prop-periodeimaginaire}
Let us note $\eta_j = \displaystyle \int_{a_j}^{b_j} \frac{R(x)}{i\sqrt{-D(x)}}\dif x$, for $j = 0,\dots g$. \\
There exists $(\epsilon_j)_{0\leq j \leq g}\in\{-1,1\}^{g+1}$ such that 
\[\sum_{j=0}^g \epsilon_j \eta_j = i\pi.\]
\end{prop}

\begin{figure}[ht!]
    \centering
    \includegraphics[scale = 0.1]{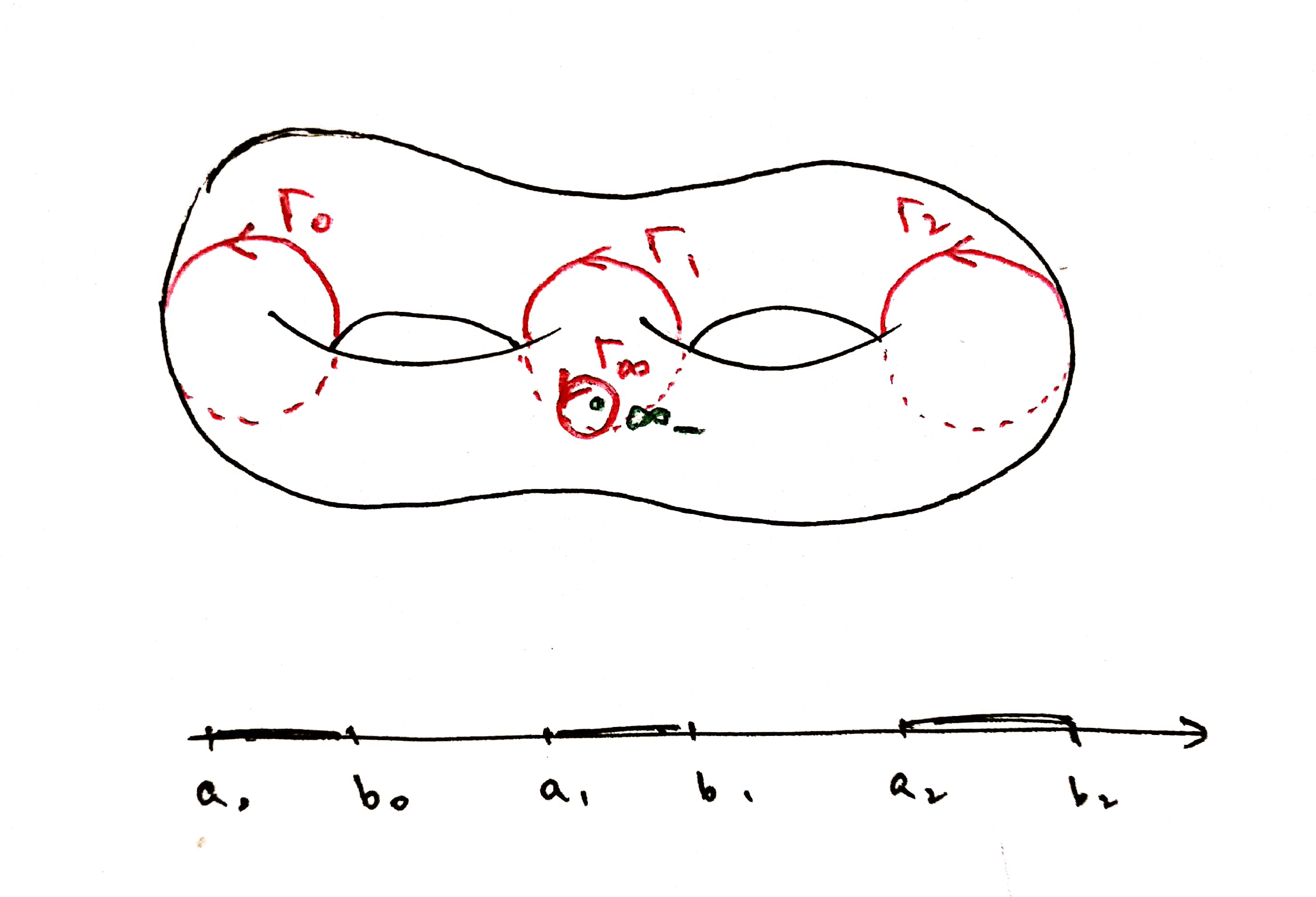}
    \caption{Case where $g=2$: $\bigcup_{j=0}^g\Gamma_j$ is the boundary of the medial hemi-surface.}
    \label{ExempleVisuel3}
\end{figure}

\begin{proof}
    Let us denote by $\Gamma_j$ the cycles which cover the segments $[a_j,b_j]$ (cf. Figure \ref{ExempleVisuel3}), the periods of $\eta$ on these cycles are equal to $\pm2\eta_j$. Let us denote by $\Gamma_\infty$ a cycle around the point $\infty_-$ (cf. Figure \ref{ExempleVisuel3}).

    The form $\eta$ is holomorphic on the domain framed by $\bigcup\Gamma_j$ and $\Gamma_\infty$.
    Therefore, Cauchy's theorem on a Riemann surface gives us that
    \[\sum_{j=0}^g \int_{\Gamma_j}\eta = \int_{\Gamma_\infty}\eta=2i\pi\]
    since the residue of $\eta$ at $\infty_-$ is equals to $1$.
    Hence $\displaystyle\sum_{j=0}^g \epsilon_j\eta_j = i\pi$, the signs depend on the orientation of the cycles.
\end{proof}

Like in the case $g=0$, the calculation of $\dif f/f$ will be useful later : 

\begin{prop}
We have $\dif f/f = r\eta$, where $r=\deg P$ and $\eta$ is the differential form of the third kind previously defined.
\end{prop}
\begin{proof}
    We know that the forms $\dif f/f$ and $r\eta$ have $\infty_+$ and $\infty_-$ as their only poles. These poles are simple, with residues $-r$ and  $r$ respectively. Therefore, $\dif f/f -r\eta$ is a holomorphic form which can be written as a linear combination of $x^j\dif x/y$ for $j=0,\dots g-1$. We know that the periods of $\eta$ on the $[b_{j-1}, a_j]$ are zero. To show that this linear combination is zero, thanks to theorem \ref{thm-accouplementParfait}, it is enough to show that the periods of $\dif f/f$ on $[b_{j-1}, a_j]$ are also zero.
    
    \medskip
     
    Let us consider $j\in\{0,\dots,g\}$. A representative of $y$ on the $[b_{j-1}, a_j]$ is $y=\sqrt{D(x)}$, since $D(x)$ is non-negative. Then $f(x) = P(x) + \sqrt{D(x)}Q(x)$ can be seen as a real function on $[b_{j-1}, a_j]$. We have $\abs{f(b_{j-1})} = \abs{P(b_{j-1})} = M$, $\abs{f(a_j)} = \abs{P(a_j)} = M$. Since $(P+\sqrt{D}Q)(P-\sqrt{D}Q) = M^2$, the function $f$ is never zero, hence has a constant sign $\varepsilon \in \{-1,1\}$ on $[b_{j-1}, a_j]$. Therefore $f(b_{j-1}) = f(a_j) = \varepsilon M$. Thus,
    \[\int_{b_{j-1}}^{a_j} \frac{\dif f }{f} =\int_{b_{j-1}}^{a_j} \dif (\log(\varepsilon f) )= \log(\varepsilon f(a_j)) - \log(\varepsilon f(b_{j-1})) = 0, \]
    hence the final result.
\end{proof}

\medskip

Let us finally prove proposition \ref{prop-cle} in the general case. More precisely we have : 

\begin{prop} Let us consider $\eta_j = \displaystyle \int_{a_j}^{b_j} \frac{R(x)}{i\sqrt{-D(x)}}\dif x$ and $r_j = r\abs{\eta_j}/\pi$ for $j = 0,\dots,g.$
\begin{enumerate}
\item The number of roots of $P$ in $E_j=[a_j,b_j]$ is $r_j$.
\item The roots of $P$ and $Q$ are simple, interlaced, and all belong to $\overset{\circ}{E}$.
\item The roots of $Q$ in $E_j$ divide $E_j$ into $r_j$ sub-intervals; in each of them, the polynomial $P$ is either strictly increasing, or strictly decreasing, with extreme values $M$ and $-M$.
\end{enumerate}
\end{prop}

\begin{figure}[h!]
    \centering
    \includegraphics[scale=0.2]{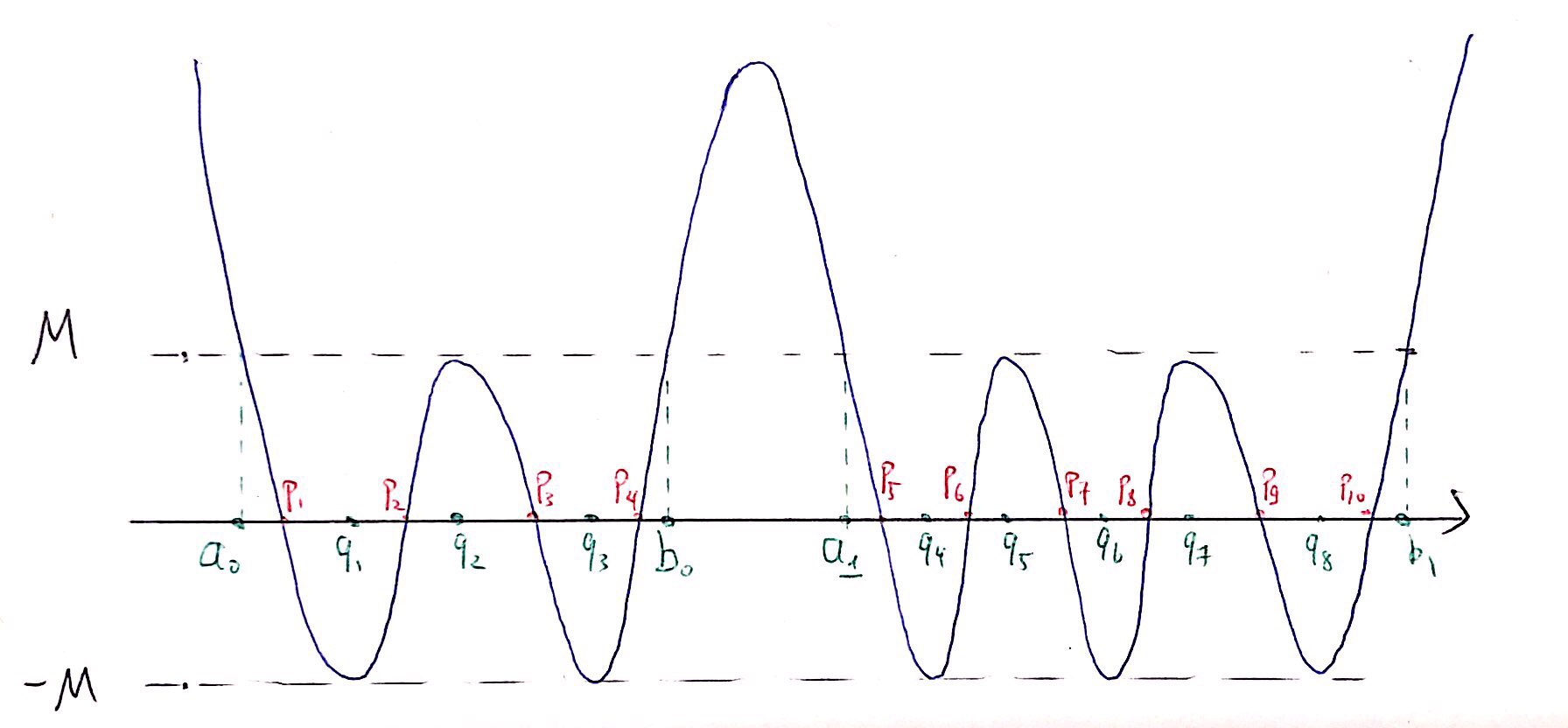}
    \caption{Shape of $P$ in the case $g=1$, $r_0=4$, $r_1 =6$}
    \label{fig:P_genre2}
\end{figure}

\begin{proof}
    The beginning of the proof is identical to the case where $g=0$. Let us consider $j\in\{0,\dots, g\}$ and the segment $E_j=[a_j,b_j]$. A representative of $f$ on $E_j$ is $f(x) = P(x) + i\sqrt{-D(x)}Q(x)$. We have $\abs{f}=M$ constant, therefore $f(x)= Me^{i\theta(x)}$, with $\theta : E_j\rightarrow \R$ continuous, and $\theta(a_j)=c_0\pi, \theta(b_j) = c_1\pi$, with $c_0,c_1 \in \Z$ since $f(a_j), f(b_j)$ are equal to $\pm M$. 
    We therefore have 
\[P(x) = M \cos (\theta(x)),\quad \sqrt{-D(x)}Q(x)=M \sin(\theta(x)).\]
    
Since $\dif f /f = r\eta = r R(x)\dif x/y$ and $R$ is never zero on $E_j$ according to \ref{rmq-Rnesannulepas}, $f'$ is never zero on $E_j$. Therefore $\theta$ is strictly monotonic on $E_j$, and the number of roots of $P$ in $E_j$ is equal to $\abs{c_1-c_0}$, the number of roots of $Q$ in $ E_j$ is equal to $\abs{c_1-c_0}-1$ (since $D$ is only zero at end points).

\medskip

Let us calculate $\abs{c_1-c_0}$. We notice that :
\[(c_1-c_0)i\pi = i (\theta(b_j)-\theta(a_j)) = i \int_{a_j}^{b_j} \dif \theta=   \int_{a_j} ^{b_j}\dif f/f = r\eta_j. \]
As a result, $\abs{c_1-c_0} = r \abs{\eta_j}/\pi = r_j$. Therefore the number of roots of $P$ in $E_j$ is equal to $r_j$. Since $P$ is of degree $r$, we first have that
\[\sum_{j=0}^g r_j \leqslant r.\]
But proposition \ref{prop-periodeimaginaire} shows that there exists some $(\epsilon_j)_{0\leq j \leq g}\in\{-1,1\}^{g+1}$ such that 
\[\displaystyle\sum_{j=0}^g \pm r_j = r.\] 
It follows that $\sum_{j=0}^g  r_j = r$. Hence, the only roots of $P$ are in the $E_j$. The sinusoidal shapes of $P$ and $Q$ give us the claims 2 and 3 of the proposition, since $\theta$ is strictly monotonic.
\end{proof}

\subsubsection{End of the proof of Robinson's theorem}
\paragraph{}
Proposition \ref{prop-cle}, which precisely describes the behavior of polynomials $P$ and $Q$, is the key to the proof of Robinson's theorem. Let us start with some immediate consequences of this result.

\medskip

Let us remind ourselves of the notations : $E=\displaystyle\bigcup_{j=0}^g E_j$, $E_j = [a_j,b_j]$, $D(X) = \displaystyle \prod_{j=0}^g (X-a_j)(X-b_j)$, $P^2-DQ^2 = M^2$, $\deg P = r$.

\begin{prop}
$\capa(E) = (\frac M2)^{1/r}$.
\end{prop}

\begin{proof}
    It is a consequence of corollary \ref{corCapaciteImageRecipPolynome}. Indeed, propositions \ref{prop-observationP} and \ref{prop-cle} show that \\ $P^{-1}([-M,M])=E$. But $\capa([-M,M]) = \frac M2$ and $P$ is a monic polynomial of degree $r$. Corollary \ref{corCapaciteImageRecipPolynome} gives us that $\capa(E) = P^{-1}([-M,M]) = (\frac M2)^{1/r}$.
\end{proof}

\begin{prop}
$P$ is the Chebyshev polynomial of degree $r$ with respect to $E$.
\end{prop}

\begin{proof}
    Proposition \ref{prop-cle} shows that $P$, a monic polynomial of degree $r$, reaches its extreme values, $\pm M$, $r+1$ times on $E$, hence the result using the equioscillation theorem \ref{thm-equioscillation}.
\end{proof}

Let us state a last very simple lemma which plays a major part in the proof : a small perturbation does not change the number of roots of $P$ in $E$.

\begin{lmm}\label{lmm-petiteperturbation}
Let us consider $q\in\R[X]$ such that $\abs{q(x)}<M$ for all $x\in E$. Then $P-q$ has at least $r$ roots in $E$.
\end{lmm}

\begin{proof}
    Let us denote by $q_1< \dots < q_{r-1}$ the roots of $Q$ with the notations $q_0=a$ and $q_r =b$. $P(q_i)=\pm M$ and since $\abs{q(q_i)}<M$, $P(q_i)-q(q_i)$ has the same sign as $P(q_i)$. Therefore $P-q$ changes sign on each interval $[q_i,q_i+1]$, for $i$ in range $0$ to $r-1$, hence the result. 
\end{proof}

\medskip

We now have all the elements to finish the proof of Robinson's theorem. Let us transform the real coefficients into rational ones.

\medskip
\begin{prop}\label{prop-reelaurationnel}
We can replace $M, P, Q,D, E$ by $\tilde{M}, \tilde{P}, \tilde{Q},\tilde{D}, \tilde{E}$, so that proposition \ref{prop-cle} is still verified, but with $\tilde{M}\in \Q$, $\tilde{P}\in \Q[X]$ (still of degree $r$), $\tilde{D}\in \Q[X]$, $\tilde{Q}=1$, $\tilde{E}\subset E$ and $\capa(\tilde{E})$ arbitrarily close to $\capa(E)$.
\end{prop}

\begin{proof}
     Let us consider $\tilde{M}\in [0,M]\cap \Q$. According to lemma \ref{lmm-petiteperturbation} applied to $q=\pm M$, $P^2-\tilde{M}^2$ has $2r$ roots in $E$ : $u_1 < \dots < u_{2r}$ with $[u_i,u_{i+1}]\subset [q_i,q_{i+1}]$, where $q_i$ are the roots of $Q$. By continuity of the roots, we can choose $\tilde{P}$ with rational coefficients sufficiently close to $P$, so that the roots of $\tilde{P}^2-\tilde{M}^2$ are still ordered the same way $\tilde{u}_1 < \dots < \tilde{u}_{2r}$. We then consider $\tilde{D}=\tilde{P}^2-\tilde{M}^2$, $\tilde{E} = \bigcup [\tilde{u}_i,\tilde{u}_{i+1}]\subset E$. Since $\capa(\tilde{E}) = (\frac {\tilde{M}}{2})^{\frac 1r}$, we can choose $\tilde{M}$, a rational number sufficiently close to $M$, so that $\capa(\tilde{E})$ is arbitrarily close to $\capa(E)$.
\end{proof}

From now on, assume that $P,Q,D,M$ have rational coefficients. Let us now prove Robinson's theorem, in the Pell-Abel case.

\begin{thm}[Robinson]
If $\capa(E)>1$, then there exists an infinite number of algebraic integers totally in $E$. More precisely, there exists a sequence of monic polynomials with integer coefficients whose degrees tend to infinity and whose roots are all in $E$.
\end{thm}

\begin{proof}
Since $\capa(E) = (\frac M2)^{1/r}$, $\capa(E)>1$ is equivalent to $M>2$. Let us consider $\lambda = \frac M2$. We have $\lambda >1$.

\medskip

We already have a polynomial with rational coefficients whose roots are all in $E$ : $P$. To find more of them, the idea is to raise the polynomial to higher powers. But instead of using $P$, like in the proof of Fekete-Szegö's theorem \ref{thmFeketeSzego}, we will implement this method using  $f=P+yQ$.

\medskip
We saw that $f(x)=P(x)+iy_1(x)Q(x) = 2\lambda e^{i\theta(x)}$. Therefore $f(x)^n = (2\lambda)^n e^{in\theta(x)} = 2^{n-1}(P_n(x) + iy_1(x) Q_n(x))$, with $P_n, Q_n \in \Q[X]$ monic polynomials. We can also rewrite $P_n$ as follows 
\[P_n(x) = 2 \lambda ^n \cos (n \theta(x)) \]

Let $T_n$ be the Chebyshev polynomial of degree $n$ with respect to $[-2,2]$, we have $P_n = \lambda ^n T_n(P/\lambda)$. With the explicit expression of $T_n$ : 
\[T_n = \sum_{k=0}^{n/2} (-1)^k \frac nk  \binom{n-k-1}{k-1}X^{n-2k} \]
We can completely expand the expression of the polynomial $P_n$ to write it the following way:
\[P_n(X) = X^{nr}+ \sum_{k=1}^{nr}\alpha_k X^{nr-k}\]
with $\alpha_k \in \Q$. The goal is now to provoke a perturbation of $P_n$ using a polynomial $q_n$ of degree $< nr$, such that $P_n -q_n \in \Z[X]$ and $q_n(x)< 2\lambda^n$ on $E$. Hence, according to lemma \ref{lmm-petiteperturbation}, $P_n-q_n$ will have all its $nr$ roots in $E$ and will still be a monic polynomial with integer coefficients. We will see that it is possible for an infinite number of degrees tending to infinity, which will conclude the proof. 

\medskip

Let us consider $m\in\N^*$ such that $P=X^r +\frac{1}{m}\Gamma $, where $\Gamma\in\Z[X]$. Let $l\in \N^*$ be a random integer (we will fix it later). Let us consider $n_l = (l!)^2m^{l}$. If $0\leqslant k \leqslant l$, we have that $k!\,m^k$ divides $n_l$ ; if $0\leqslant
2i+j \leqslant l$, we have that $m^j\,i!\,j!$ divides $n_l$. Therefore, we verify that $\alpha_1,\dots,\alpha_{lr}$ are integers. What remains in the sum is a polynomial $R_n$ of degree not greater than $nr-lr-1$ whose coefficients belong to $\Q$. Let us choose $(X^jP_k(X))_{0\leqslant j < r,\  0\leqslant k < n-l}$ as a basis of $\Q_{nr-lr-1}[X]$, and $c_{j,k}\in\Q\cap[0,1[$ such that $P_{n_l}-q_{n_l}\in\Z[X]$ (c.f. the proof of lemma \ref{lmmlmmFeketeSzego}), where
\[q_{n_l}=\sum_{0\leqslant j < r}\sum_ { 0\leqslant k < n-l} c_{j,k}X^jP_k\]

Let us remind that $P_k$ oscillates between $\pm 2\lambda^k$. Let us consider $C=\displaystyle\max_{x\in E}(\sum_{k=0}^{r-1} \abs{x}^k)$. Since $\lambda >1$, there exists $l_0$ such that for all $l>l_0$, $\frac{C}{\lambda^l (\lambda -1)}<1$. Therefore for $l\geqslant l_0$ : 
\[\abs{q_{n_l}(x)}\leqslant \sum_{j,k}\abs{x}^j 2\lambda^k
\leqslant
2C \sum_{k=0}^{n-l-1}\lambda^k 
\leqslant
2C\frac{\lambda^{n-l}}{\lambda -1}
< 2\lambda ^n \]

Therefore, we built a sequence of monic polynomials with integer coefficients whose degrees tend to infinity and whose roots all belong to $E$ (the sequence being $(P_{n_l}-q_{n_l})_{l\geqslant l_0}$).\\
QED
\end{proof}

\medskip

By refining this proof and using the fact that a polynomial function transforms an equilibrium measure into an equilibrium measure, we have the following stronger theorem : 

\begin{thm}\label{thm-densitePA}
For the sequence $P_n$ previously built, let us denote by $\mu_{P_n}$ the  (normalized) counting measure with respect to the roots of $P_n$. We then have $\mu_{P_n}\cvf \mu_K$, where $\mu_K$ is the equilibrium measure of $E$.
\end{thm}

\medskip 

This result is deep. Not only have we built an infinite number of algebraic integers totally in compact set $E$, but we also know that they are equidistributed with respect to the equilibrium measure of $E$. It partially answers the question asked in the introduction which dealt with the distribution of algebraic integers totally in a compact set. The proof of this theorem can be found in Serre's article \cite{serre2018bourbaki}.

\subsubsection{From the Pell-Abel case to the general case}
\paragraph{}
Let us give arguments that allows us to reduce the general case to Pell-Abel case. Let us consider $g\geqslant 0$ and \[U = \{(a_0, b_0, \dots, a_g, b_g)\in \R ^{2g+2} \ | \ a_0 < b_0 < \dots < a_g < b_g \}\]
$U_{PA}$ is the subset of $U$ such that for polynomial $D = \displaystyle \prod_{j=0}^g(X-a_j)(X-b_j)$, Pell-Abel equation has a solution. We have the following result :

\begin{thm}
$U_{PA}$ is dense in $U$. 
\end{thm}

\medskip We shall give here the main steps of the proof. With an element $u\in U$ can be associated a point on the jacobian variety of curve $y^2=D(x)$ : $\nu(u)=\infty_- - \infty_+ \in \R^g/\Z^g$. Theorem \ref{thm-PellianEquivautTorsion} states that $u\in U_{PA}$ if and only if $\nu(u)=\infty_- - \infty_+$ is a point with finite multiplicity (a torsion point) on the jacobian variety. However we can lift $\nu : U \rightarrow \R^g/\Z^g$ in order to obtain a continuous function $\theta : U \rightarrow \Q^g$, and $\nu(u)$ is a point of finite multiplicity in $\R^g/\Z^g$ if and only if $\theta(u)\in\Q^g$. We then conclude thanks to the density of $\Q^g$ in $\R^g$. Additional details of this proof are provided in Serre's article \cite{serre2018bourbaki}. 

\medskip

Thanks to theorem \ref{thm-densitePA}, it is easy to reduce the general case to Pell-Abel case. In fact, if $E_u$ is a union of segments given by $u\in U$ such that $\capa(E_u)>1$, we can find $u'\in U_{PA}$ arbitrarily close to $u$ so that we still have $\capa(E_{u'})>1$, by continuity of $\capa(E_u)$ with respect to $u$ (corollary \ref{cor-continutecapacitesegment}).

\clearpage

\section*{Conclusion}
\paragraph{}
In this paper, we have responded to the initial problem in the majority of cases: 

\begin{quote}
" Which $\R$ segments have an infinity of algebraic integers totally in them ?"
\end{quote}

\paragraph{}

First, based on "elementary" remarks and numerical results, we had the intuition that segments longer than 4 had an infinity of algebraic integers totally included in them, while segments shorter than 4 had only a finite number of them. However, in order to be able to demonstrate these conjectures, it was necessary to consider our problem as a particular case of a more general problem and not only to consider the $\R$ segments, but all the compact subsets of the complex plane.

\medskip

We had to clarify the notion of size to have an equivalent of the length of a segment for compact subsets of $\C$. Thus, we introduced the notion of capacity resulting from the theory of potential. Using this concept, we have highlighted that the limit value of capacity for this problem is 1. Potential theory gives several results in this direction but in the case of real segments, it is necessary to combine it with properties of algebraic curves to obtain the desired results.
\paragraph{}

The "elementary" results of the section \ref{section1} allowed us to develop the following intuitions:

\begin{enumerate}

    \item Capacity is the measurement of size adapted to our problem for compact subsets of $\C$. 
    
    \item Generally compact sets with a capacity of less than 1 (including segments with a length of less than 4) have a finite number of algebraic integers totally in them.
    
    \item Generally compact sets with a capacity greater than 1 (including segments with a length greater than 4) have an infinity of algebraic integers totally in them. 
\end{enumerate}

The rest of our paper consists in proving results that are in the direction of these intuitions. The main ones are set out below.

\medskip

The first theorem resulting from the theory of potential gives a complete answer for the second point:

\begin{thm*}[Fekete]
    Any compact with capacity strictly smaller than $1$ has a finite number of algebraic integers totally in it. 
\end{thm*}
\begin{cor*}
    Any segment with length strictly smaller than $4$ has a finite number of algebraic integers totally in it.
\end{cor*}

The third point is more complex, but the theory of potential manages to give a partial result:

\begin{thm*}[Fekete-Szeg\"o]
    Let us consider a compact subset of $\C$ with capacity greater than or equal to $1$, symmetric with respect to complex conjugation. Then any neighborhood (whose complement is connected) of this compact subset has an infinite number of algebraic integers totally in it.
\end{thm*}

However, unlike the previous theorem, this theorem is not applicable to $\R$ segments.  By mixing notions of algebraic curves with the notion of capacity, we obtain a much stronger theorem that completes the discussion for the $\R$ segments.

\begin{thm*}[Robinson]
Any finite union of intervals of $\R$ with capacity strictly greater than $1$ has an infinite number of algebraic integers totally in it.
\end{thm*}

\begin{cor*}
    Any segment with length strictly greater than $4$ has an infinite number of algebraic integers totally in it.
\end{cor*}

\paragraph{}

We give here several avenues for reflection for the interested reader. Many problems remain open. 

\medskip

In the case of $\R$ segments, the situation where the segment length is equal to 4 is unknown, except in the case where the boundaries of the segment are integers. We know that there is then an infinite number of algebraic integers totally in it. What happens if the segment has length 4 and the boundaries are not integers ? 

\medskip

Then, in this paper we were mainly interested in the finiteness of algebraic integers totally in a segment. A natural extension would be to seek to know more about these numbers, especially to know their distribution. Some of the results are contained in our paper. For example, we have found all the algebraic integers totally in the $[-2.2]$ segment. In addition, for a segment $E$ with length strictly greater  than 4, by refining the demonstration of Robinson's theorem, it is possible to demonstrate the existence of a sequence of monic polynomials with integer coefficients having all their roots in $E$ and whose counting measure with respect to their roots converges weakly-* to  the equilibrium measure of $E$ (\ref{thm-densitePA}). Let us consider a sequence of distinct algebraic integers totally in $E$. Do the counting measure with respect to their conjugates converge weakly-* to the equilibrium measure of $E$?

\clearpage

\begin{appendix}

\section{Measure theory}\label{appendixmeasure}

\subsection{Measures}

The notion of measure that we use is Radon measure. 

\begin{defi}[Radon measure]
Consider $D\subset\C$. Let $C_c^0(D,\R)$ denote the space of functions on $D$ with compact support.

We call \emph{Radon measure} any continuous linear functional on $C_c^0(D,\R)$.

If $\mu$ is a Radon measure, we write $\displaystyle\int_D f \dif \mu = \mu(f)$.
\end{defi}

\begin{ex}[Dirac measure]
Let us consider $z_0\in D$. The Dirac measure on $z$, denoted by $\delta_{z_0}$, is the linear continuous functional which maps $f\in C_c^0(D,\R)$ to $f(z_0)$, \emph{i.e.} $\displaystyle \int f\dif\delta_{z_0} = f(z_0)$.
\end{ex}

\begin{ex}[(Averaged) Counting measure]
Let $F = \{z_1,\dots,z_n\}$ be a finite subset of $D$, the (averaged) counting measure with respect to $F$ is given by $\nu_F = \frac 1n \displaystyle\sum_{i=1}^n \delta_{z_i}$, \emph{i.e.}  $\displaystyle\int f\dif \nu_F =\frac 1n \displaystyle\sum_{i=1}^n f(z_i)$.
\end{ex}

\begin{defi}[Measure of a set]\label{defmeasureEns} Let $\mu$ be a measure on $D\subset\C$.
    \begin{enumerate}
        \item If $D$ is compact, the measure of $D$ is defined by $\mu(D):=\displaystyle\int_D 1 \dif \mu$. 
        \item If  $K\subset D$ is a compact subset, then $\mu(K)=\displaystyle \int_D \ind{K}\dif \mu = \displaystyle\int_K 1 \dif \mu.$
     \item If $D = \displaystyle\cup_{n\in\N} K_n$ where $(K_n)_{n\in\N}$ is an increasing sequence of compact subsets, then the measure of $D$ is  $\mu(D):=\lim_{n\to\infty}\mu(K_n)$.
     \item If $\mu(D)=1$, we call $\mu$ a probability measure on $D$.
     \end{enumerate}
\end{defi}

\begin{rmq}
$\int_D\ind{K}\dif\mu=\int_K1\dif\mu$ is defined by $\inf\{\mu(f):f\in C^0_c(D,\R),0\leq f\leq1, f|_K=1\}$.
\end{rmq}

\begin{defi}[Support of a measure]
Let $\mu$ be a measure on $D\subset\C$ and let $U\subset D$. We say that $\mu$ has its support in $U$ if and only if
$$\forall f\in C_c^0(D,\R), f|_U = 0 \Rightarrow \int_D f \dif \mu = 0$$

The \emph{support} of $\mu$ is  the intersection of all the closures of such $U$. Equivalently, it is also the smallest closed subset $U$ of $D$ verifying the previous condition. The support of $\mu$ is denoted by $\supp(\mu)$.
\end{defi}

\begin{defi}[Restriction/Extension of a measure]
Let $\mu$ be a measure on $D$ whose support is included  in a compact subset $K\subset D$. For all $f\in C^0(K,\R)$, there exists $\tilde{f}\in C^0_c(D,\R)$ such that $\Tilde{f}|_K=f$. We define the integral of $f$ with respect to $\mu$ by
\[\mu(f):=\int_Kf\dif\mu:=\int_D \Tilde{f}\dif\mu\]
The definition does not dependent on the choice of $\tilde{f}$; $\mu$ becomes in this way a measure restricted to $K$ (which will be denoted by $\mu$ as well for convenience). For all $f\in C^0(D,\R)$, the following integral is well-defined.
\[\mu(f):=\int_Df\dif\mu:=\mu(f|_K)\]
This extends the domain of definition of $\mu$ to $C^0(D,\R)$.
\end{defi}

\begin{rmq}
Let $\mu$ be a measure on $D$ with support in a compact subset $K$, then $\mu(D)=\mu(K)=\int_D1\dif\mu$.
\end{rmq}

\begin{rmq}Following these definitions, we can define a measure $\Tilde{\mu}$ by
\[\Tilde{\mu}(A)=\inf_{U\supset A} \sup_{K\subset U}\mu(K)\]
where $U$ is open and $K$ is closed. In this way we extend the measure to any mesurable subset of $D$. 
\end{rmq}

\subsection{Weak-$*$ convergence}
Let us define now the import notion of weak-$*$ convergence of measures. 

\begin{defi}[Weak-$*$ convergence]
    Let $D\subset\C$ and $(\mu_n)_{n\in\N}$ be a sequence of measures on $D$. Let $\mu$ be a measure on $D$. We say that $(\mu_n)_{n\in\N}$ converges weakly to $\mu$, denoted by  $\mu_{n} \overset{*}{\longrightarrow} \mu  $, if and only if 
    $$\forall f\in C_c^0(D,\R),\ \lim_{n\to\infty}\int_D f\dif \mu_n = \int_D f\dif\mu.$$
\end{defi}

\begin{ex}[Riemann integral]
We can consider Riemann integral as the limit measure of a sequence of counting measures : indeed, if $f:[0,1]\to\R$, we have $$\lim_{n\to\infty}\frac 1n \displaystyle\sum_{k=0}^n f(\frac kn) = \int_0^1 f\ \dif x.$$ Let $\nu_n$ denote the counting measure with respect to points $\{\frac kn\ | \ k=0,\dots,n\}$. The continuous linear functional  $\displaystyle\int_0^1$ is the limit measure of the sequence $(\nu_n)$ for weak-$*$ convergence.
\end{ex}

\clearpage

\begin{ex}
If $\nu_n$ is the counting measure with respect to $n$-th roots of unity $\{e^{\frac{2\pi i k}{n}}\ |\ k=0,\dots,n-1\}$, then 
$$\lim_{n\to\infty}\int f\dif \nu_n = \frac{1}{2\pi}\int_0^{2\pi}f(e^{i\theta})\dif \theta.$$
\end{ex}

\begin{prop}\label{propSuppDecroi}
Let $K$ be a compact subset of $\C$ such that $K=\displaystyle\bigcap_{n\in\N} B_n$ , where $(B_n)_{n\in\N}$ is a decreasing sequence of compact sets.
If there exists a sequence of probability measures $(\mu_n)_n$ such that $\mu_n \cvf \mu$ and $\supp(\mu_n)\subset B_n$ for each $n$, then $\supp(\mu)\subset K$.
\end{prop}

\begin{proof}
    Consider $f\in C^0_c(D,\R)$ such that $f|_K=0$, let us show that for all $\epsilon>0$, we have $\abs{\int_Kf\dif\mu}<\epsilon$.
    The open set ($f$ is continuous) $U:=\{|f|<\epsilon\}$ contains $K=\bigcap_nB_n$. Since all $K_i$ are compact, there exists an integer $N$ such that $\forall\,n\geq N,B_n\subset U$. Consider $\phi_n\in C^0_c(D,\R)$ such that
	\[\phi_n|_{B_n}=1,\quad\phi_n|_{U^c}=0,\quad 0\leq\phi_n\leq1\]
	Then for all $n\geq N$, $\abs{\phi_nf}<\epsilon, (1-\phi_n)f|_{B_n}=0$, so that $\abs{\int_Df\dif\mu_n}=\abs{\int_D\phi_nf\dif\mu_n}<\epsilon$; thus $\mu_n\cvf\mu$ gives
	$\abs{\int_Kf\dif\mu}\leq\epsilon$
\end{proof}

\medskip
We are now interested in the space of measures on a given topological space. Let $ X $ be a compact metric space. Let $ C(X) $ denote the set of continuous functions from $X$ to $\mathbb{R}$, and $ \mathcal{P}(X) $ denote the set of probability measures on $ X $. 

\begin{thm}[Banach-Alaoglu-Bourbaki]\label{thmBanachAlaoglu}
$ \mathcal{P}(X) $ is sequentially compact for the topology associated with the weak-$*$ convergence (\emph{i.e.}, weak-$*$ topology).

In other words, for any sequence of probability measures $(\mu_n)_n$, there are a sub-sequence $(\mu_{\varphi(n)})_n$ and a probability measure $\mu$ such that $\mu_{\varphi(n)} \overset{*}{\longrightarrow} \mu  $, \emph{i.e.} 
$$
\forall \phi \in C(X), \int_X \phi ~ \dif \mu_{\varphi(n)} {\longrightarrow}   \int_X \phi ~ \dif \mu
$$
\end{thm}

\begin{proof}

$X$ is a compact metric space, so that $ C(X) $ is separable, \emph{i.e.} there is a dense sequence $(\phi_n)_n$ in $C(X)$.
\\

The idea is to do a diagonal extraction. For $\phi_1$, we have $(\int_X \phi_1 ~ d\mu_{n})_n$ which is bounded. So we can extract a sub-sequence $(\mu_{\varphi_1(n)})_n$ such that $(\int_X \phi_1 ~ d\mu_{\varphi_1(n)})_n$ converges. We can extract in the same way $(\mu_{\varphi_2 (n)})_n$ from the sequence $(\mu_{\varphi_1(n)})_n$. Thus, we construct $(\mu_{\varphi_k (n)})_n $ for all $k \in \mathbb{N}$, which is extracted from the previous sequences.
\\

We consider then the sequence  $(\mu_{\varphi_n(n)})_n$. For $k \in \mathbb{N}, (\int_X \phi_k ~ d\mu_{\varphi_n(n)})_n$ converges : indeed, for some sufficiently large $n$ , ${\varphi_n}$ is extracted from ${\varphi_k}$.
Using the density of $(\phi_n)_n$, $(\int_X \phi ~ d\mu_{\varphi_n(n)})_n$ converges for $\phi \in C(X)$.
\\

Let us define the functional 
$$  \Lambda : C(X) \longrightarrow \mathbb{R}  ~~~~~~~~~~~~~~~~~~~~~~~~~~~~~~~~$$
$$ \phi \longmapsto \lim\limits_{n \rightarrow +\infty} \int_X \phi ~ d\mu_{\varphi_n(n)}
$$

It is clearly a positive linear functional. By Riesz–Markov–Kakutani representation theorem, we have : $  \Lambda(\phi) =  \int_X \phi ~ d\mu$ where $\mu $ is a measure $C(X)$.
\\

To prove the compactness of $\mathcal{P}(X)$, it is enough to show that $\mu \in \mathcal{P}(X)$, since we have by construction $
\forall \phi \in C(X), \int_X \phi ~ d\mu_{\varphi_n(n)} {\longrightarrow}   \int_X \phi ~ d\mu
$. This fact is true because $ \int_X  d\mu = \lim\limits_{k \rightarrow +\infty} \int_X  d\mu_{\varphi_n(n)}  = 1$.
\end{proof}

\subsection{Lower semi-continuous functions}

We wish to extend the concept of measure and weak-$*$ convergence to functions that are not necessarily continuous.

\begin{defi}[Lower semi-continuous function] A function $f:D\subset\R^m\rightarrow\R\cup\{+\infty\}$ is said to be \emph{lower semi-continuous (l.s.c.)} if one of the following equivalent properties is satisfied:
\begin{enumerate}[label = (\roman*)]
    \item $\forall z_0 \in D, f(z_0)\leqslant \displaystyle\liminf_{z\rightarrow z_0}{f(z)}$
    \item On all compact subsets $K\subset D$, $f$ is a pointwise limit of an increasing sequence of continuous functions
    \item $\forall \alpha\in\R$, $\{f>\alpha\}:=f^{-1}(]-\infty,\alpha[)$ is a open set.
\end{enumerate}
\end{defi}

\begin{defi}
Let $f$ be a l.s.c.  function and $\mu$ be a positive measure with support in a compact subset $K\subset D$.  We define the integral of $f$ with respect to $\mu$ : $$\int_K f  \dif \mu = \lim_{n\rightarrow \infty }\int_K f_n \dif \mu$$ where $(f_n)$ is an increasing sequence of continuous functions on $K$ which converges pointwise to $f|_{K}$
\end{defi}

\begin{prop}\label{propSciMin}
Any l.s.c function defined on a compact set has a minimum point.
\end{prop}

\begin{proof}
Consider the sets $\{f>n\}_{n\in\Z}$  and $\{f\leq\frac{1}{n}+\inf f\}_{n\in\N^*}$.
\end{proof}

\begin{prop}\label{propIntegralSciPositive}
Let $f$ be a l.s.c. function verifying $f|_K\geq\alpha$ and let $\mu\in\mathcal{P}(D)$ with $\supp(\mu)=K\subset D$. Then $\int_Kf\dif\mu\geq\alpha$. In addition, $\int_Kf\dif\mu=\alpha$ if and only if $f|_K\equiv\alpha$.
\end{prop}

\begin{proof}
Since $\mu$ is a probability measure,  we can assume without loss of generality that $\alpha=0$.
Let $(f_n)$ be an increasing sequence of continuous functions on $K$ which converges pointwise to $f|_K\geq0$. Then $\forall\epsilon>0$, $K=\cup_{n\geq1}\{f_n>-\epsilon\}$. By compactness, $\exists N\,\forall n\geq N,f_n>-\epsilon$ on $K$, which implies $\int_Kf\dif\mu\geq-\epsilon$. So that $\int_Kf\dif\mu\geq0$.
\medskip\par For the second assertion, assume there exists $\zeta_0\in K$ such that $f(\zeta_0)>0$. Then we have $\int_Kf\dif\mu>0$.
Indeed, there exists two open neighborhoods  $B_1\subset B_2$ of $\zeta_0$ in $D$ such that $f|_{B_2}\geq\frac{1}{2}f(\zeta_0)$. Let  $\chi\in C^0_c(B_2,\R)$ such that $\chi|_{B_1}=1$ and $0\leq\chi\leq 1$. Then $f-\frac{1}{2}f(\zeta_0)\chi$ is a l.s.c  function on $D$ and positive on $K$.
By the above, $\int_K(f-\frac{1}{2}f(\zeta_0)\chi)\dif\mu\geq0$, so that $\int_Kf\dif\mu\geq\frac{1}{2}f(\zeta_0)\int_K\chi\dif\mu$.
The last quantity is strictly positive, otherwise we would have $\supp(\mu)\subset K\backslash B_1\subsetneq K$ (let $\phi\in C^0_c(D,\R)$ such that $\phi|_{K\backslash B_1}=0$, then $(1-\chi)\phi^\pm|_K=0$, $\int_K\phi^\pm\dif\mu=\int_K\chi\phi^\pm\dif\mu\leq\|\phi\|_\infty\int_K\chi\dif\mu=0$, then $\int_K\phi\dif\mu=0$), which is a contradiction.
\end{proof}

\begin{prop} Let $f$ be a l.s.c. function and $(\mu_n)$ be a sequence of positive measures with support included in a compact subset $K$ such that $\mu_n\overset{*}{\rightarrow} \mu$. Then:
$$\int_K f\dif\mu \leqslant \liminf_{n\rightarrow \infty}\int_K f\dif\mu_n$$
\end{prop}

\begin{proof}\label{prop-smi-cvf}
Let $(f_n)$ be an increasing sequence of continuous functions which converges pointwise to $f$. Let $\epsilon >0$. There exists $n\in\N$ such that $\int_K f_n \dif\mu \geq \int_K f \dif\mu-\epsilon$.
There exists $m_0\in \N$ such that for $m\geq m_0$, $\int_K f_n \dif\mu_m\geq \int_K f_n \dif\mu-\epsilon$.
Then : $$\forall m \geq m_0, \int_K f d\mu_m \geq \int_K f_n\dif\mu_m\geq\int_Kf \dif\mu -2\epsilon$$
Thus $\displaystyle\liminf_{m\rightarrow \infty} \int_K f \dif \mu_m \geq \int_K f \dif\mu -2\epsilon$
\end{proof}

\begin{ex}
Let $z\in\C$, the function $t\mapsto\log(1/\abs{z-t})$ is lower semi-continuous, because it is continuous on $t\neq z$ and equal to $+\infty$ on $t=z$.
\end{ex}

\clearpage
\section{Semi-harmonic functions}\label{appendixeSemiharmonic}

\begin{defi}[(super-,sub-)harmonic functions]
    Let $D\subset\C$ be an open subset. A function $f:D\to\R$ is said to be \emph{harmonic} (\resp. \emph{super-harmonic, sub-harmonic}) if it is continuous (\resp. lower/upper semi-continuous) and satisfies the \emph{mean value property}: for all $z\in D$ and a disc $\{\abs{\zeta-z}\leq r\}\subset D$, we have
    \[f(z)=\frac{1}{2\pi}\int_0^{2\pi}f(z+re^{i\theta})\dif\theta\quad(\resp.\;\geq,\leq)\]
\end{defi}

A function is called  \emph{semi-harmonic} if it is harmonic or super-harmonic or sub-harmonic.

\medskip

Intuitively, a \emph{super}-harmonic function takes values \emph{larger} than its local mean values. However, it cannot take too large values because of the lower semi-continuity. 
\[f(z)\leq\liminf_{\zeta\to z}f(\zeta).\]

\begin{rmq}
Let $f:D\to\R$ be a continuous (\resp. lower semi-continuous, upper semi-continuous) function. Then $f$ is harmonic (\resp. super-harmonic, sub-harmonic) if and only if it satisfies the mean (\resp. super-mean, sub-mean) value property : for all $z\in D$, there exists $\delta>0$ such that the disk $\{\abs{\zeta-z}\leq \delta\}\subset D$ and for all $0<r\leq\delta$, we have
    \[f(z)=\frac{1}{2\pi}\int_0^{2\pi}f(z+re^{i\theta})\dif\theta\quad(\resp.\;\geq,\leq)\]
\end{rmq}

\begin{rmq}
    Assume  $f\in C^2(D)$. Then $f$ is super-harmonic if and only if $-\Delta f\geq0$. To understand this claim, we can consider the function
    \[m(r;z)=\frac{1}{2\pi}\int_0^{2\pi}f(z+re^{i\theta})\dif\theta\]
    and use the following formula
    \[\partial_rm(r;z)=\frac{1}{2\pi \abs{r}}\int_{B(z,\abs{r})}\Delta f.\]
\end{rmq}
\begin{proof}
    Consider $f$ as a function from $\R^2$ to $\R$ . Let $\overrightarrow{n}(\theta) = (\cos\theta,\sin\theta)$:
    \[\partial_r f((Re(z),Im(z))+r\mathbf{n}(\theta)) = \overrightarrow{\nabla} f \cdot \overrightarrow{n}(\theta). \]
    By Green's theorem:
    \begin{align*}
        \partial_r m(r;z)  &= \frac{1}{2\pi} \oint_{\partial B(z,r)} \overrightarrow\nabla f \cdot \overrightarrow{n}(\theta) \frac {\dif l}{r}\\
                           &= \frac{1}{2\pi r} \iint_{B(0,r)} \overrightarrow\nabla \cdot \overrightarrow\nabla f \cdot ds.
    \end{align*}
\end{proof}

\begin{ex}
Le $F$ be a holomorphic function on an open subset $D\in\C$. Then for $p>0$, $\abs{F}^p$ and $\log\abs{F}$ are sub-harmonic. Indeed, it is enough to prove this claim on the open set $\{F\neq 0\}$. Since locally $F^p$ (\resp. $\log F$) has a holomorphic branch, we can verify the continuity and apply Cauchy's formula. Then we conclude by using the triangular inequality for integrals (by considering the real part \resp.).
\end{ex}

\begin{prop}\label{propMin2super-harmonic}
    Let $f,g$ be two super-harmonic function on $U$. Then $\min(f,g)$ is also super-harmonic on $U$.
\end{prop}

\begin{proof}
The minimum of two l.s.c. functions is still a l.s.c. function. Then we conclude by verifying the super-mean property.
\end{proof}

\begin{thm}[Minimum principle for super-harmonic functions]\label{thmMinFnctSurharmnq}
Let $D\subset\C$ be a connected and bounded open subset. Let $f$ be a non-constant super-harmonic function on $D$ such that
\[\liminf_{z\to\zeta}{f(z)}\geq m,\quad\forall\zeta\in\partial D\]
Then $f(z)>m$ for all $z\in D$. Thus, no non-constant super-harmonic function on a connected open set (not necessarily bounded) reaches its  minimum.
\end{thm}

\begin{proof}
The idea of the proof is to use the super-mean property and an open-closed set argument.
\medskip\par Firstly, we extend $f(z)$ to all points $\zeta\in\partial D$ by $\Tilde{f}(\zeta)=\liminf_{z\to\zeta}{f(z)}$ and $\Tilde{f}|_D=f$. So $\Tilde{f}$ is  a l.s.c. function on $\overline{D}$ and super-harmonic on $D$. Since $D\subset\C$ is bounded and $\overline{D}$ is compact, $\Tilde{f}$ reaches its  minimum $m'=\min_{\Bar{D}}\Tilde{f}$ on $\overline{D}$ (Prop. \ref{propSciMin}). 
\medskip\par Since $f$ is l.s.c., the set $\{f=m'\}=\{f\leq m'\}$ is closed in $D$. On the other hand, for all $\zeta_0\in D$ verifying $f(\zeta_0)=m'$, we have for some $\delta>0$ and all $r\in]0,\delta]$,
\[m'=f(\zeta_0)\geq\frac{1}{2\pi}\int_0^{2\pi}f(\zeta_0+re^{i\theta})\dif\theta\geq m'\]
By Prop. \ref{propIntegralSciPositive}, for all $r\in]0,\delta],\,\theta\in[0,2\pi]$, we have $f(\zeta_0+re^{i\theta})=f(\zeta_0)=m'$, \emph{i.e.} $f\equiv m'$ in a neighborhood of  $\zeta_0$, which shows that $\{f=m'\}$ is an open subset of $D$. Thus, $\{f\leq m'\}=\{f=m'\}$ is open-closed in $D$; it is not $D$ since $f$ is non-constant, so it is empty because $D$ is connected, \emph{i.e.} $f(z)>m',\,\forall z\in D$.
\medskip\par On the other hand, since $m'$ is effectively reached by $\Tilde{f}$, there exists $\zeta_0\in\overline{D}\backslash D=\partial D$ such that $\Tilde{f}(\zeta_0)=m'$. By definition of $\Tilde{f}$ and the hypothesis in the statement of the theorem (finally!), we have
\[\Tilde{f}(\zeta_0)=\liminf_{z\to\zeta_0}f(z)\geq m\].
So that $m'\geq m$. Then we conclude that $f(z)>m,\,\forall z\in D$.
\end{proof}

\begin{cor}[Maximum principle for potentials]\label{appendixthmPrincpMaxPotent}
    Let $\mu$ be a finite positive measure with compact support. If $U^\mu(z)\leq M$ for all $z\in\supp(\mu)$, then it holds for all $z\in\C$.
\end{cor}

\begin{proof}
    Let $K=\supp(\mu)$. Let us define $f(z):=-U^\mu(z)$, which is a harmonic function, \emph{a fortiori} super-harmonic on $\C\backslash K$. It is non-constant and converges to $+\infty$ when $z\to\infty$.
    For all $\zeta_0\in\partial K$, we have
    \begin{equation}\label{eqPrincMaxPotent}
        \liminf_{z\to\zeta_0}f(z)\geq-M.
    \end{equation}
    Then $f$ reaches its minimum if there exists $z\in\C$ such that $U^\mu(z)>M$.
    
    \medskip Now we should prove (\ref{eqPrincMaxPotent}). Let $r>0$. We have
    \[f(\zeta)-f(z)=\int\log\frac{\abs{\zeta-t}}{z-t}\dif\mu(t)=\left(\int_{K\cap B(\zeta_0,r)}+\int_{K\backslash B(\zeta_0,r)}\right)\log\frac{\abs{\zeta-t}}{\abs{z-t}}\dif\mu(t).\]
    Let us choose $\zeta=\zeta(z)\in\arg\min_{\zeta\in K}\abs{\zeta-z}$. Then for all $t\in K$, we have
    \[\frac{\abs{\zeta-t}}{\abs{z-t}}=\frac{\abs{\zeta-z}+\abs{z-t}}{\abs{z-t}}\leq\frac{\abs{t-z}+\abs{z-t}}{\abs{z-t}}=2.\]
    We derive that 
    \[\int_{K\cap B(\zeta_0,r)}\log\frac{\abs{\zeta-t}}{\abs{z-t}}\dif\mu(t)\leq 2\mu(B(\zeta_0,r)).\]
    On the other hand, when $z\to\zeta_0$, we have $\zeta\to\zeta_0$. So,
    \[\int_{K\backslash B(\zeta_0,r)}\log\frac{\abs{\zeta-t}}{\abs{z-t}}\dif\mu(t)\to \mu(B(\zeta_0,r)^c)\times\log(1)=0.\]
    The assumption on $U^\mu(z)$ implies that $f(\zeta)\geq-M>-\infty$. By passing to the limit $z\to\zeta_0$, we have
    \begin{equation}
    \label{ineqPrincMaxPotent}
        \liminf_{z\to\zeta_0}f(z)\geq-M-2\mu(B(\zeta_0,r)).
    \end{equation}
    Since $f\geq-M$ on $K$ by hypothesis, we have $f>-\infty$ everywhere on $\C$ by the definition of $f=-U^\mu(z)$. So
    \[\lim_{r\to0}\mu(B(\zeta_0,r))=\mu(\{\zeta_0\})=0.\]
    Combining with (\ref{ineqPrincMaxPotent}), we conclude the proof of (\ref{eqPrincMaxPotent}).
\end{proof}

We state the following theorem which generalizes the \hyperref[thmMinFnctSurharmnq]{minimum principle}:

\begin{thm}[Generalized minimum principle]\label{thmGenerlPrincpMin}
Let $D\subset\overline{\C}$ be a connected open set such that $\capa(\partial D)>0$. Let $f$ be a non-constant super-harmonic function, lower-bounded in $D$, verifying \emph{quasi-almost everywhere} in $D$ the inequality :
\[\liminf_{z\to\zeta}{f(z)}\geq m.\]
Then we have $f(z)>m$ for all $z\in D$.
\end{thm}

A proof of the theorem can be found in \cite{ransford1995potential}, Theorem 3.6.9.

\begin{prop}[Backward-pulling by holomorphic mapping]\label{propPullback}
Let $f$ be a non-constant holomorphic function on a connected open subset $D\subset\C$. Let  $D'=f(D)$ (a connected open set as well). Then for any (sub-, super-)harmonic function  $u:D'\to\R$, $u\circ f$ is also (sub-, super-)harmonic.
\end{prop}

A proof of the proposition can be found (\cite{ransford1995potential}, Corollary 2.4.3).

This proposition allows us to extend the definition of (sub-, super-)harmonic functions to the Riemann sphere $\overline{\C}=\C\cup\{\infty\}$:

\begin{defi}[(Super-, sub-)harmonic function on Riemann sphere]
A function $f(z)$ defined on a neighborhood of $\infty$ ($\infty$ included) is said to be \emph{(sub-,super-)harmonic on a neighborhood of $\infty$} if $f(1/z)$ is (sub-, super-)harmonic at point $0$. A function $f:\overline{\C}\supset D\to\R$ is said to be \emph{(sub-,super-)harmonic} if all points of $D$ have a neighborhood where $f$ is \emph{(sub-,super-)harmonic}.
    
\end{defi}

\section{Some technical proofs}

\subsection{Proof of proposition \ref{propConditionConeReg}}\label{propConditionConeRegProof}

\begin{prop}
	Let $K\subset\C$ be a compact subset with $\capa{(K)}>0$. If $\Omega_K$ verifies the cone condition, then $\Omega_K$ is regular.
\end{prop}

\begin{proof}
	We shall prove that for all $\zeta_0\in\partial\Omega_K$,
	\[\lim_{z\to\zeta_0}g_K(z,\infty)=0\]
	\par (i) First, let $\chi(z)$ be a positive super-harmonic function defined on a non-empty connected open set $\Omega_0=\Omega_K\cap B(\zeta_0,\delta)$ such that
	\[A=\inf\left\{\chi(z):z\in\Omega_K,\frac{\delta}{2}<\abs{z-\zeta_0}<\delta\right\}>0\]
	\[\lim_{z\to\zeta_0}\chi(z)=0.\]

	Since $g_K(\cdot,\infty)$ is bounded on $\Omega_0$ (Thm. \ref{thmGreen}), one can find a large enough $M>0$ such that $MA>g_{K}(\cdot,\infty)$ on $\Omega_0$; thus for all $\zeta\in\Omega_K\cap\partial B(\zeta_0,\delta)$, we have
	\[\liminf_{z\to\zeta}(M\chi(z)-g_{K}(z,\infty))\geq0;\]
	we also deduce that $M\chi-g_K(\cdot,\infty)$ is \emph{bounded from below} on $\Omega_0$.
	Also, because $\chi$ is positive, for all $\zeta\in\partial\Omega_K$ such that $\abs{\zeta-\zeta_0}\leq\delta$, we have
	\[\liminf_{z\to\zeta}\chi(z)\geq0,\]
	then we have \emph{q.-a.e.} in $\partial\Omega_K$,
	\[\liminf_{z\to\zeta}(M\chi(z)-g_{K}(z,\infty))\geq0.\]
	On the other hand, $\chi$ is super-harmonic on $\Omega_0$,
	so $M\chi-g_K(\cdot,\infty)$ is super-harmonic because $g_K(\cdot,\infty)$ is harmonic.
	Observing that
	\[\partial\Omega_0\subset\{\zeta\in\partial\Omega_K:\abs{\zeta-\zeta_0}\leq\delta\}\cup\left\{\zeta\in\Omega_K:\abs{\zeta-\zeta_0}=\delta\right\}\]
	and that (Prop. \ref{propCapa}(c)) $\capa(\partial \Omega_0)=\capa(\overline{\Omega_0})>0$ since $\Omega_0$ is open and non-empty (Cor. \ref{corPropCapa}(a)).
	By applying the \hyperref[thmGenerlPrincpMin]{generalized minimum principle} on $\Omega_0$, we have for all $z\in\Omega_0$,
	\[M\chi(z)-g_{K}(z,\infty)\geq0,\]
	therefore,
	\[0=\lim_{z\to\zeta_0}M\chi(z)\geq\limsup_{z\to\zeta_0}g_K(z,\infty)\geq0\]
	\[\lim_{z\to\zeta_0}g_K(z,\infty)=0.\]
	\par (ii) Let $\chi$ be a positive super-harmonic function on $\Omega_1=\Omega_K\cap B(\zeta_0,\delta)\neq\emptyset$ such that
	\[A=\inf\left\{\chi(z):z\in\Omega_K,\frac{\delta}{2}<\abs{z-\zeta_0}<\delta\right\}>0\]
	\[\lim_{z\to\zeta_0}\chi(z)=0\]
	For $z\in\Omega_K$, define
	\[\tilde{\chi}(z)=
		\begin{cases}
			\min\left(\frac{\chi(z)}{A},1\right) & z\in\Omega_1\\
			1 & z\in\Omega_K\backslash\Omega_1
		\end{cases}\]
	By definition, $0<\tilde{\chi}\leq1$. As the minimum of two super-harmonic functions, $\tilde{\chi}$ is also super-harmonic on $\Omega_1$; $\tilde{\chi}\equiv1$ is super-harmonic on $\Omega_K\cap\{\abs{z-\zeta_0}>\frac{\delta}{2}\}$; therefore, $\tilde{\chi}$ is super-harmonic on $\Omega_K$. Furthermore, we have
	\[\lim_{z\to\zeta_0}\tilde{\chi}(z)=0.\]
	Take $R>2\delta$ large enough such that $\partial\Omega_K\subset B(\zeta_0,R)$, then $\Omega_0=\Omega_K\cap B(\zeta_0,R)\neq\emptyset$ is \emph{connected} because $\Omega_K$ is \emph{connected}. Also, we have
	\[\inf\left\{\chi(z):z\in\Omega_K,\frac{R}{2}<\abs{z-\zeta_0}<R\right\}=1>0.\]
	We have reduced our proof to the case (i) with $(\chi,\delta):=(\tilde{\chi},R)$.
	\medskip\par (iii) For all $\zeta_0\in\partial\Omega_K$, we shall prove the existence of such a function $\chi$ as in (ii) (for some $\delta>0$). We can take the composition of a M\"obius transformation and an analytic branch of $\Re(1/\log w),\abs{w}<1$.

	More precisely, since the cone condition is verified, one can choose $\zeta_1\neq\zeta_0$ such that the segment $[\zeta_0,\zeta_1]\subset\C\backslash\Omega_K$. Then take
	\[\chi(z)=\Re\left(1\Big/\log\frac{z-\zeta_0}{z-\zeta_1}\right),\quad z\in\Omega_K,\abs{\frac{z-\zeta_0}{z-\zeta_1}}<1\]
	for some well-chosen $\delta$. Taking Prop. \ref{propPullback} into account, the verification of this construction will be direct.
\end{proof}

\begin{rmq}\label{rmqReg}
	Notice that it is only in step (iii) that we have used the cone condition, which aimed at constructing a pair $(\delta,\chi)$ for (ii); therefore, without the cone condition, if one can construct a pair $(\delta,\chi)$ for all $\zeta_0\in\partial\Omega_K$ verifying the conditions in (ii),
	then $\Omega_K$ is still regular.
\end{rmq}
\clearpage

\subsection{Proof of theorem \ref{thmAppliHolom}}

\begin{thm}
	Let $K_1,K_2\subset\C$ be two non-empty compact subsets. Let $f:\Omega_{K_1}\cup\{\infty\}\to\Omega_{K_2}\cup\{\infty\}$ be a non-constant holomorphic mapping such that $f(\infty)=\infty$.
	Then,
	\begin{enumerate}[label=(\alph*)]
		\item $\abs{A_f}\capa(K_1)^{n_f}\geq\capa(K_2)$.
		\item If we assume in addition :
			\begin{enumerate}[label=(\roman*)]
				\item $\capa(K_2)>0$;
				\item $f^{-1}(\infty)=\{\infty\}$;

				\item $\Omega_{K_2}$ is regular;
				\item $f$ can be continuously extended to the boundaries $\partial\Omega_{K_1}\to\partial\Omega_{K_2}.$
			\end{enumerate}
			
			Then, $\capa(K_1)$>0, $\Omega_{K_1}$ is regular and 
			\[g_{K_1}(\cdot,\infty)=\frac{1}{n_f}g_{K_2}(f(\cdot),\infty)\]
			\[\abs{A_f}\capa(K_1)^{n_f}=\capa(K_2).\]
	\end{enumerate} 
\end{thm}

\begin{proof}
	(a). We can assume that $\capa(K_2)>0$ without loss of generality. Then $g_{K_2}(\cdot,\infty)>0$ is harmonic on  $\Omega_{K_2}$.
	So that $g_{K_2}(f(\cdot),\infty)>0$ is super-harmonic (Prop. \ref{propPullback}) on $\Omega_{K_1}$ (and it is harmonic if $f^{-1}(\infty)=\{\infty\}$). Furthermore, we have $g_{K_2}(w,\infty)=\log\abs{w}+V_{K_2}+o(1)$ when $w\to\infty$. Then,
	\[g_{K_2}(f(z),\infty)=\log\abs{f(z)}+V_{K_2}+o(1)=n_f\log\abs{z}+\log\abs{A_f}+V_{K_2}+o(1),\quad z\to\infty\]
	Assume first that $\capa(K_1)>0$. Let us define
	\begin{equation*}\label{defh(z)}
	h(z):=g_{K_2}(f(z),\infty)-n_fg_{K_1}(z,\infty).
	\end{equation*}
	Then $h$ is super-harmonic on $\Omega_{K_1}$, and even on $\Omega_{K_1}\cup\{\infty\}$ if we remove this isolated singularity by defining
	\[h(\infty):=\log\abs{A_f}+V_{K_2}-n_fV_{K_1}.\]
	According to Thm. \ref{thmGreen}, $\lim_{z\to\zeta}g_{K_1}(z,\infty)=0$ for \emph{quasi-almost every} $\zeta\in\partial\Omega_{K_1}$.
	So, we have for \emph{quasi-almost every} $\zeta\in\partial\Omega_{K_1}$
	\[\liminf_{z\to\zeta}h(z)\geq0.\]
	Besides, we have seen that $h(z)$ in bounded on a neighborhood of $\infty$; beyond this neighborhood , $g_{K_1}(\cdot,\infty)$ is bounded. Thus $h(z)$ is lower-bounded.
	So, the \hyperref[thmGenerlPrincpMin]{generalized minimun principle} implies that either $h(z)\equiv0$ or $h(z)>0$ on $\Omega_{K_1}$. In particular, 
	\[h(\infty)=\log\abs{A_f}+V_{K_2}-n_fV_{K_1}\geq0.\]
	\par Henceforth, we no longer assume $\capa(K_1)>0$. For $\epsilon>0$, let  $K_1^\epsilon=\{z\in\C:\dist(z,K_1)\leq\epsilon\}$, which has non-empty interior. Thus, $\capa(K_1^\epsilon)>0$ according to Cor. \ref{corPropCapa}(a). Then we have $\abs{A_f}\capa(K_1^\epsilon)^{n_f}\geq\capa(K_2)$ by previous arguments.
	Applying  Prop. \ref{propCapa}(d) to the sequence $(K_1^{1/n})$, we have
	\[\abs{A_f}\capa(K_1)^{n_f}\geq\capa(K_2)\]
	\par (b). We have $\capa(K_1)>0$ by (a). By the continuous extension theorem, we have that $\lim_{z\to\zeta_0}f(z)$ exists and belongs to  $\Omega_{K_2}$ for all $\zeta_0\in\partial\Omega_{K_1}$. So the regularity of $\Omega_{K_2}$ implies that
	\[\lim_{z\to\zeta_0}g_{K_2}(f(z),\infty)=0,\quad\forall\zeta_0\in\partial\Omega_{K_1}\]

We keep the notations of (a). We remind that $h(z)$ is bounded on a neighborhood $U$ of $\infty$;
The conditions (ii) and (iv) imply that $f(U)$ is included in the complementary of a neighborhood of $\infty$ in $\Omega_{K_2}$. Then $g_{K_2}(f(\cdot),\infty)$ is bounded in $\Omega_{K_1}\backslash U$. We have then shown that $h(z)$ is a harmonic function which is bounded on $\Omega_{K_1}\cup\{\infty\}$ and verifies 
\[\lim_{z\to\zeta}h(z)=0.\] for 
\emph{quasi-almost every} $\zeta\in\partial\Omega_{K_1}$.
By the \hyperref[thmGenerlPrincpMin]{generalized minimum principle} applied to $h$ and to $-h$, we have $h(z)\equiv0$, in particular $h(\infty)=0$. So, we have
	\[V_{K_1}=\frac{1}{n_f}(\log\abs{A_f}+V_{K_2})\]
	\[\capa(K_1)=\left(\frac{\capa(K_2)}{\abs{A_f}}\right)^{1/{n_f}}.\]
\end{proof}

\end{appendix}

\newpage
\nocite{*}
\bibliographystyle{unsrt-fr}
\bibliography{biblio}

\end{document}